\documentclass[11pt,reqno]{amsart}
\usepackage{amsmath, amsfonts, amsthm, amssymb,amscd, graphicx, amscd}
\usepackage{float,epsf}
\usepackage[english]{babel}
\usepackage{enumerate}
\usepackage{tikz}
\usepackage{mathrsfs}
\usepackage[numbers,sort&compress]{natbib}

\usepackage{srcltx}
\usepackage{geometry}
\usepackage{verbatim}
\usepackage{mathrsfs}
\usepackage{hyperref}
\usepackage{enumitem} 
\usepackage{bbm} 
\usepackage{stmaryrd}

\usepackage{relsize}
\usepackage{exscale}

\usepackage{mathtools}

\usepackage{mathtools}

\newtheorem{thm}{Theorem}[section]

\newtheorem{prop}[thm]{Proposition}
\newtheorem{lem}[thm]{Lemma}

\newtheorem{rem}[thm]{Remark}

\numberwithin{equation}{section}

\def\dd{{\rm d}}
\hypersetup{bookmarksdepth=2}

\begin{document}
\title[Supercritical Mean--field Limit]{Supercritical mean--field limit for magnetized Hamiltonian dynamics}
\author{Immanuel Ben-Porat}
\address{Immanuel Ben Porat, The University of Texas at Austin, Austin, TX 78712, USA.}
\email{immanuel.benporat@austin.utexas.edu}

\author{Gui-Qiang G. Chen}
\address{Gui-Qiang G. Chen, Mathematical Institute, University of Oxford, Oxford OX2 6GG, UK.}
\email{gui-qiang.chen@maths.ox.ac.uk}

\author{Difan Yuan}
\address{Difan Yuan,
School of Mathematical Sciences and Laboratory of
Mathematics and Complex Systems, Beijing Normal University.}
 \email{yuandf@amss.ac.cn}


\begin{abstract}
The magnetized Vlasov--Poisson equation is a fundamental kinetic model for collisionless plasmas. We establish its quasi-neutral limit in two dimensions in the presence of a spatially inhomogeneous and time-dependent external magnetic field. In the same setting, we also establish the combined mean-field and quasi-neutral limit of the associated repulsive Coulomb particle system, thereby extending the results of \cite{han2021newton}. Both limits lead to the incompressible Euler equations with Lorentz forcing. A key structural feature is the pointwise skew-symmetry of the Lorentz operator, which produces exact cancellations in the modulated-energy estimates and is crucial for controlling the limiting dynamics. We further derive a closed vorticity formulation in which the sum of the fluid vorticity and the magnetic field is transported by the flow, coupled to an evolution equation for the spatial mean velocity. Finally, we establish global 
well-posedness of the resulting limiting system in the Lipschitz class.
\end{abstract}

\date{\today}
\maketitle
\section{Introduction}
We are concerned with rigorously justifying the quasi-neutral limit for magnetized plasmas, in particular the quasi-neutral limit for the Vlasov-Poisson equation with a non-uniform and 
time-dependent external magnetic field, as well as the combined mean-field quasi-neutral 
limit of Newton's second law with a Lorentz force.

\subsection{The magnetized Vlasov-Poisson equation.} 
A plasma is an ionized gas, {\it i.e.,} a system of interacting charged particles. Denoting by $(X(t,x,\xi),\Xi(t,x,\xi))$ the position and the velocity 
of a particle with initial position $x$ and initial  velocity $\xi$ and assuming that a planar Lorentz force $\mathbf{F}=-\mathbf{E}- B\Xi^{\perp}$ is exerted on the particles, then Newton's second law specified for $(X(t,x,\xi),\Xi(t,x,\xi))$ writes 
\begin{align*}
\begin{cases}
\frac{\dd}{\dd t}X(t,x,\xi)=\Xi(t,x,\xi), & X(0,x,\xi)=x,\\[1mm]
\frac{\dd}{\dd t}\Xi(t,x,\xi)=-\mathbf{E}(t,X(t,x,\xi))- B(t,X(t,x,\xi))\Xi^{\perp}(t,x,\xi),& \Xi(0,x,\xi)=\xi. 
\end{cases}    
\end{align*}
Suppose further that $f$ designates the number of particles per unit volume with position $x$ and velocity $\xi$ at time $t$. 
Assuming that no collisions occur means that 
$$
\frac{\dd }{\dd t}f(t,X(t,x,\xi),\Xi(t,x,\xi))=0,
$$ 
which, by the chain rule, implies 
\begin{align*}
\partial_{t}f+\Xi\cdot \nabla_{x}f-(\mathbf{E}+B\Xi^{\perp})\cdot\nabla_{\xi}f=0.  \label{}   
\end{align*}
This is the Vlasov equation -- we refer to \cite[Chapter 7]{chen2015introduction}  for an extensive discussion of its importance in Plasma Physics. The \textit{magnetized Vlasov-Poisson} equation, written in the two-dimensional ($2D$) periodic settings, is a nonlocal version of the Vlasov equation 
in which the electric force $\mathbf{E}$ is given self-consistently in terms of the distribution $f$ by coupling a  transport equation to a Poisson equation: 
\begin{align}\tag{VPB}
\begin{cases}
\partial_{t}f+\xi\cdot \nabla_{x}f-(\nabla_{x}\Phi+ B\xi^{\perp})\cdot\nabla_{\xi} f=0, \qquad f(0,\cdot)=f^{\mathrm{in}},\\[1mm]
\Delta_{x} \Phi=1-\rho,
\end{cases}
\label{magnetized vp}
\end{align}
with $\rho(t,x)=\int_{\mathbb{R}^{2}} f(t,x,\xi)\,\dd \xi$,
where $B:[0,T]\times \mathbb{T}^{2}\rightarrow \mathbb{R}$ is a time-dependent 
scalar field, the unknown is a time dependent probability density $f(t,x,\xi)=f(t,\cdot)\in \mathcal{P}(\mathbb{T}^{2}\times \mathbb{R}^{2})$, 
and we have abbreviated by $\xi^{\perp}$ the rotation of $\xi$ by $\pi/2$, 
that is, $\xi^{\perp}=(-\xi_{2},\xi_{1})$. 
In this paper, we are concerned with the \textit{quasi-neutral limit}  of \eqref{magnetized vp} and the \textit{supercritical mean field limit} 
of Newton's second law with a Lorentz force, which is the microscopic dynamics associated with \eqref{magnetized vp}. 

\subsection{The quasi-neutral and supercritical mean field limits}
Consider now the magnetized Vlasov-Poisson equation with the following scaling, known as the \textit{quasi-neutral scaling}:
\begin{align}\tag{VPB$\varepsilon$}
\begin{cases}
\partial_{t}f_{\varepsilon}+\xi\cdot \nabla_{x}f_{\varepsilon}-(\nabla_{x}\Phi_{\varepsilon}+ B\xi^{\perp})\cdot\nabla_{\xi} f_{\varepsilon}=0, \qquad\, f_{\varepsilon}(0,\cdot)=f^{\mathrm{in}}_{\varepsilon},\\
\varepsilon\Delta_{x} \Phi_{\varepsilon}=1-\rho_{\varepsilon},
\end{cases}
\label{magnetized vp with quasineutral scaling}
\end{align}
with $\rho_{\varepsilon}(t,x)=\int_{\mathbb{R}^{2}} f_{\varepsilon}(t,x,\xi)\ \dd \xi$. 
It is instructive at this stage to comment about the physical interpretation of the quantities involved in the above equation:
\begin{itemize}
     \item[{\rm(i)}] $\rho_{\varepsilon}$ is the density of  the electrons. 
   \item[{\rm(ii)}] $1$ in the Poisson equation
   is the stationary density of the ions. 
   \item[{\rm(iii)}] $\varepsilon$ is the electric permittivity constant, which measures how close the density of the electrons is to the density of the ions.  
\end{itemize}
Under uniform bounds that control $\varepsilon\Delta\Phi_\varepsilon$, one expects the electron density to approach the ion density as $\varepsilon\to0$. This limit regime is referred to as the quasi-neutral limit. 
Since plasmas with $0$ total charge exhibit a fluid-like behavior, 
the quasi-neutral limit can be thought of as a hydrodynamical limit. 

To understand in more detail how to link the Vlasov-Poisson equation with a fluid equation, we  define the \textit{current} to be the following averaged quantity 
\begin{align}
J_{\varepsilon}(t,x):=\int_{\mathbb{R}^{2}}\xi f_{\varepsilon}(t,x,\xi)\ \dd \xi, \label{definition of current}    
\end{align}
and assume momentarily that $B\equiv 0$. Assume formally that $\rho_\varepsilon>0$ and set $u_\varepsilon=J_\varepsilon/\rho_\varepsilon$.  The monokinetic ansatz $f_\varepsilon(t,x,\xi)=\rho_\varepsilon(t,x)\delta(\xi-u_\varepsilon(t,x))$ satisfies the kinetic equation precisely when the following system of equations is satisfied: 
\begin{align*}
\begin{cases}
\partial_{t}\rho_{\varepsilon}+\mathrm{div}_{x}J_{\varepsilon}=0,\\
\partial_{t}J_{\varepsilon}+\mathrm{div}_{x}(\frac{J_{\varepsilon}\otimes J_{\varepsilon}}{\rho_{\varepsilon}})+\rho_{\varepsilon}\nabla_{x}\Phi_{\varepsilon}=0,\\
\varepsilon\Delta_{x}\Phi_{\varepsilon}=1-\rho_{\varepsilon}. 
\end{cases}    
\end{align*}
Formally, assuming the uniform bounds strong enough to pass 
to the limit and well-prepared monokinetic data, 
then the vanishing limit $\varepsilon\to0$ 
yields the incompressible Euler equations:
\begin{align}\tag{E}
\begin{cases}
\partial_{t}v+\mathrm{div}_{x}(v\otimes v)+\nabla_{x} p=0, \\
\mathrm{div}_{x}v=0. 
\end{cases}  
\label{incompressible Euler}
\end{align}
Thus, when $B=0$, it is expected 
that 
$\rho_{\varepsilon}\underset{\varepsilon \rightarrow 0}{\rightharpoonup}1$
and $J_{\varepsilon}\underset{\varepsilon \rightarrow 0}{\rightharpoonup}v$.

\smallskip
Our second main objective is to study the supercritical mean field limit, which can be viewed as a \textit{combined quasi-neutral mean field limit}. More specifically, consider the microscopic dynamics associated with \eqref{magnetized vp with quasineutral scaling}, that is, the following system of $2\times 2N$ ODEs: 
\vspace{2pt}
\begin{align}\tag{NB$\varepsilon$}
 \begin{cases}
\frac{\dd}{\dd t}x_{i}^{N}(t)=\xi_{i}^{N}(t), &\,\, x_{i}^{N}(0)=x_{i}^{N,0},\\[1mm]
\frac{\dd}{\dd t}\xi_{i}^{N}(t)=-B(t,x_{i}^{N}(t))(\xi_{i}^{N}(t))^{\perp}-\frac{1}{N\varepsilon}\underset{j:j\neq i}{\sum}\nabla_{x} V(x_{i}^{N}(t)-x_{j}^{N}(t)),
&\,\,\xi_{i}^{N}(0)=\xi_{i}^{N,0}. 
\end{cases} 
\label{trajectories intro}
\end{align}
Here $V$ is the Green kernel on the torus, {\it i.e.}, the solution of 
$$
\Delta_{x}V=1-\delta_{0} 
\quad \mbox{on $\mathbb{T}^{2}\qquad$  \text{with} $\int_{\mathbb{T}^{2}}V(x)\, \dd x=0$}.$$

It is customary to refer to $V$ as the $2D$ periodic Coulomb interaction. 
The system \eqref{trajectories intro} of ODEs is Newton's second law written for each particle in the plasma, in contrast with the Vlasov-Poisson equation which provides a macroscopic description of the plasma. 
Thus, such a limit is well motivated from the physical point of view. Mathematically, the link between \eqref{trajectories intro} and \eqref{magnetized vp with quasineutral scaling} is provided via the \textit{empirical measure}: 
given a solution $(x_{1}(t),\cdots\!,x_{N}(t),\xi_{1}(t),\cdots\!,\xi_{N}(t))$ to \eqref{trajectories intro} (for brevity, we henceforth omit the superscript $N$), 
the empirical measure is the time-dependent probability measure defined by   
\begin{align*}
\mu_{N}^{\varepsilon}(t):=\frac{1}{N}\sum_{i=1}^{N}\delta_{(x_{i}(t),\xi_{i}(t))}\in \mathcal{P}(\mathbb{T}^{2}\times \mathbb{R}^{2}),  
\end{align*}
and is an exact solution to \eqref{magnetized vp with quasineutral scaling}. Similar to the above, the density and current are denoted by ($\rho_{\varepsilon,N},J_{\varepsilon,N}$), and defined by  
\begin{align}
\rho_{\varepsilon,N}(t,\dd x):=\frac{1}{N}\sum_{i=1}^{N}\delta_{x_{i}(t)}(\dd x),\qquad
J_{\varepsilon,N}(t,\dd x):=\frac{1}{N}\sum_{i=1}^{N}\xi_{i}(t)\delta_{x_{i}(t)}(\dd x).\label{discrete density and current}    
\end{align}

Again, taking $B=0$ and invoking the same considerations as before, we expect
$\rho_{\varepsilon,N}\rightharpoonup1$ and $J_{\varepsilon,N}\rightharpoonup v$
as $\varepsilon+N^{-1}\to0$, where $v$ solves system \eqref{incompressible Euler}.

\subsection{Previous literature} 
The rigorous derivation of the incompressible Euler equations as a quasi-neutral limit from the Vlasov-Poisson system has been proved by Brenier-Grenier \cite{brenier1994limite} and Brenier  \cite{brenier2000convergence} via the modulated energy method. For spatially analytic data, Grenier
\cite{grenier1996oscillations} studied quasineutral asymptotics
and the associated high-frequency plasma oscillations.
Masmoudi \cite{masmoudi2001vlasov} extended the cold-electron
quasineutral limit beyond well-prepared initial data,
on time intervals where the limiting incompressible Euler
solution remains regular. Later on,  Golse-Saint-Raymond \cite{golse2003vlasov} studied this limit in the combined gyrokinetic quasi-neutral regime, meaning in the presence of a constant magnetic field and with a different scaling. The quasi-neutral limit of the ionic Vlasov-Poisson equation, in which the Vlasov equation is coupled to a Poisson-Boltzmann-type equation, is also interesting, as it reduces to the equations of compressible fluids. 
We refer to \cite{han2011quasineutral, han2014quasineutral} for the derivation of the isothermal Euler equations from the Vlasov-Poisson equation for ions. 
We also mention \cite{ben2025derivation} for the study of the quasi-neutral limit in the case of velocity fields with low regularity and  \cite{rosenzweig2021quantum, puel2002convergence,jungel2003convergence} for combined quasi-neutral semi-classical limits, leading from dispersive dynamics to fluid dynamics.  
In case the electromagnetic force accounts for a self-consistent magnetic field, we refer to  \cite{puel2004quasineutral,gagnebin2025relativistic} where the quasi-neutral limit is studied for the relativistic Vlasov-Maxwell equations,  and to \cite{brenier2003incompressible} where 
the quasi-neutral limit is studied for the non-relativistic Vlasov-Maxwell equations. 
In both cases, the limit equations are of MHD-type.  
The monokinetic version of this problem, {\it i.e.} quasi-neutral limits for the 
Euler-Maxwell system, has also been addressed; see \cite{Peng2024, Peng2017}.  In a related fluid setting, Ju-Liu \cite{ju2026zero}
studied the distinct zero-electron-mass limit and proved
global-in-time convergence of smooth solutions of the
one-fluid Euler--Poisson system with non-constant ion density
to incompressible Euler equations, under uniformly
small initial perturbations. For ionic Euler--Poisson systems, Pu
\cite{pu2016quasineutral} studied the simultaneous quasineutral
and strong magnetic field limit, obtaining one-dimensional
compressible Euler dynamics along the magnetic field.
This regime differs from the unscaled external magnetic
field considered here.
Burby {\it et al.} \cite{burby2025hamiltonian} developed a formal
slow manifold reduction and a Hamiltonian formulation for
the planar quasineutral Vlasov--Poisson system with a static,
spatially varying magnetic field.
We also mention the long-time strong-field analysis of the Vlasov–Poisson–Fokker–Planck 
system by Bostan-Vu \cite{bostan2023asymptotic}, in which the vorticity formulation 
of the two-dimensional incompressible Euler equations is recovered 
in the uniform-field case. An expository account of contemporary research 
themes in quasi-neutrality is included in \cite{iacobelli2026quasineutral,griffin2021recent}.  

 Chen-Jung-Pickl-Wang \cite{chen2026vpfp}
established propagation of chaos for Newtonian particle systems
with regularized Coulomb interactions and a cutoff vanishing
as $N\to\infty$, deriving the Vlasov--Poisson--Fokker--Planck
equation, or the Vlasov--Poisson equation when velocity noise
is absent or asymptotically subdominant. To prove a joint mean-field and quasi-neutral limit requires new ideas 
 because the empirical measure is only a measure-valued solution 
 of the Vlasov--Poisson equation, rather than a sufficiently regular density. 
 In Han-Kwan--Iacobelli \cite{han2021newton}, 
 the combined mean-field and quasi-neutral limit together with the mean-field gyrokinetic 
 limit was proved, and the term ``supercritical'' was introduced to describe the limits of this type. 
 Such limits require a quantitative asymptotic relation between $\varepsilon$ and $N$, 
 often parametrized as $\varepsilon=\varepsilon(N)=N^{-\theta}$ for some $\theta>0$. 
 The admissible range of exponent $\theta$ was further investigated in \cite{rosenzweig2023rigorous}.

Quantum analogs of these limits, from many-body dynamics to incompressible fluids, were studied in \cite{rosenzweig2021quantum,porat2023derivation}, 
while derivations of compressible fluid models from quantum dynamics were developed in \cite{grenier1998semiclassical,jungel2003convergence,chen2025meanfield,porat2025quantum}. 
Recent advances also include the derivation of the lake equation through a supercritical 
mean-field limit \cite{rosenzweig2025lake} and a general framework 
for mean-field limits of singular flows in supercritical 
regimes \cite{rosenzweig2025commutators}. A key ingredient 
is the family of Coulomb functional inequalities developed by Serfaty \cite{duerinckx2020mean}, 
together with the modulated free-energy methods of \cite{bresch2020modulated}; 
both are essential to our analysis. We also refer 
Jiang-Qiao-Wu-Zhang \cite{jiang2026yukawa} for the application on the Yukawa modulated energy.

\subsection{New difficulties and strategy} 
Beyond the intrinsic difficulties of the supercritical regime, 
allowing the external magnetic field to be spatially inhomogeneous and time-dependent 
introduces two additional challenges:

\smallskip
\noindent\textbf{Fluid level:}
Unlike the unmagnetized quasi-neutral limits in \cite{brenier1994limite,brenier2000convergence} 
and the constant-field problem considered under a different gyrokinetic scaling 
in \cite{golse2003vlasov}, the Lorentz force in the present setting remains 
in the limiting Euler equations. 
The spatial dependence of $B$ modifies the vorticity equation, while the spatial mean velocity 
is generally no longer conserved. 
Consequently, the vorticity alone does not determine the full velocity field on the torus. 
To overcome this difficulty, we introduce the modified vorticity $\Omega=\omega+B$, 
which satisfies the transport equation:
$$
\partial_t\Omega+v\cdot\nabla_x\Omega=\partial_tB,
$$
and couple it to the evolution equation for the spatial mean velocity:
$$
\dot{\bar v}(t)=-\int_{\mathbb T^2}B(t,x)v^\perp(t,x)\,\dd x
$$
 This yields a closed vorticity formulation 
and provides the estimates needed to establish the global well-posedness 
of the limiting system in the Lipschitz class.

\smallskip
\noindent\textbf{Kinetic and particle level:}
At the particle level, the empirical density is atomic and the Coulomb kernel is singular on the diagonal, so the classical modulated energy cannot be evaluated directly.
At both the kinetic and particle levels, the variable magnetic field also generates additional 
Lorentz terms in the modulated energy estimates. A key observation is that the magnetic rotation 
$u\mapsto B(t,x)u^\perp$ is skew-symmetric at every point. Consequently, the Lorentz contributions from the microscopic dynamics and the limiting Euler equations cancel exactly, even when $B$ depends on both space and time. To control the remaining singular particle terms, we integrate away from the diagonal and introduce a renormalized modulated energy relative to the corrected background $1+\varepsilon\mathfrak{U}$, where $\mathfrak{U}=-\Delta_xp$. The Coulomb functional inequalities from \cite{duerinckx2020mean,bresch2020modulated,han2021newton}, together with the finite-$N$ positivity correction, then close the estimate under the condition
$\frac{\log N}{N\varepsilon}\to0$.

\subsection{Main results}
In this paper, we investigate the quasi-neutral and supercritical mean-field limits 
for \eqref{magnetized vp with quasineutral scaling} and \eqref{trajectories intro}, respectively, 
in the presence of a \textit{spatially non-uniform and time-dependent} external magnetic field
$B$. 
Our first objective is to establish the quasi-neutral limit and derive the corresponding magnetized 
fluid dynamics from \eqref{magnetized vp with quasineutral scaling}. 
As shown formally in Section \ref{sec formal derivation}, 
the limiting system is the following magnetized incompressible Euler equations:
\begin{align}\tag{EB}
\begin{cases}
\partial_{t}v+\mathrm{div}_{x}(v\otimes v)+Bv^{\perp}+\nabla_{x} p=0, \quad 
v(0,\cdot)=v^{\mathrm{in}},\\[1mm]
\mathrm{div}_{x}v=0,\quad  \int_{\mathbb{T}^{2}}p(t,x)\ \dd x=0. 
\end{cases}    
\label{magnetized Euler}  
\end{align}
As will be shown in Section \ref{sec well posedness}, 
the equivalent vorticity formulation of \eqref{magnetized Euler} reads 
\begin{align}
\begin{cases}
\partial_{t}\omega+\mathrm{div}_{x}(v\omega)
+\nabla_{x}B\cdot v=0,
\quad \omega(0,\cdot)=\omega^{\mathrm{in}},
\quad v(0,\cdot)=-\nabla_{x}^{\perp}V\ast\omega^{\mathrm{in}},\\[1mm]
v
=\int_{\mathbb{T}^{2}}v(t,x)\,\dd x
-\nabla_x^\perp V\ast\omega,
\end{cases}
\label{vorticity formulation intro}
\end{align}
which implies
\begin{equation}\label{1.4a}
\displaystyle
\frac{\dd}{\dd t}\int_{\mathbb{T}^{2}}v(t,x)\,\dd x
=
-\int_{\mathbb{T}^{2}}B(t,x)v^{\perp}(t,x)\,\dd x.
\end{equation}
Define 
\begin{equation}\label{1.5a}
\Delta_{x}\psi=\omega,\qquad
\int_{\mathbb{T}^{2}}\psi(t,x)\,\dd x=0,
\end{equation}
then
\begin{equation}\label{1.6a}
\nabla_x^\perp V\ast\omega=-\nabla_{x}^{\perp}\psi,
\qquad
v
=\int_{\mathbb{T}^{2}}v(t,x)\,\dd x+\nabla_{x}^{\perp}\psi.
\end{equation}

Equation \eqref{magnetized Euler} has already appeared previously in the literature 
(see, for instance, Remark 2.5 in \cite{rege2025stability}. 
Also, see  \cite{gallagher2005pressureless} for a different, but related, 
magnetized Euler equations). 
Nevertheless, the identification of \eqref{magnetized Euler} as a quasi-neutral limit 
of \eqref{magnetized vp with quasineutral scaling} has not been addressed, 
and our first theorem aims to fill in this gap by adapting Brenier's modulated 
energy method. 
In this context, part of our contribution is the proof of the global well-posedness 
of \eqref{magnetized Euler} and its reformulation in vorticity form. 
To state our first main result, we need to define the following auxiliary quantity known as the \textit{modulated energy}. 
Given a solution $v$ to \eqref{magnetized Euler} and a solution $f_{\varepsilon}$ 
to \eqref{magnetized vp with quasineutral scaling}, we define the modulated energy by 
\begin{align}
\mathcal{E}_{\varepsilon}(t):= \frac{1}{2}\int_{\mathbb{T}^{2}\times \mathbb{R}^{2}} \left\vert \xi-v(t,x)\right\vert^{2}f_{\varepsilon}(t,x,\xi)\ \dd x\dd \xi + \frac{\varepsilon}{2}\int_{\mathbb{T}^{2}} \left\vert \nabla_{x}\Phi_{\varepsilon}\right\vert^{2}(t,x) \ \dd x. \label{modulated energy def}    
\end{align}

\begin{thm}\label{main thm 1}
Suppose that 
\begin{itemize}
\item[{\rm(i)}] $B\in W^{1,\infty}(\mathbb{R}_{+};W^{2,\infty}(\mathbb{T}^{2}))${\rm;}
\item[{\rm(ii)}]
$f_{\varepsilon}\in W^{1,\infty}([0,T]\times \mathbb{T}
    ^{2}\times \mathbb{R}^{2})$ is a solution to \eqref{magnetized vp with quasineutral scaling} with $f^{\mathrm{in}}_{\varepsilon}\in \mathcal{P}(\mathbb{T}^{2}\times \mathbb
    {R}
    ^{2}),$ finite second velocity moment and sufficient velocity decay{\rm;}
\item[{\rm(iii)}]
    $\omega^{\mathrm{in}}\in W^{1,\infty}( \mathbb{T}^{2})$, 
    $\int_{\mathbb T^2}\omega^{\mathrm{in}}\,\dd x=0$,  
    and $(\omega,v)$ is the solution to \eqref{vorticity formulation intro}  on $[0,T]$ with initial data $\omega^{\mathrm{in}}$. 
\end{itemize}
 Let $\mathcal{E}_{\varepsilon}(t)$ be given by \eqref{modulated energy def}.
Then
\begin{align*}
\underset{t\in [0,T]}{\sup} \mathcal{E}_{\varepsilon}(t)\underset{\varepsilon \rightarrow 0}{\longrightarrow}0,
\end{align*}
provided $\mathcal{E}_{\varepsilon}(0)\underset{\varepsilon \rightarrow 0}{\longrightarrow}0$.    
This implies that
\begin{align*}
\rho_{\varepsilon}(t,\cdot)\xrightharpoonup[\varepsilon \rightarrow 0]{}1\quad\, 
\mbox{and} \quad\, J_{\varepsilon}(t,\cdot)
\xrightharpoonup[\varepsilon \rightarrow 0]{}v
\end{align*}
weakly in the sense of measures.  
\end{thm}

Our second main objective is to derive \eqref{magnetized Euler} as a supercritical mean field limit from \eqref{trajectories intro}. Our proof is modeled after the combined mean field quasi-neutral limit proved in \cite{han2021newton}, which in turn is inspired by the \textit{renormalized} modulated energy method developed in \cite{duerinckx2020mean} (and further extended in \cite{bresch2020modulated}). Prior to stating our second main result, we elaborate on the latter notion. 
Given a solution $v$ to \eqref{magnetized Euler} we set $\mathfrak{U}:= -\Delta_{x} p$. 
Note that, upon taking the divergence in \eqref{magnetized Euler}, 
we can write explicitly 
\begin{align*}
-\Delta_{x}p=\mathrm{div}_{x}\mathrm{div}_{x}(v\otimes v)+\mathrm{div}_{x}(Bv^{\perp})
=\sum_{1\leq i,j\leq 2}\partial_{x_{j}}v^{i}\partial_{x_{i}}v^{j}+\mathrm{div}_{x}(Bv^{\perp}). 
\end{align*}
Given a solution $(x_{1}(t),\cdots\!,x_{N}(t),\xi_{1}(t),\cdots\!,\xi_{N}(t))$ 
to \eqref{trajectories intro}, we define the renormalized modulated energy by  
\begin{align}
\mathcal{E}_{\varepsilon,N}(t):=
&\, \frac{1}{2N}\sum_{i=1}^{N}\left\vert \xi_{i}(t)-v(t,x_{i}(t))\right\vert^{2} \notag\\
&\,+\frac{1}{2\varepsilon}
\iint_{\Delta^{c}}V(x-y)\left(\rho_{\varepsilon,N}-1-\varepsilon \mathfrak{U}\right)^{\otimes 2}(t,\dd x\dd y) \notag\\
&\,+\frac{C}{2\varepsilon}\big(1+\left\Vert 1+\varepsilon\mathfrak{U} \right\Vert_{L^\infty([0,T]\times\mathbb T^2)}\big)
\frac{1+\log N}{N}.
\label{magnetized renormalized modulated energy}       
\end{align} 
A few remarks are in order, following the definition of  $\mathcal{E}_{\varepsilon,N}(t)$: 
\begin{itemize}
    \item[{\rm(a)}] The constant $C>0$ in \eqref{magnetized renormalized modulated energy}
    is the same as in Proposition \ref{non negativity} in Section \ref{sec mean field}. We must include the last term in \eqref{magnetized renormalized modulated energy} in order to ensure that $\mathcal{E}_{\varepsilon,N}(t)$ is non-negative.
    \item[{\rm(b)}]  $\Delta^{c}$ designates the complement of the diagonal, {\it i.e.}, 
    $\Delta^{c}=(\mathbb{T}^{2}\times \mathbb{T}^{2})\setminus \Delta=\{(x,y)\in \mathbb{T}^{2}\times \mathbb{T}^{2}\,:\, x\neq y\}$. 
    More generally, we invoke the notation 
    $$
    \Delta_{N}^{c}:=\{(x_{1},\cdots\!,x_{N})\in \mathbb{T}^{2N}\,:\, x_i\neq x_j\ \text{for all } i\neq j\}.
    $$
    \item[{\rm(c)}]  It is essential to integrate over $\Delta^{c}$ 
    in \eqref{magnetized renormalized modulated energy}, 
    since $\rho_{\varepsilon,N}(t,\cdot)$ is an atomic measure while $V$ has a singularity at the origin.
\end{itemize}

\begin{thm}
Suppose that
\begin{itemize}
  \item[{\rm(i)}] $\omega^{\mathrm{in}}\in W^{1,\infty}( \mathbb{T}^{2})$, $\int_{\mathbb T^2}\omega^{\mathrm{in}}\,\dd x=0,$ $B\in W^{1,\infty}(\mathbb{R}_{+};W^{2,\infty}(\mathbb{T}^{2}))$, 
  and $(\omega,v)$ is the solution to \eqref{vorticity formulation intro} with initial data $\omega^{\mathrm{in}};$
    \item[{\rm(ii)}] $(x_{1}^{0},\cdots\!,x_{N}^{0})\in \Delta_{N}^{c}$, 
    $(x_{1}(t),\cdots\!,x_{N}(t),\xi_{1}(t),\cdots\! ,\xi_{N}(t))\in C^{1}(\mathbb{R}_{+};\mathbb{T}^{2N}\times \mathbb{R}^{2N})$ is the solution to \eqref{trajectories intro},
    and $\rho_{\varepsilon,N}$ and 
    $J_{\varepsilon,N}$ are given by \eqref{discrete density and current}{\rm;}
   \item[{\rm(iii)}] $\mathcal{E}_{\varepsilon,N}(t)$ is given by \eqref{magnetized renormalized modulated energy}{\rm;}
    \item[{\rm(iv)}]$\varepsilon=\varepsilon(N)\to0$ is such that  $\mathcal{E}_{\varepsilon,N}(0)\underset{N\rightarrow \infty}{\longrightarrow }0$ and $\frac{\log N}{N \varepsilon}\underset{N\rightarrow \infty}{\longrightarrow}0$. 
\end{itemize}
Then
\begin{align*}
\underset{t\in [0,T]}{\sup}\mathcal{E}_{\varepsilon,N}(t)\underset{\varepsilon+\frac{1}{N}\rightarrow 0}{\longrightarrow} 0,
\end{align*}
provided 
$\mathcal{E}_{\varepsilon,N}(0)\underset{\varepsilon+\frac{1}{N}\rightarrow 0}{\longrightarrow}0$.    
This implies that
\begin{align*}
\rho_{\varepsilon,N}(t,\cdot)\xrightharpoonup[{\varepsilon+\frac{1}{N}\rightarrow 0}]{}1 
\quad\,\, \mbox{and} \quad\,\, 
J_{\varepsilon,N}\xrightharpoonup[\varepsilon+\frac{1}{N}\rightarrow 0]{}v
\end{align*}
weakly in the sense of measures. \label{main thm renormalized}
\end{thm}

The paper is organized as follows: 
In Section \ref{sec formal derivation}, we give a detailed account of the formal calculations 
leading to \eqref{magnetized Euler}. 
In Section \ref{section quasi neutral}, we prove the quasi-neutral limit leading 
from \eqref{magnetized vp with quasineutral scaling} to \eqref{magnetized Euler} 
in the limit as $\varepsilon \rightarrow 0$. 
In Section \ref{sec mean field}, we adapt the methods of \cite{duerinckx2020mean} 
and \cite{han2021newton} in order to prove the derivation of \eqref{magnetized Euler} from \eqref{trajectories intro} in the joint limit $\varepsilon+\frac{1}{N}{\rightarrow}0$. 
Sections \ref{sec well posedness} and \ref{wellposed-trajectories-sec} are devoted 
to the well-posedness of \eqref{magnetized Euler} and \eqref{trajectories intro}, respectively.

\section{Formal Derivation}\label{sec formal derivation}
We explain the formal manner in which \eqref{magnetized Euler} is derived from \eqref{magnetized vp with quasineutral scaling} in the limit as $\varepsilon \rightarrow 0$. The following notations will be frequently invoked in the sequel: 
\begin{itemize} 
\item[{\rm(a)}]Given a vector field $v:\mathbb{T}^{d}\rightarrow \mathbb{R}^{d}$, 
denote by  $D_{x}v$ the \textit{Jacobian matrix} of $v$ defined by $(D_{x}v)_{ij}=(\frac{\partial v^{j}}{\partial x_{i}})_{ij}$. 
\item[{\rm(b)}] Given a matrix valued function $\mathbf{a}:\mathbb{R}^{d}\rightarrow M_{d}(\mathbb{R})$, denote by $\mathrm{div}(\mathbf{a}):\mathbb{R}^{d}\rightarrow \mathbb{R}^{d}$ the vector field defined componentwise by  $\mathrm{div}(\mathbf{a})_{i}:=\mathrm{div}(\mathbf{a}_{i})$ for all $1\leq i\leq d$ where $\mathbf{a}_{i}$ stands for  the $i$-th row of $\mathbf{a}$. 
\item[{\rm(c)}] 
Given $d$-dimensional vectors $u,v$, denote by $u\otimes v$ the tensor product of $u$ and $v$, 
that is the $d\times d$ matrix whose $ij$-th entry is $u_{i}v_{j}$.
\item[{\rm(d)}] Given matrices $\mathbf{a},\mathbf{b}\in M_{d}(\mathbb{R})$, 
we use the convention $\mathbf{a}:\mathbf{b}=\operatorname{tr}(\mathbf{a}\mathbf{b})$, {\it i.e.,} 
    \begin{align*}
\mathbf{a}:\mathbf{b}:=\sum_{i,j}a_{ij}b_{ji}.    
    \end{align*}
\end{itemize}
We have the following basic identities from vector calculus. 
\begin{lem} \label{elementary identities}
Let $u,v,w:\mathbb{T}^{d}\rightarrow \mathbb{R}^{d}$ be given smooth vector fields. 
Then the following formulas hold{\rm:} 
\begin{itemize}  \item[{\rm(i)}]
$\mathrm{div}_{x}(u\otimes v)=\mathrm{div}_{x}(v)u+(D_xu)^{\top}v$.
 \item[{\rm(ii)}]
$u(v\cdot w)=(u\otimes v)w.$ 
 \item[{\rm(iii)}] If $f:\mathbb{T}^{d}\rightarrow \mathbb{R}$ and 
$\mathbf{a}:\mathbb{T}^{d}\rightarrow M_{d}(\mathbb{R})$ are smooth, 
then 
\begin{align*}
\int_{\mathbb{T}^{d}}\mathbf{a}(x)\nabla_{x}f(x)\ \dd x
=-\int_{\mathbb{T}^{d}}\mathrm{div}_{x}(\mathbf{a})f(x)\ \dd x.     
\end{align*}
 \item[{\rm(iv)}] If $\mathbf{a}:\mathbb{T}^{d}\rightarrow M_{d}(\mathbb{R})$ is smooth, 
 then 
\begin{align*}
\mathrm{div}_{x}(\mathbf{a}u)
=\mathrm{div}_{x}(\mathbf{a}^\top)\cdot u+\mathbf{a}^\top:D_{x}u.
\end{align*}

 \item[{\rm(v)}] If $v:\mathbb{T}^{2}\rightarrow \mathbb{R}^{2}$ 
 is a smooth divergence-free vector field, {\it i.e.}, $\mathrm{div}_{x}v=0$,
 then 
\begin{align*}
-\nabla^{\perp}_{x}(v\cdot \nabla_{x} \mathrm{div}_{x}(v^{\perp}))
=\Delta_{x} \mathrm{div}_{x}(v\otimes v)-\nabla_{x} \mathrm{div}_{x}\mathrm{div}_{x}(v\otimes v).    
\end{align*}
\end{itemize}
\end{lem}

\begin{prop}\label{theorem2.2}
Let $f_{\varepsilon}$ be a smooth solution to \eqref{magnetized vp with quasineutral scaling} with finite second velocity moment and sufficient decay in $\xi,$ and let $J_{\varepsilon}$ be given by \eqref{definition of current}. Then the pair $(\rho_{\varepsilon},J_{\varepsilon})$ satisfies the following equations{\rm:}
\begin{align}
&\partial_{t}\rho_{\varepsilon}+\mathrm{div}_{x}J_{\varepsilon}=0, 
\label{density}\\ 
&\partial_{t}J_{\varepsilon}
+\mathrm{div}_{x}\Big(\int_{\mathbb{R}^{2}}\xi \otimes \xi f_{\varepsilon} \ \dd \xi\Big)+BJ_{\varepsilon}^{\perp}+\rho_{\varepsilon}\nabla_{x} \Phi_{\varepsilon}=0. 
\label{equation for Jepsilon}  
\end{align}
\end{prop}

\begin{proof}
Equation \eqref{density} is obtained immediately by integrating \eqref{magnetized vp with quasineutral scaling} with respect to $\xi$. 
We proceed by deriving equation \eqref{equation for Jepsilon}. Multiplying \eqref{magnetized vp with quasineutral scaling} by $\xi$ and integrating in $\xi$, we have 
\begin{align*}
\partial_{t}J_{\varepsilon} +\int_{\mathbb{R}^{2}} (\xi\otimes \xi) \nabla _{x}f_{\varepsilon}\ \dd \xi-\int_{\mathbb{R}^{2}} \xi \nabla_{x}\Phi_{\varepsilon}\cdot \nabla_{\xi} f_{\varepsilon} \ \dd \xi-\int_{\mathbb{R}^{2}}\xi \otimes (B\xi^{\perp})\nabla_{\xi}f_{\varepsilon} \ \dd \xi=0. 
\end{align*}
Integrating by parts yields
\begin{align}
-\int_{\mathbb{R}^{2}}\xi \nabla_{x}\Phi_{\varepsilon}\cdot \nabla_{\xi}f_{\varepsilon}\ \dd \xi=\nabla_{x}\Phi_{\varepsilon}\int_{\mathbb{R}^{2}} f_{\varepsilon}\ \dd \xi=\nabla_{x}\Phi_{\varepsilon}\rho_{\varepsilon},  \label{third term}   
\end{align}
Therefore, it follows that $J_{\varepsilon}$ verifies the equation:  
\begin{align}
&\partial_{t}J_{\varepsilon}+\int_{\mathbb{R}^{2}}\xi\otimes \xi\nabla_{x}f_{\varepsilon}\ \dd \xi \notag+\nabla_{x}\Phi_{\varepsilon}\rho_{\varepsilon}-\int_{\mathbb{R}^{2}}\xi\otimes (B\xi ^{\perp})\nabla_{\xi}f_{\varepsilon}\ \dd \xi=0. \label{first eq for J}     
\end{align}
Using Lemma \ref{elementary identities}(3), we find 
\begin{align}
-\int_{\mathbb{R}^{2}}\xi\otimes (B\xi^{\perp})\nabla_{\xi}f_{\varepsilon} \ \dd \xi=\int_{\mathbb{R}^{2}}\mathrm{div}_{\xi}\big(\xi\otimes (B\xi^{\perp})\big)
f_{\varepsilon}\ \dd \xi.   
\end{align}
With the aid of Lemma \ref{elementary identities}(i), we compute 
\begin{align*}
\mathrm{div}_{\xi}(\xi\otimes (B\xi^{\perp}))=B\mathrm{div}_{\xi}(\xi\otimes \xi^{\perp})=B(\mathrm{div}_{\xi}(\xi^{\perp})\xi+\xi^{\perp})=B\xi^{\perp},   
\end{align*}
where we have used that $\mathrm{div}_{\xi}(\xi^{\perp})=0$ in the last equation. 
Consequently, we deduce 
\begin{align*}
-\int_{\mathbb{R}^{2}} \xi\otimes B\xi^{\perp}\nabla_{\xi}f_{\varepsilon}\ \dd \xi=B\int_{\mathbb{R}^{2}} \xi^{\perp} f_{\varepsilon}\ \dd \xi=BJ_{\varepsilon}^{\perp}.  
\end{align*}
Differentiating under the integral sign and using the row-wise definition of divergence, we obtain
\begin{align}
\mathrm{div}_{x}\Big(\int_{\mathbb{R}^{2}}\xi\otimes \xi f_{\varepsilon} \ \dd \xi\Big)
=\int_{\mathbb{R}^{2}}\xi\otimes \xi\nabla_{x}f_{\varepsilon}\ \dd \xi. \label{identity for div}    
\end{align}
To conclude, we have 
\begin{align}
\partial_{t}J_{\varepsilon}+\mathrm{div}_{x}
\Big(\int_{\mathbb{R}^{2}}\xi\otimes \xi f_{\varepsilon} \ \dd \xi\Big)+BJ_{\varepsilon}^{\perp}+\rho_{\varepsilon}\nabla_{x} \Phi_{\varepsilon}=0,    
\end{align}
which is \eqref{equation for Jepsilon}. 
\end{proof}
Now formally assume $\rho_\varepsilon>0$ and take $f_{\varepsilon}$ to be monokinetic, {\it i.e.,} $f_{\varepsilon}(t,x,\xi)=\rho_{\varepsilon}(t,x)\delta(\xi-\frac{J_{\varepsilon}(t,x)}{\rho_{\varepsilon}(t,x)})$ so that \eqref{equation for Jepsilon} becomes 
\begin{align}
\partial_{t}J_{\varepsilon}+\mathrm{div}_{x}\Big(\frac{J_{\varepsilon}\otimes J_{\varepsilon}}{\rho_{\varepsilon}}\Big)
+BJ_{\varepsilon}^{\perp}+\rho_{\varepsilon}\nabla_{x}\Phi_{\varepsilon}=0.   \label{evolution of curremt for monokinetic ansatz}  
\end{align}
Taking the divergence in $x$ of \eqref{evolution of curremt for monokinetic ansatz} and using that $\mathrm{div}_{x}J_{\varepsilon}=-\partial_{t}\rho_{\varepsilon}=\varepsilon \partial_{t}\Delta_{x}\Phi_{\varepsilon}$, we obtain  
\begin{align}
&\varepsilon \partial_{tt}\Delta_{x} \Phi_{\varepsilon}+\mathrm{div}_{x}\mathrm{div}_{x}\Big(\frac{J_{\varepsilon}\otimes J_{\varepsilon}}{\rho_{\varepsilon}}\Big)\notag+\mathrm{div}_{x}(BJ_{\varepsilon}^{\perp})+\mathrm{div}_{x}(\nabla_{x} \Phi_{\varepsilon}(1-\varepsilon\Delta_{x} \Phi_{\varepsilon}))=0. 
\label{equation for deltaphieps}       
\end{align} 
To conclude, we arrive at the following system of equations:  
\begin{align*}
\begin{cases}
\partial_{t}\rho_{\varepsilon}+\mathrm{div}_{x}(J_{\varepsilon})=0,\\
\partial_{t}J_{\varepsilon}+\mathrm{div}_{x}\big(\frac{J_{\varepsilon}\otimes J_{\varepsilon}}{\rho_{\varepsilon}}\big)+BJ_{\varepsilon}^{\perp}+\rho_{\varepsilon}\nabla_{x}\Phi_{\varepsilon}=0,\\
-\Delta_{x} \Phi_{\varepsilon}=\mathrm{div}_{x}(BJ_{\varepsilon}^{\perp})
+\mathrm{div}_{x}\mathrm{div}_{x}\big(\frac{J_{\varepsilon}\otimes J_{\varepsilon}}{\rho_{\varepsilon}}\big)+\varepsilon\partial_{tt}\Delta_{x} \Phi_{\varepsilon}-\varepsilon\mathrm{div}_{x}(\nabla_{x}\Phi_{\varepsilon}\Delta_{x}\Phi_{\varepsilon}). 
\end{cases}    
\end{align*}
Therefore, under some uniform bounds sufficient to pass to the limit, 
letting $\varepsilon\rightarrow 0$, we formally obtain
\begin{align*}
\begin{cases}
\partial_{t}v+\mathrm{div}_{x}(v\otimes v)+Bv^{\perp}+\nabla_{x}p=0,\\
\mathrm{div}_{x}v=0,
\end{cases}
\end{align*}
so that
$$
-\Delta_{x} p=\mathrm{div}_{x}\mathrm{div}_{x}(v\otimes v)+\mathrm{div}_{x}(Bv^{\perp}). 
$$

\smallskip
\section{Well-Posedness of \eqref{magnetized Euler}}\label{sec well posedness}

In this section, we prove the existence and uniqueness of global classical solutions 
of the Cauchy problem associated with \eqref{magnetized Euler}. 
In the following lemma, we establish the equivalence between the velocity and vorticity formulation 
of the magnetized incompressible Euler equations, 
as it will be more convenient to prove the well-posedness for the equations 
in vorticity formulation. 

\begin{lem} \label{Hodge decomposition lemma}
Let $v$ be a smooth vector field on $\mathbb{T}^{2}$, 
and let $\omega $ be a smooth scalar field on $\mathbb{T}^{2}$ 
such that $\int_{\mathbb{T}^{2}}\omega(x)\ \dd x=0$ and 
\begin{align*}
\omega=\mathrm{curl}_{x}(v)=-\mathrm{div}_{x}(v^{\perp}).     
\end{align*}
Then there is a constant vector $m\in \mathbb{R}^{2}$ and smooth scalar functions $\phi$ and $\psi$
such that 
\begin{equation}\label{rep of v}
v=m+\nabla_{x}\phi+\nabla_{x}^{\perp}\psi, \quad \omega=\Delta_{x}\psi,
\quad \int_{\mathbb{T}^{2}}\psi(x)\, \dd x=0.  
\end{equation}
\end{lem}

\begin{proof}
First, recall that $v_{\ast}=-\nabla_{x}^{\perp}V\ast \omega$ 
is a \textit{specific} solution determined by $\omega$ 
to the equation $\mathrm{curl}_{x}(v)=\omega$. 
Equivalently,   $v_{\ast}$ is the unique solution to 
\begin{align*}
v_{\ast}=\nabla_{x}^{\perp}\psi, \quad \Delta_{x} \psi=\omega, 
\quad \int_{\mathbb{T}^{2}}\psi(x)\ \dd x=0.      
\end{align*}
Suppose that $v$ is another solution satisfying  $\mathrm{curl}_{x}(v)=\omega$,  
so that $\mathrm{curl}_{x}(v-v_{\ast})=0$. 
By Poincar\'e's lemma, there is some constant $m\in \mathbb{R}^{2}$ and some 
smooth  scalar function $\phi\in C^{\infty}
(\mathbb{T}^{2})$ such that 
\begin{align*}
v-v_{\ast}=m+\nabla_{x}\phi\Rightarrow v=m+\nabla_{x}\phi+v_{\ast},     
\end{align*}
which establishes \eqref{rep of v}. 
\end{proof}

\begin{lem}\label{lemma3.2}
Suppose that $(v,p)$ is a classical solution to  \eqref{magnetized Euler}. 
Set 
\begin{align*}
\bar{v}(t):= \int_{\mathbb{T}^{2}}v(t,x)\ \dd x, \quad 
\omega=\mathrm{curl}_{x}(v)=-\mathrm{div}_{x}(v^{\perp}).  \end{align*} 
Then $(\omega,v)$ is a classical solution to the system{\rm:}
\begin{align}
\begin{cases}
\partial_{t}\omega+\mathrm{div}_{x}(v\omega)+\nabla_{x} B\cdot v = 0, \\[1mm]
\partial_t\bar{v}=-\int_{\mathbb{T}^{2}}B(t,x)v^{\perp}(t,x)\ \dd x,\\[1mm]
\ \omega =\Delta_{x} \psi,\quad v=\bar{v}(t)+\nabla^{\perp}_{x}\psi,
\quad \int_{\mathbb{T}^{2}}\psi(t,x)\ \dd x=0.
\end{cases}  
\label{vorticity formulation lemma}
\end{align}
Conversely, if  $(\omega,v)$ is a classical solution to
\eqref{vorticity formulation},  define 
$$
p=(-\Delta_{x})^{-1}
\big(\mathrm{div}_{x}\mathrm{div}_{x}(v\otimes v)+\mathrm{div}_{x}(Bv^{\perp})\big),
\quad \int_{\mathbb T^2}p\,\dd x=0.
$$ 
Then $(v,p)$ is a classical solution to  \eqref{magnetized Euler}.  \label{vorticity formulation}
\end{lem}
\begin{proof} We divide the proof into two steps.

\smallskip
\textbf{1}. 
Suppose that $(v,p)$ satisfies \eqref{magnetized Euler}. 
We wish to show that $(\omega,v)$ is a solution to \eqref{vorticity formulation lemma}. 

First, we have
\begin{align*}
\partial_{t}v^{\perp}+(vD_{x}v)^{\perp}-Bv+(\nabla_{x} p)^{\perp}=0.    
\end{align*}
Taking the divergence and using that $\mathrm{div}_{x}(\nabla^{\perp}_{x} p)=0$ and 
$\mathrm{div}_{x}(v)=0$, we obtain
\begin{align*}
\partial_{t}\omega-\mathrm{div}_{x}(vD_{x} v)^{\perp}+\nabla_{x} B\cdot v=0.  
\end{align*}
We claim that the following identity holds: 
\begin{align}
-\mathrm{div}_{x}(vD_{x} v)^{\perp}=\mathrm{div}_{x}(v\omega). 
\label{div of vnablav}
\end{align} 
To prove \eqref{div of vnablav}, we start by computing the right-hand side: 
\begin{align}
\mathrm{div}_{x}(v\omega)&=v\cdot \nabla_{x} \omega=-v\cdot \nabla_{x} \mathrm{div}_{x}(v^{\perp})=-(v_{1},v_{2})\cdot \nabla_{x} (-\partial_{x_{1}}v_{2}+\partial_{x_{2}}v_{1}) \notag\\
&=-(v_{1},v_{2})\cdot (-\partial_{x_{1}x_{1}}v_{2}+\partial_{x_{1}x_{2}}v_{1},
\,-\partial_{x_{1}x_{2}}v_{2}+\partial_{x_{2}x_{2}}v_{1})\notag\\
&=v_{1}\partial_{x_{1}x_{1}}v_{2}-v_{1}\partial_{x_{1}x_{2}}v_{1}+v_{2}\partial_{x_{1}x_{2}}v_{2}-v_{2}\partial_{x_{2}x_{2}}v_{1}. \label{div(v omega)}
\end{align}
We proceed by computing the left-hand side of \eqref{div of vnablav}: 
\begin{align}
-\mathrm{div}_{x}(vD_{x}v)^{\perp}&=-\mathrm{div}_{x}(v_{1}\partial_{x_{1}}v_{1}+v_{2}\partial_{x_{2}}v_{1},v_{1}\partial_{x_{1}}v_{2}+v_{2}\partial_{x_{2}}v_{2})^{\perp} \notag\\
&=\mathrm{div}_{x}(v_{1}\partial_{x_{1}}v_{2}+v_{2}\partial_{x_{2}}v_{2},-v_{1}\partial_{x_{1}}v_{1}-v_{2}\partial_{x_{2}}v_{1})\notag\\
&=v_{1}\partial_{x_{1}x_{1}}v_{2}+v_{2}\partial_{x_{1}x_{2}}v_{2}-v_{1}\partial_{x_{1}x_{2}}v_{1}-v_{2}\partial_{x_{2}x_{2}}v_{1}. \label{-div vnabla v}
\end{align}
Note that, in the last equation, we have used the identity: 
\begin{align*}
&\partial_{x_{1}}v_{1}\partial_{x_{1}}v_{2}+\partial_{x_{1}}v_{2}\partial_{x_{2}}v_{2}-\partial_{x_{2}}v_{1}\partial_{x_{1}}v_{1}-\partial_{x_{2}}v_{2}\partial_{x_{2}}v_{1}\\
&=\partial_{x_{1}}v_{1}(\partial_{x_{1}}v_{2}-\partial_{x_{2}}v_{1})+\partial_{x_{2}}v_{2}(\partial_{x_{1}}v_{2}-\partial_{x_{2}}v_{1})\underset{(\partial_{x_{1}}v_{1}=-\partial_{x_{2}}v_{2})}{=}0.   
\end{align*}
Comparing \eqref{-div vnabla v} with \eqref{div(v omega)} yields identity \eqref{div of vnablav},
which establishes \eqref{vorticity formulation lemma}$_{\mathrm{a}}$.  

Identity \eqref{vorticity formulation lemma}$_{\mathrm{b}}$ follows 
by integrating \eqref{magnetized Euler} in $x$ and noticing that $\mathrm{div}_{x}(v\otimes v)$ and $\nabla_{x}p$ are free derivatives. 

We move to show \eqref{vorticity formulation lemma}$_{\mathrm{c}}$. 
By Lemma \ref{Hodge decomposition lemma},  
the relation $\omega(t,\cdot)=\mathrm{curl}_{x}v(t,\cdot)$ implies that there is 
some time-dependent $\widetilde{\bar{v}}(t)\in \mathbb{R}^{2}$ 
and smooth scalar functions $\psi(t,x)$ and $\phi(t,x)$ such that
\begin{align*}
v(t,\cdot)=\widetilde{\bar{v}}(t)+\nabla_{x}\phi(t,x)+\nabla_{x}^{\perp}\psi(t,x), 
\quad \Delta_{x} \psi(t,x)=\omega, \quad \int_{\mathbb{T}^{2}}\psi(t,x)\ \dd x=0. 
\end{align*}
Taking the divergence in $x$ and using that  $\mathrm{div}_{x}v=\mathrm{div}_{x}\mathrm{\nabla}_{x}^{\perp}\psi=0$, 
we find that $\Delta_{x}\phi=0$ so that $\phi$ must be constant. It follows that 
\begin{align*}
v(t,x)=\widetilde{\bar{v}}(t)+\nabla_{x}^{\perp}\psi(t,x).    
\end{align*}
Finally, integrating in $x$ and using that $\int_{\mathbb{T}^{2}}\nabla_{x}^{\perp}\psi(t,x)\ \dd x=0$,
we conclude
\begin{align*}
\widetilde{\bar{v}}(t)=\int_{\mathbb{T}^{2}}v(t,x)\ \dd x=\bar{v}(t).     
\end{align*}

\smallskip
\textbf{2}. Conversely, suppose that $(\omega,v)$ solves
\eqref{vorticity formulation lemma}. 
Note that 
\begin{align*}
v^{\perp}=\bar{v}^{\perp}(t)-\nabla_{x}(\Delta_{x})^{-1}\omega.    
\end{align*}
By taking the divergence in $x$, we see that
$\omega=-\mathrm{div}_{x}v^{\perp}=\mathrm{curl}_{x}(v)$. 
Define
\[
F:=
\partial_tv
+\mathrm{div}_x(v\otimes v)
+Bv^\perp.
\]
Using \eqref{vorticity formulation lemma}$_{\mathrm{a}}$, we obtain
\[
\begin{aligned}
\mathrm{curl}_xF&= \partial_{t}\mathrm{curl}_{x}(v)+\mathrm{curl}_{x}\mathrm{div}_{x}(v\otimes v)+\mathrm{curl}_{x}(Bv^{\perp})
\\&=
\partial_t\omega
+\mathrm{curl}_x\bigl(vD_xv\bigr)
+\mathrm{curl}_x(Bv^\perp)\\
&=
\partial_t\omega
+\mathrm{div}_x(v\omega)
+\nabla_xB\cdot v=0.
\end{aligned}
\]
In the third equality, we used the identities: $\mathrm{curl}_{x}(vD_{x}v)=\mathrm{div}_{x}((\mathrm{curl}_{x}v)v)$ (equation \eqref{div of vnablav}) and 
\begin{align*}
\mathrm{curl}_{x}(Bv^{\perp})=\mathrm{div}_{x}(Bv)=\nabla_{x}B\cdot v.     
\end{align*}
Furthermore, we have
\[
\begin{aligned}
\int_{\mathbb T^2}F(t,x)\,\dd x
&=
\dot{\bar{v}}(t)
+\int_{\mathbb T^2}B(t,x)v^\perp(t,x)\,\dd x=0,
\end{aligned}
\]
because the integral of
$\operatorname{div}_x(v\otimes v)$ vanishes. 
It follows that $F(t,\cdot)$ is a curl free vector 
field with zero
mean and therefore, by Poincar\'e's lemma, there exists some smooth scalar function 
$p(t,\cdot)$ such that $F(t,x)=-\nabla_{x}p(t,x)$ so that
\begin{align}
\partial_{t}v+\mathrm{div}_{x}(v\otimes v)+Bv^{\perp}+\nabla_{x}p=0.\label{eq for velocity lemma}
\end{align}
By taking the divergence of \eqref{eq for velocity lemma}, we conclude 
\begin{align*}
p=(-\Delta_{x})^{-1}\big(\mathrm{div}_{x}\mathrm{div}_{x}(v\otimes v)
+\mathrm{div}_{x}(Bv^{\perp})\big).    
\end{align*}
\end{proof}

In view of  Lemma \ref{lemma3.2}, the well-posedness problem for \eqref{magnetized Euler} is reduced 
to the well-posedness problem for \eqref{vorticity formulation intro}, 
which is our next objective. Recall the following endpoint Calderon-Zygmund-type theorem.   

\begin{prop}[\cite{bahouri2011fourier},\,Proposition 7.7]\label{near boundedness}  
Let $s\in(0,1)$ and $\omega\in L^1(\mathbb T^2)\cap C^{0,s}(\mathbb T^2).$ 
Then there is some $C>0$ such that 
\begin{align*}
\big\Vert D_{x}\nabla_{x}^{\perp} V\ast \omega \big\Vert_{\infty}
\leq 
C\big(\Vert \omega\Vert_{1} 
+\Vert \omega\Vert_{\infty}(1+|\log\Vert \omega\Vert_{\infty}|\big)\log(e+\left\Vert \omega\right\Vert_{\infty}+\left\Vert \omega \right\Vert_{C^{0,s}})\big).
\end{align*} 
\end{prop}

\begin{thm}\label{Well posedness of EB}
Fix $T>0$. Suppose that $\omega^{\mathrm{in}}\in W^{1,\infty}(\mathbb{T}^{2})$ and
$\int_{\mathbb T^2}\omega^{\mathrm{in}}\,\dd x=0.$
Assume also that
$
B\in W^{1,\infty}(\mathbb{R}_{+};W^{2,\infty}(\mathbb{T}^{2})).
$
Then
\begin{enumerate}
   \item[{\rm(i)}]There exists a unique solution $(\omega,v)\in W^{1,\infty}([0,T]\times \mathbb{T}^{2})\times W^{1,\infty}([0,T]\times \mathbb{T}^{2})$ of the Cauchy problem \eqref{vorticity formulation intro} with initial data $\omega^{\mathrm{in}}$. 
Consequently, there exists a unique solution $v\in W^{1,\infty}([0,T]\times \mathbb{T}^{2})$ of 
the Cauchy problem \eqref{magnetized Euler} with initial data $v^{\mathrm{in}}=-\nabla_{x}^{\perp}V\ast \omega^{\mathrm{in}}$. 
 \item[{\rm(ii)}]For any $r\in [1,\infty)$,
 \begin{align}
D^{2}_{x}v\in L^{\infty}([0,T];L^{r}(\mathbb{T}^{2})), \qquad 
\partial_{t}D_{x}v\in L^{\infty}([0,T];L^{r}(\mathbb{T}^{2})). \label{useful est on vel}    
\end{align}
\end{enumerate}

\end{thm}
\begin{proof}
We divide the proof into three steps.

\smallskip
\textbf{1. }\textbf{Existence}.
Fix a standard mollifier $\chi_{\eta}$ on $\mathbb{T}^{2}$ and let $V_{\eta}:= \chi_{\eta}\ast V$. Consider a solution $\omega_{\eta}$  of the regularized Cauchy problem:  
\begin{align}
\begin{cases}
\partial_{t}\omega_{\eta}+\mathrm{div}_{x}(v_{\eta} (\omega_{\eta}+B))=0, \,\,\, \omega_{\eta}(0,\cdot)=\omega^{\mathrm{in}},\,\,\, v_{\eta}(0,\cdot)=-\nabla_{x}^{\perp}V_{\eta}\ast \omega^{\mathrm{in}}, \\[1mm]
\bar{v}_{\eta}(t)=-\int_{0}^{t}\int_{\mathbb{T}^{2}}B(\tau,x)v_{\eta}^{\perp}(\tau,x)\ \dd x \dd\tau,\\[1mm]
v_{\eta}(t,x)=\bar{v}_{\eta}(t)-\nabla_{x}^{\perp}V_{\eta}\ast \omega_{\eta}(t,x). 
\end{cases}
\label{regularized magnetized}
\end{align}
Define $\Omega_{\eta}:= \omega_{\eta}+B$ so that 
$\Omega_{\eta}$ is governed by the equation  
\begin{align}
\partial_{t}\Omega_{\eta}+v_{\eta}\cdot \nabla_{x}\Omega_{\eta}=\partial_{t}B. \label{eq for capital omega}   \end{align}
\textbf{Uniform in $\eta$ estimate on $\left\Vert \Omega_{\eta}\right\Vert_{W^{1,\infty}([0,T]\times \mathbb{T}^{2})}$ and $\left\Vert v_{\eta}\right\Vert_{W^{1,\infty}([0,T]\times \mathbb{T}^{2})}$.} We start by propagating the $W^{1,\infty}_{x}$--norm of $\Omega_{\eta}(t,\cdot)$ uniformly in $\eta$. 
Multiplying \eqref{eq for capital omega} by $\mathbf{s}(t,x)=\mathrm{sgn}(\Omega_{\eta}(t,x))$, we have
\begin{align}
\partial_{t}\left\vert \Omega_{\eta}\right\vert+v_{\eta}\cdot \nabla_{x}\left\vert \Omega_{\eta}\right\vert =\mathbf{s}\partial_{t}B. \label{eq for abs of Omega} \end{align}
Evaluating \eqref{eq for abs of Omega} at a maximum point $x=x^{\max}$ of $x\mapsto \left\vert \Omega_{\eta}\right\vert(t,x)$ and using $\nabla_{x}\left\vert \Omega_{\eta}\right\vert(t,x)=0,$ 
we obtain 
\begin{align*}
\frac{\dd}{\dd t}\left\Vert \Omega_{\eta}(t,\cdot)\right\Vert_{\infty}\leq \left\Vert \partial_{t}B\right\Vert_{\infty}.      
\end{align*}
Therefore, it follows that  
\begin{align}
\left\Vert \Omega_{\eta}(t,\cdot)\right\Vert_{\infty}\leq \left\Vert \omega^{\mathrm{in}}+B(0,\cdot)\right\Vert_{\infty}+T\left\Vert \partial_{t}B\right\Vert_{\infty}\leq \left\Vert \omega^{\mathrm{in}}\right\Vert_{\infty}+\left\Vert B\right\Vert_{\infty}+T\left\Vert \partial_{t}B\right\Vert_{\infty}:= A.           \label{Def of A}  
\end{align}
Moreover, taking the gradient in $x$ of \eqref{eq for capital omega} and multiplying by $\mathbf{s}'(t,x)=\frac{\nabla_{x}\Omega_{\eta}(t,x)}{\left\vert \nabla_{x}\Omega_{\eta}\right\vert(t,x)}$, we obtain 
\begin{align}
\partial_{t}\left\vert \nabla_{x}\Omega_{\eta}\right\vert+D_{x}v_{\eta}\nabla_{x}\Omega_{\eta}\cdot \mathbf{s}'+v_{\eta}\cdot \nabla_{x}\left\vert \nabla_{x}\Omega_{\eta}\right\vert=\nabla_{x}\partial_{t}B\cdot \mathbf{s}'.  \label{eq for abs nabla Omega}
\end{align}
Evaluating \eqref{eq for abs nabla Omega} at a point of maximum $y^{\max}=y$ of $x\mapsto \left\vert \nabla_{x}\Omega_{\eta}\right\vert(t,x)$, we see that 
\begin{align*}
\frac{\dd}{\dd t}\left\Vert \nabla_{x}\Omega_{\eta}(t,\cdot)\right\Vert_{\infty}+D_{x}v_{\eta}\nabla_{x}\Omega_{\eta}\cdot\mathbf{s}'(t,y)=\nabla_{x}\partial_{t}B\cdot \mathbf{s}'(t,y).      
\end{align*}
Therefore, it follows that  
\begin{align}
\frac{\dd}{\dd t}\left\Vert \nabla_{x}\Omega_{\eta}(t,\cdot)\right\Vert_{\infty}&\leq \left\Vert \nabla_{x}\partial_{t}B\right\Vert_{\infty}+\left\Vert D_{x}v_{\eta}(t,\cdot)\right\Vert_{\infty} \left\Vert \nabla_{x}\Omega_{\eta}(t,\cdot)\right\Vert_{\infty} \notag\\
&=\left\Vert \nabla_{x}\partial_{t}B\right\Vert_{\infty}+\big\Vert D_{x}\nabla_{x}^{\perp}V_{\eta}\ast(\Omega_{\eta}-B)(t,\cdot) \big\Vert_{\infty}\left\Vert \nabla_{x}\Omega_{\eta}(t,\cdot)\right\Vert_{\infty}. \label{est time der of nablaomega}         \end{align}
By Proposition \ref{near boundedness} and \eqref{Def of A}, we have 
\begin{align*}
&\big\Vert D_{x}\nabla_{x}^{\perp}V_{\eta}\ast (\Omega_{\eta}-B)(t,\cdot)\big\Vert_{\infty}\\
&\leq \big\Vert D_{x}\nabla_{x}^{\perp}V_{\eta}\ast \Omega_{\eta}(t,\cdot)\big\Vert_{\infty}
+\big\Vert D_{x}\nabla^{\perp}V_{\eta}\ast B(t,\cdot)\big\Vert_{\infty} \\
&\leq C\big(A+A(1+|\log(A)|)\log(e+A+\Vert \nabla_{x}\Omega_{\eta}(t,\cdot)\Vert_{\infty}  )\big)
+\Vert \nabla_{x}V\Vert_{1} \Vert D_{x}B\Vert_{\infty}. 
\end{align*}
Therefore, there exist some constant $\mathcal{C}>e$ such that 
\begin{align}
\big\Vert D_{x}\nabla_{x}^{\perp}V_{\eta}\ast (\Omega_{\eta}-B)(t,\cdot)\big\Vert_{\infty}
\leq \mathcal{C}\log(\mathcal{C}+\Vert \nabla_{x}\Omega_{\eta}(t,\cdot)\Vert_{\infty}). 
\label{est for hessian}  
\end{align}
Substituting \eqref{est for hessian} inside \eqref{est time der of nablaomega}, 
we deduce that 
there exist some $\mathcal{C}'\geq e$ such that 
\begin{align*}
\frac{\dd}{\dd t}\left\Vert \nabla_{x}\Omega_{\eta}(t,\cdot)\right\Vert_{\infty}\leq \log(\mathcal{C}'+\left\Vert \nabla_{x}\Omega_{\eta}(t,\cdot)\right\Vert_{\infty})
\big(\mathcal{C}'+\left\Vert \nabla_{x}\Omega_{\eta}(t,\cdot)\right\Vert_{\infty}\big).      
\end{align*}
Setting $S(t):= \mathcal{C}'+\left\Vert \nabla_{x}\Omega_{\eta}(t,\cdot)\right\Vert_{\infty}$, 
we infer the inequality: 
\begin{align*}
\frac{\dd}{\dd t}S(t)\leq S(t)\log(S(t))\,\,\Longrightarrow\,\, \frac{\dd}{\dd t}\log\log(S(t))\leq 1.
\end{align*}
Solving the above inequality yields 
\begin{align*}
S(t)\leq S(0)^{e^{t}}.     
\end{align*}
It follows that 
$\left\Vert \Omega_{\eta}(t,\cdot)\right\Vert_{W^{1,\infty}}$ is uniformly bounded in $\eta$, 
{\it i.e.}, 
\begin{align}
\left\Vert \Omega_{\eta}(t,\cdot)\right\Vert_{W^{1,\infty}}= O_{\eta}(1).  \label{W1infty of Omegaeta} \end{align}
Consequently, we can also bound the time-derivative of $\Omega_{\eta}$ uniformly in $\eta$. 
Differentiating in time $v_{\eta}=\bar{v}_{\eta}(t)+\nabla_{x}^{\perp}V_{\eta}\ast (\Omega_{\eta}-B)$,
we have
\begin{align*}
\frac{\dd}{\dd t}v_{\eta}(t,x)&=\dot{\bar{v}}_{\eta}(t)-\nabla_{x}^{\perp}V_{\eta}\ast \partial_{t}\omega_{\eta}(t,x)\\
&=-\int_{\mathbb{T}^{2}}B(t,x)v_{\eta}^{\perp}(t,x)\ \dd x+\nabla_{x}^{\perp}V_{\eta}\ast (v_{\eta}\cdot \nabla_{x}\Omega_{\eta})(t,x). 
\end{align*}
Therefore, in view of \eqref{W1infty of Omegaeta}, we deduce the estimate: 
\begin{align*}
\frac{\dd}{\dd t}\left\Vert v_{\eta}(t,\cdot)\right\Vert_{\infty}\leq \left\Vert B\right\Vert_{\infty} \left\Vert v_{\eta}(t,\cdot)\right\Vert_{\infty}+\left\Vert \nabla_{x}V\right\Vert_{1} \left\Vert \nabla_{x}\Omega_{\eta}(t,\cdot)\right\Vert_{\infty} \left\Vert v_{\eta}(t,\cdot)\right\Vert_{\infty}\leq O_{\eta}(1)\left\Vert v_{\eta}(t,\cdot)\right\Vert_{\infty}.          
\end{align*}
Thus, it follows by Gr\"onwall's lemma that 
\begin{align}
\left\Vert v_{\eta}\right\Vert_{\infty}=O_{\eta}(1). \label{veta est} \end{align}
By virtue of  \eqref{eq for capital omega} and \eqref{veta est}, we have
\begin{align}
\left\Vert \partial_{t}\Omega_{\eta}\right\Vert_{\infty}\leq \left\Vert \partial_{t} B\right\Vert_{\infty}+\left\Vert v_{\eta}\right\Vert_{\infty}\left\Vert \nabla_{x}\Omega_{\eta}\right\Vert_{\infty}\leq O_{\eta}(1)\,\,\Longrightarrow\,\,
\left\Vert \partial_{t}\Omega_{\eta}\right\Vert_{\infty}=O_{\eta}(1).   \label{time der of capital omega}      \end{align}
Now it follows directly from \eqref{W1infty of Omegaeta}--\eqref{time der of capital omega}, 
by observing that 
\begin{align}
\left\Vert \partial_{t}\omega_{\eta}\right\Vert_{\infty}\leq \left\Vert \partial_{t}\Omega_{\eta}\right\Vert_{\infty}+\left\Vert \partial_{t}B\right\Vert_{\infty}=O_{\eta}(1)\ \mbox{and}\ \left\Vert \nabla_{x}\omega_{\eta}\right\Vert_{\infty}\leq \left\Vert \nabla_{x}\Omega_{\eta}\right\Vert_{\infty}+\left\Vert \nabla_{x}B\right\Vert_{\infty}=O_{\eta}(1),  
\label{x der of omegaeta}        
\end{align}
that $\left\Vert \omega_{\eta} \right\Vert_{W^{1,\infty}([0,T]\times \mathbb{T}^{2})}$ 
is uniformly bounded in $\eta$. 

\smallskip
\textbf{Uniform in $\eta$ estimate on $\left\Vert v_{\eta}\right\Vert_{W^{1,\infty}([0,T]\times \mathbb{T}^{2})}$.} Taking the derivative in $x_{j}$, $j=1,2$, in the equation for $v_{\eta}$,
we have
\begin{align*}
\partial_{x_{j}}v_{\eta}=-\nabla_{x}^{\perp}V_{\eta}\ast \partial_{x_{j}}\omega_{\eta}    
\end{align*} 
so that, in view of \eqref{x der of omegaeta}, 
\begin{align*}
\left\Vert \partial_{x_{j}}v_{\eta}\right\Vert_{\infty}\leq \left\Vert\nabla_xV_\eta\right\Vert_1\left\Vert \nabla_{x}\omega_{\eta}\right\Vert_{\infty}\leq O_{\eta}(1)\Rightarrow\left\Vert D_{x}v_{\eta}\right\Vert_{\infty}=O_{\eta}(1).        
\end{align*}
In addition, we have 
\begin{align*}
\partial_{t}v_{\eta}=-\int_{\mathbb{T}^{2}}B(t,x)v^{\perp}_{\eta}(t,x)\ \dd x-\nabla_{x}^{\perp}V_{\eta}\ast \partial_{t}\omega_{\eta}   \end{align*} 
so that, by \eqref{veta est} and \eqref{x der of omegaeta},
\begin{align*}
\left\Vert \partial_{t}v_{\eta}\right\Vert_{\infty}\leq \left\Vert B\right\Vert_{\infty}\left\Vert v_{\eta}\right\Vert_{\infty}+\left\Vert\nabla_xV_\eta\right\Vert_1\left\Vert \partial_{t}\omega_{\eta}\right\Vert_{\infty}=O_{\eta}(1)\Rightarrow \left\Vert \partial_{t}v_{\eta}\right\Vert_{\infty}=O_{\eta}(1).  \end{align*}
This proves that $\left\Vert v_{\eta}\right\Vert_{W^{1,\infty}([0,T]\times \mathbb{T}^{2})}=O_{\eta}(1)$. 
So we have proved that 
\begin{align*}
 \left\Vert \omega_{\eta}\right\Vert_{W^{1,\infty}([0,T]\times \mathbb{T}^{2})}=O_{\eta}(1)\ \mbox{and}\  \left\Vert v_{\eta}\right\Vert_{W^{1,\infty}([0,T]\times \mathbb{T}^{2})}=O_{\eta}(1). 
 \end{align*}
 After extracting a subsequence, the Banach--Alaoglu Theorem gives weak-* convergence 
 of the first derivatives in $L^\infty$, while the uniform Lipschitz bounds 
 and the Arzela--Ascoli theorem give $(\omega_\eta,v_\eta)\to(\omega,v)$ uniformly on 
 $[0,T]\times\mathbb T^2$. 
 The uniform convergence permits passage to the nonlinear transport term; 
 the convergence $\nabla_xV_\eta\to\nabla_xV$ in $L^1(\mathbb T^2)$ identifies $v=\bar{v}-\nabla_x^\perp V*\omega$. 
 Therefore, $(\omega,v)$ solves \eqref{vorticity formulation intro}.

 \smallskip
\textbf{2. Uniqueness}. To achieve this, we apply a stability argument 
with respect to the $L^{2}$--norm of $(\omega,v)$. 
Given two solutions $(\omega_{1},v_{1})$ and $(\omega_{2},v_{2})$ 
to \eqref{vorticity formulation intro} 
with initial data $\omega_{1}^{\mathrm{in}}$ and $\omega_{2}^{\mathrm{in}}$, 
respectively,  the equation for the difference $\omega_{1}-\omega_{2}$ reads 
\begin{align*}
\partial_{t}(\omega_{1}-\omega_{2})+v_{1}\cdot \nabla_{x}(\omega_{1}+B)-v_{2}\cdot \nabla_{x}(\omega_{2}+B)=0,      
\end{align*}
which is recast as 
\begin{align}
\partial_{t}(\omega_{1}-\omega_{2})+(v_{1}-v_{2})\cdot \nabla_{x}(\omega_{1}+B)+v_{2}\cdot \nabla_{x}(\omega_{1}-\omega_{2})=0.\label{eq for difference of omega}    
\end{align}
Multiplying \eqref{eq for difference of omega} by $\omega_{1}-\omega_{2}$ and integrating the resulting identity, we have 
\begin{align}
\frac{1}{2}\frac{\dd}{\dd t}\left\Vert (\omega_{1}-\omega_{2})(t,\cdot)\right\Vert_{2}^{2}=&-\int_{\mathbb{T}^{2}}(v_{1}-v_{2})\cdot \nabla_{x}(\omega_{1}+B)(\omega_{1}-\omega_{2})(t,x)\ \dd x \notag\\
&-\frac{1}{2}\int_{\mathbb{T}^{2}}v_{2}\cdot \nabla_{x}\left\vert \omega_{1}-\omega_{2}\right\vert^{2}(t,x)\ \dd x. \label{time derivative of L2 norm of difference}  \end{align}
The first integral on the right-hand side of \eqref{time derivative of L2 norm of difference} is 
bounded by 
\begin{align*}
\frac{1}{2}\big(\Vert \nabla_{x}\omega_{1}\Vert_{\infty}
+\Vert \nabla_{x}B\Vert_{\infty}\big)
\big(\Vert (v_{1}-v_{2})(t,\cdot)\Vert_{2}^{2}+\Vert (\omega_{1}-\omega_{2})(t,\cdot)\Vert_{2}^{2}\big).    
\end{align*}
The second integral in \eqref{time derivative of L2 norm of difference} vanishes by integration by parts and using that $\mathrm{div}_{x}v_{2}=0$. Therefore, we conclude 
\begin{align}
\frac{\dd}{\dd t}\left\Vert (\omega_{1}-\omega_{2})(t,\cdot)\right\Vert_{2}^{2}
\leq \big(\Vert \nabla_{x}\omega_{1}\Vert_{\infty}
  +\Vert \nabla_{x}B\Vert_{\infty}\big)
  \big(\Vert (v_{1}-v_{2})(t,\cdot)\Vert_{2}^{2}+\Vert (\omega_{1}-\omega_{2})(t,\cdot)\Vert_{2}^{2}\big). \label{time der of L2 norm of difference} 
\end{align}
Moreover,  the  equation for the difference  $v_{1}-v_{2}$ reads  
\begin{align}
\partial_{t}(v_{1}-v_{2})=\int_{\mathbb{T}^{2}}(v_{2}^{\perp}-v_{1}^{\perp})(t,x)B(t,x)\ \dd x-\nabla_{x}^{\perp}V\ast (\partial_{t}\omega_{1}-\partial_{t}\omega_{2}). \label{eq for v1-v2}    \end{align}
Multiplying \eqref{eq for v1-v2} by $v_{1}-v_{2}$, we obtain
\begin{align}
\frac{1}{2}\frac{\dd}{\dd t}\left\Vert (v_{1}-v_{2})(t,\cdot)\right\Vert_{2}^{2}&=-\int_{\mathbb{T}^{2}\times \mathbb{T}^{2}}(v_{1}^{\perp}-v_{2}^{\perp})(t,x)\cdot (v_{1}-v_{2})(t,y)B(t,x)\ \dd x\dd y \notag\\
&\quad-\int_{\mathbb{T}^{2}}\nabla_{x}^{\perp}V\ast \partial_{t}(\omega_{1}-\omega_{2})(t,x)\cdot (v_{1}-v_{2})(t,x)\ \dd x:= I+J. \label{time der of diff of v} 
\end{align} 
First, we have the bound: 
\begin{align}
I\leq \left\Vert B\right\Vert_{\infty}\left\Vert (v_{1}-v_{2})(t,\cdot)\right\Vert_{2}^{2}. \label{est I} 
\end{align}
As for $J$, we recast it as 
\begin{align}
J
&=-\int_{\mathbb T^2}\nabla_x^\perp V*\partial_t(\omega_1-\omega_2)\cdot(v_1-v_2)\,\dd x\notag\\
&=\int_{\mathbb T^2}V*\partial_t(\omega_1-\omega_2)(\omega_1-\omega_2)\,\dd x\notag\\
&=-\int_{\mathbb T^2}V*\big((v_1-v_2)\cdot\nabla_x(\omega_1+B)\big)(\omega_1-\omega_2)\,\dd x\notag\\
&\quad-\int_{\mathbb T^2}V*\operatorname{div}_x\big(v_2(\omega_1-\omega_2)\big)(\omega_1-\omega_2)\,\dd x
=:J_1+J_2.\label{second integral}
\end{align}
Substituting \eqref{eq for difference of omega} inside \eqref{second integral}, we have 
\begin{align*}
J=&\int_{\mathbb{T}^{2}}(v_{1}-v_{2})(t,x)\cdot \nabla_{x}(\omega_{1}+B)(t,x)V\ast (\omega_{1}-\omega_{2})(t,x)\ \dd x\\
&+\int_{\mathbb{T}^{2}}V\ast(\mathrm{div}_{x}(v_{2}(\omega_{1}-\omega_{2})))(t,x)(\omega_{1}-\omega_{2})(t,x)\ \dd x:= J_{1}+J_{2}.  
\end{align*}
To bound the term $J_{1}$, we observe 
\begin{align}
J_{1}\leq \frac{1}{2}\big(\Vert \nabla_{x}\omega_{1}\Vert_{\infty}
+\Vert \nabla_{x}B\Vert_{\infty}\big)\big(\Vert (v_{1}-v_{2})(t,\cdot)\Vert_{2}^{2}
+\Vert V\Vert_{1}^{2} \Vert (\omega_{1}-\omega_{2})(t,\cdot)\Vert_{2}^{2}\big). \label{J1ES}   
\end{align}
In addition, we have 
\begin{align}
J_{2}&\leq \left\Vert (\omega_{1}-\omega_{2})(t,\cdot)\right\Vert_{2}^{2}+\left\Vert \nabla_{x}V\ast (v_{2}(\omega_{1}-\omega_{2}))(t,\cdot)\right\Vert_{2}^{2}\notag\\
&\leq \big(1+\Vert v_{2}\Vert_{\infty}^{2} \Vert \nabla_{x}V\Vert_{1}^{2}\big)
\left\Vert (\omega_{1}-\omega_{2})(t,\cdot)\right\Vert_{2}^{2}.  \label{J2ES}      
\end{align}
In view of \eqref{J1ES}--\eqref{J2ES}, we infer that there is some constant 
$C>0$ such that 
\begin{align}
J\leq C\big(\Vert (v_{1}-v_{2})(t,\cdot)\Vert_{2}^{2}+\Vert (\omega_{1}-\omega_{2})(t,\cdot)\Vert_{2}^{2}\big).  \label{EstJ}    
\end{align}
Combining \eqref{est I} with \eqref{EstJ} and owing to \eqref{time der of diff of v}, 
we conclude 
\begin{align}
\frac{\dd}{\dd t}\left\Vert (v_{1}-v_{2})(t,\cdot)\right\Vert_{2}^{2}\lesssim \left\Vert (v_{1}-v_{2})(t,\cdot)\right\Vert_{2}^{2}+\left\Vert (\omega_{1}-\omega_{2})(t,\cdot)\right\Vert_{2}^{2}. \label{L2 norm diff of v}
\end{align}
Combining \eqref{time der of L2 norm of difference} with \eqref{L2 norm diff of v}, we conclude 
\begin{align*}
\frac{\dd}{\dd t}\big(\left\Vert (\omega_{1}-\omega_{2})(t,\cdot)\right\Vert_{2}^{2}+\left\Vert (v_{1}-v_{2})(t,\cdot)\right\Vert_{2}^{2}  \big)\leq C'\big(\left\Vert (\omega_{1}-\omega_{2})(t,\cdot)\right\Vert_{2}^{2}+\left\Vert (v_{1}-v_{2})(t,\cdot)\right\Vert_{2}^{2}\big).     
\end{align*}
Hence, by Gr\"onwall's lemma, we obtain 
\begin{align*}
\left\Vert (\omega_{1}-\omega_{2})(t,\cdot)\right\Vert_{2}^{2}+\left\Vert (v_{1}-v_{2})(t,\cdot) \right\Vert_{2}^{2}
\leq e^{C't}\big(\Vert \omega_{1}^{\mathrm{in}}-\omega_{2}^{\mathrm{in}}\Vert_{2}^{2}
+\Vert v_{1}^{\mathrm{in}}-v_{2}^{\mathrm{in}}\Vert_{2}^{2} \big),
\end{align*}
from which the uniqueness follows immediately.

The well-posedness of \eqref{magnetized Euler} is now a direct consequence of Lemma \ref{vorticity formulation}: If $(\omega,v)$ is a solution to \eqref{vorticity formulation intro}, 
then $(v,p)$ is a solution to \eqref{magnetized Euler}, which gives the existence. 
For the uniqueness, suppose that  $(v_{1},p_{1}),(v_{2},p_{2})$  are two solutions 
to \eqref{magnetized Euler} with associated vorticity fields $(\omega_{1},\omega_{2})$ respectively 
and with the same initial data $v^{\mathrm{in}}$. 
Then $(v_{1},\omega_{1}),(v_{2},\omega_{2})$ are two solutions to \eqref{vorticity formulation intro}. Therefore, by uniqueness, it follows that $v_{1}\equiv v_{2}$, which in turn implies that 
\begin{align*}
\nabla_{x}(p_{1}-p_{2})=0.   
\end{align*}
Since 
\begin{align*}
\int_{\mathbb{T}^{2}}p_{1}(t,x)\ \dd x=\int_{\mathbb{T}^{2}}p_{2}(t,x)\ \dd x=0,    
\end{align*}
it follows that $p_{1}\equiv p_{2}$, so the uniqueness for \eqref{magnetized Euler} holds.

\medskip
\textbf{3. The estimate \eqref{useful est on vel}. } Fix $1<r<\infty$. By  the Calderon-Zygmund theorem
for any $1\leq j,k\leq 2$, we have
\begin{align}
\big\Vert (\partial_{x_{j}x_{k}}v)(t,\cdot)\big\Vert_{r} 
=\big\Vert (\partial_{x_{j}x_{k}}\nabla_{x}^{\perp}V\ast \omega)(t,\cdot)\big\Vert_{r}
\leq C\big\Vert \nabla_{x}\omega(t,\cdot)\big\Vert_{r}
\lesssim \big\Vert \omega\big\Vert_{W^{1,\infty}([0,T]\times \mathbb{T}^{2})}
\lesssim 1,  \label{Hessian of v Lr} 
\end{align}
because $\omega\in W^{1,\infty}([0,T]\times \mathbb{T}^{2})$. 
This proves that $D^{2}_{x}v\in L^{\infty}([0,T];L^{r}(\mathbb{T}^{2}))$.   
In addition, for $k=1,2$, we have 
\begin{align*}
\partial_{t}\partial_{x_{k}}v=-\partial_{x_{k}}\big((v\cdot\nabla_x)v+Bv^{\perp}+\nabla_{x}p\big).    
\end{align*}
Therefore, we obtain the estimate:  
\begin{align*}
\left\Vert (\partial_{t}\partial_{x_{k}}v)(t,\cdot)\right\Vert_{r}\lesssim
\left\Vert v\right\Vert_\infty\left\Vert D_x^2v\right\Vert_r
+\left\Vert D_xv\right\Vert_\infty\left\Vert D_xv\right\Vert_r
+\left\Vert B\right\Vert_{W^{1,\infty}}\left\Vert v\right\Vert_{W^{1,r}}
+\left\Vert \nabla_x^2p\right\Vert_r.   
\end{align*}
By the periodic Calder\'on--Zygmund estimate and $\int_{\mathbb T^2}p\,\dd x=0$, we have 
\begin{align*}
\left\Vert \nabla^{2}_{x}p(t,\cdot)\right\Vert_{r}\lesssim \left\Vert \Delta_{x} p(t,\cdot)\right\Vert_{r}=\left\Vert \mathfrak{U}(t,\cdot)\right\Vert_{r}.        
\end{align*}
Together with  \eqref{Hessian of v Lr}, we conclude 
\begin{align*}
\left\Vert \partial_{t}\partial_{x_{k}}v(t,\cdot)\right\Vert_{r}
&\lesssim \left\Vert v\right\Vert_{\infty}
\left\Vert D_{x}^{2}v\right\Vert_{L^{\infty}_{t}L^{r}_{x}}
+\left\Vert D_xv\right\Vert_{\infty}
\left\Vert D_xv\right\Vert_{L^\infty_tL^r_x}\\
&\quad+\left\Vert B\right\Vert_{L^{\infty}_{t}W^{1,\infty}_{x}}
\left\Vert v\right\Vert_{L^{\infty}_{t}W^{1,r}_{x}}
+\left\Vert \mathfrak{U}\right\Vert_{L^{\infty}_{t}L^{r}_{x}}\lesssim 1.
\end{align*}
This proves that $\partial_{t}D_{x}v\in L^{\infty}([0,T];L^{r}(\mathbb{T}^{2}))$ for $1<r<\infty$. 
The case $r=1$ follows from the case $r=2$ and the embedding 
$L^2(\mathbb T^2)\hookrightarrow L^1(\mathbb T^2)$.
\end{proof}

The following simple estimates will be used in the context of the supercritical and quasi-neutral limits. 
\begin{lem}\label{basic estimates}
Let the assumptions of {\rm Theorem \ref{Well posedness of EB}} hold.  Assume 
that $B\in W^{1,\infty}(\mathbb{R}_{+};W^{2,\infty}(\mathbb{T}^{2}))$. 
Let $(v,p)$ be the unique solution to \eqref{magnetized Euler}. 
Then the following bounds hold{\rm:} 
\begin{itemize}
   \item[{\rm(i)}]
    $\left\Vert \nabla_{x}p\right\Vert_{L^{\infty}_{t}L^{2}_{x}}\lesssim
    \left\Vert D_{x}v\right\Vert_{\infty} \left\Vert D_{x}v\right\Vert_{L^{\infty}_{t}L^{2}_{x}}
    +\left\Vert \nabla_{x}B\right\Vert_{\infty}\left\Vert v\right\Vert_{L^\infty_tL^2_x}
    +\left\Vert B\right\Vert_{\infty}\left\Vert D_xv\right\Vert_{L^\infty_tL^2_x}$.
    \item[{\rm(ii)}] $\left\Vert \nabla_{x}^{2}p\right\Vert_{L^{\infty}_{t}L^2_{x}}\lesssim
    \left\Vert D_{x}v\right\Vert_{\infty}\left\Vert D_{x}v\right\Vert_{L^{\infty}_{t}L^{2}_{x}}
    +\left\Vert \nabla_{x}B\right\Vert_{\infty}\left\Vert v\right\Vert_{L^\infty_tL^2_x}
    +\left\Vert B\right\Vert_{\infty}\left\Vert D_xv\right\Vert_{L^\infty_tL^2_x}$.
   \item[{\rm(iii)}] 
$\left\Vert \nabla_{x}\mathrm{div}_{x}\Delta^{-1}_{x}(v\mathfrak{U})\right\Vert_{\infty}
   \lesssim\left\Vert v\right\Vert_{\infty}\big(\Vert D_{x}v\Vert_{\infty}\Vert D_{x}^{2}v\Vert_{L^{\infty}_{t}L^3_{x}}+\Vert \nabla_{x}^{2}B\Vert_{\infty}\Vert v\Vert_{\infty}\big)$
   
   $\qquad\qquad \qquad\qquad\quad+\left\Vert v\right\Vert_{\infty} \big(\Vert D_{x}v\Vert_{\infty}\Vert \nabla_{x}B\Vert_{\infty}+\Vert B\Vert_{\infty}\Vert D_{x}^{2}v\Vert_{L^\infty_{t}L^{3}_{x}}\big).$     
 \item[{\rm(iv)}] For some $s\in (0,1)$, there is a constant $C=C(\left\Vert B\right\Vert_{W^{1,\infty}_{t}C^{0,s}_{x}},\left\Vert v\right\Vert_{W^{1,\infty}_{t}C^{0,s}_{x}})>0$ such that 
$$
\left\Vert \partial_{t}\nabla_{x}p\right\Vert_{\infty}\lesssim \left\Vert \partial_{t}D_{x}v:D_{x}v\right\Vert_{L^{\infty}_{t}L^{3}_{x}}+C.
$$ 
\end{itemize}    
\end{lem}

\begin{proof}
Recall that 
\begin{align*}
\mathfrak{U}=\mathrm{div}_{x}\mathrm{div}_{x}(v\otimes v)+\nabla_{x}B \cdot v^{\perp}+B\mathrm{div}_{x}(v^{\perp}),
\end{align*}
so that
\begin{align*}
\nabla_{x}\mathfrak{U}=\nabla_{x}\mathrm{div}_{x}\mathrm{div}_{x}(v\otimes v)+\nabla_{x}^{2}Bv^{\perp}+D_{x}v^{\perp}\nabla_{x}B +\nabla_{x}B\mathrm{div}_{x}(v^{\perp})+B\nabla_{x}\mathrm{div}_{x}(v^{\perp}).    
\end{align*}

Since $-\Delta_{x} p=\mathfrak{U}$, it follows that  $\nabla_{x}p=\nabla_{x}V\ast \mathfrak{U}$ and consequently 
\begin{align*}
\left\Vert \nabla_{x}p(t,\cdot)\right\Vert_{2}&=\left\Vert \nabla_{x}V\ast \mathfrak{U}(t,\cdot)\right\Vert_{2}\leq \left\Vert \nabla_{x}V\right\Vert_{1} \left\Vert \mathfrak{U}(t,\cdot) \right\Vert_{2}\\
&\lesssim \left\Vert D_{x}v\right\Vert_{\infty} \left\Vert D_{x}v\right\Vert_{L^{\infty}_{t}L^{2}_{x}}+\left\Vert \nabla_{x}B\right\Vert_{\infty}\left\Vert v\right\Vert_{L^{\infty}_{t}L^2_{x}}+\left\Vert B\right\Vert_{\infty}\left\Vert D_{x}v\right\Vert_{L^{\infty}_{t}L^{2}_{x}}.  
\end{align*}

By the Calderon-Zygmund theorem, we have 
\begin{align*}
 \left\Vert \nabla_{x}^{2}p(t,\cdot)\right\Vert_{2}&=\left\Vert \nabla _{x}
 ^{2}V\ast \mathfrak{U}(t,\cdot)\right\Vert_{2}\lesssim \left\Vert \mathfrak{U}(t,\cdot)\right\Vert_{2}\\
 &\lesssim  \left\Vert D_{x}v\right\Vert_{\infty}\left\Vert D_{x}v\right\Vert_{L^{\infty}_{t}L^2_{x}}+\left\Vert \nabla_{x}B\right\Vert_{\infty}\left\Vert v\right\Vert_{2}+\left\Vert B\right\Vert_{\infty}\left\Vert D_{x}v\right\Vert_{L^{\infty}_{t}L^{2}_{x}} .    \end{align*}
 
We also see that
\begin{align*}
&\left\Vert \nabla_{x}\mathrm{div}_{x}\Delta^{-1}_{x}(v\mathfrak{U})(t,\cdot)\right\Vert_{\infty}=\left\Vert \nabla_{x}\Delta^{-1}_{x}(v\cdot \nabla_{x}\mathfrak{U})(t,\cdot)\right\Vert_{\infty}\\
&= \left\Vert \nabla_{x}V\ast (v\cdot \nabla_{x}\mathfrak{U})(t,\cdot)\right\Vert_{\infty}\leq \left\Vert \nabla_{x}V\right\Vert_{\frac{3}{2}}\left\Vert (v\cdot \nabla_{x}\mathfrak{U})(t,\cdot)\right\Vert_{3}\\
&\lesssim \left\Vert v\right\Vert_{\infty}\big(\left\Vert D_{x}v\right\Vert_{\infty}\left\Vert D_{x}^{2}v\right\Vert_{L^{\infty}_{t}L^3_{x}}+\left\Vert \nabla_{x}^{2}B\right\Vert_{\infty}\left\Vert v\right\Vert_{\infty}\big)\\
&\quad+\left\Vert v\right\Vert_{\infty}\big(\left\Vert D_{x}v\right\Vert_{\infty}\left\Vert \nabla_{x}B\right\Vert_{\infty}+\left\Vert B\right\Vert_{\infty}\left\Vert D_{x}^{2}
v\right\Vert_{L^\infty_{t}L^{3}_{x}}\big),          
\end{align*}
and
$\partial_{t}\nabla_{x}p=\nabla_{x}V\ast \partial_{t}\mathfrak{U}=\nabla_{x}V\ast \partial_{t}\mathrm{div}_{x}\mathrm{div}_{x}(v\otimes v)+\nabla_{x}V\ast \mathrm{div}_{x}(\partial_{t}Bv^{\perp}+B\partial_{t}v^{\perp})$.

In addition, the following identity holds: 
\begin{align*}
\partial_{t}\mathrm{div}_{x}\mathrm{div}_{x}(v\otimes v)=\sum_{i,j}\partial_{t}\partial_{x_{i}}v^{j}\partial_{x_{j}}v^{i}+\sum_{i,j}\partial_{x_{i}}v^{j}\partial_{t}\partial_{x_{j}}v^{i}=2\partial_{t}D_{x}v:D_{x}v.   
\end{align*}
Consequently, we obtain  
\begin{align}
\left\Vert \nabla_{x}V\ast \partial_{t}\mathrm{div}_{x}\mathrm{div}_{x}(v\otimes v)(t,\cdot)\right\Vert_{\infty}&=2\left\Vert \nabla_{x}V\ast (\partial_{t}D_{x}v:D_{x}v)(t,\cdot)\right\Vert_{\infty} \notag\\
&\leq 2\left\Vert \nabla_{x}V\right\Vert_{\frac{3}{2}}\left\Vert \partial_{t}D_{x}v:D_{x}v\right\Vert_{L^{\infty}_{t}L^{3}_{x}}. \label{first est iv}       
\end{align}
Moreover, we have 
\begin{align*}
\big\Vert \nabla_{x}V\ast\mathrm{div}_{x}(\partial_{t}Bv^{\perp}+B\partial_{t}v^{\perp})(t,\cdot)\big\Vert_{\infty}
&\leq \big\Vert \nabla_{x}^{2}V\ast (\partial_{t}Bv^{\perp}+B\partial_{t}v^{\perp})(t,\cdot)\big\Vert_{\infty}\\
&\leq 
\left\Vert \nabla_{x}^{2}V\ast (\partial_{t}B v)(t,\cdot)\right\Vert_{\infty}+\left\Vert \nabla_{x}^{2}V\ast (B\partial_{t}v)(t,\cdot)\right\Vert_{\infty}.   
\end{align*}
Choose $r>2$ in \eqref{useful est on vel}.  Morrey's embedding gives
$\partial_t v\in W^{1,r}(\mathbb T^2)\hookrightarrow C^{0,s}(\mathbb T^2)$
with $s=1-2/r$; the regularity of $B$ gives the same H\"{o}lder bound for the other terms.
By Proposition \ref{near boundedness}, we see that
\begin{align*}
\left\Vert \nabla_{x}^{2}V\ast (\partial_{t}B v)(t,\cdot)\right\Vert_{\infty}\leq C_{1}\,\qquad
\left\Vert \nabla_{x}^{2}V\ast (B\partial_{t}v)(t,\cdot)\right\Vert_{\infty}\leq C_{2}.       
\end{align*}
where $C_{1}=C_{1}(\left\Vert \partial_{t}B\right\Vert_{L^{\infty}_{t}C^{0,s}_{x}},\left\Vert v\right\Vert_{L^{\infty}_{t}C^{0,s}_{x}} )$ and $C_{2}=C_{2}(\left\Vert B\right\Vert_{L^{\infty}_{t}C^{0,s}},\left\Vert \partial_{t}v\right\Vert_{L^{\infty}_{t}C^{0,s}_{x}})$. Together with \eqref{first est iv}, this concludes the final estimate.  
\end{proof}

\section{The Quasi-Neutral Limit}\label{section quasi neutral}
We start by reviewing the well-posedness theory of the underlying equations. 
The well-posedness of \eqref{magnetized Euler} has been addressed in Section 3.
The well-posedness of \eqref{magnetized vp}
follows by 
a modification of 
the work 
by Lions-Perthame \cite{lions1991propagation}. 
Namely, we have the following theorem (see Theorem 1.1 in \cite{porat2025propagation} for more details):
\begin{thm}
Let $f^{\mathrm{in}}\in W^{1,\infty}(\mathbb{T}^{2}\times \mathbb{R}^{2})$
and satisfy
\begin{align*}
\int_{\mathbb{T}^{2}\times \mathbb{R}^{2}}\left\vert \xi\right\vert^{k} f^{\mathrm{in}} (x,\xi) \ \dd x\dd \xi<\infty\qquad\mbox{for some}\ k\ \mbox{sufficiently large}.  
\end{align*}
Then there exists a unique solution $f\in W^{1,\infty}([0,T]\times \mathbb{T}^{2}\times \mathbb{R}^{2})$ to \eqref{magnetized vp} with initial data $f^{\mathrm{in}}$. 
\end{thm}

Next, we prove the conservation of total energy. 

\begin{lem} \label{time der of total energy}
Let $f_{\varepsilon}\in W^{1,\infty}([0,T]\times \mathbb{T}^{2}\times \mathbb{R}^{2})$ be a non-negative solution, with finite second velocity moment and sufficient decay as \(\lvert\xi\rvert\to\infty\), to  \eqref{magnetized vp with quasineutral scaling}, 
and let the total energy be defined by 
\begin{align*}
\mathcal{F}_{\varepsilon}(t):= \frac{1}{2}\int_{\mathbb{T}^{2}\times \mathbb{R}^{2}}\left\vert \xi\right\vert^{2}f_{\varepsilon}(t,x,\xi)\ \dd x\dd \xi
+\frac{\varepsilon}{2}\int_{\mathbb{T}^{2}}\left\vert \nabla_{x} 
\Phi_{\varepsilon}\right\vert^{2}(t,x)\ \dd x.    
\end{align*}
Then
\begin{align*}
\frac{\dd}{\dd t}\mathcal{\mathcal{F}}_{\varepsilon}(t)=0 
\qquad \mbox{for all}\ t\in [0,T],   
\end{align*}
that is,
\begin{align*}
\mathcal{F}_{\varepsilon}(t)=\mathcal{F}_{\varepsilon}(0)
\qquad \mbox{for all}\ t\in [0,T].     
\end{align*}
\end{lem}
\begin{proof}
We compute the time derivative of $\mathcal{F}_{\varepsilon}(t)$: 
\begin{align*}
\frac{\dd}{\dd t}\mathcal{F}_{\varepsilon}(t)=&-\frac{1}{2}\int_{\mathbb{T}^{2}\times \mathbb{R}^{2}}\left\vert \xi\right\vert^{2}(\xi\cdot \nabla_{x}f_{\varepsilon}-(\nabla_{x}\Phi_{\varepsilon}+B\xi^{\perp})\cdot \nabla_{\xi}f_{\varepsilon})  \ \dd x\dd\xi
\\&+\varepsilon\int_{\mathbb{T}^{2}} \partial_{t}\nabla_{x} \Phi_{\varepsilon}\cdot \nabla_{x}\Phi_{\varepsilon}\ \dd x=:T_{1}+T_{2}. 
\end{align*}
Integrating by parts and observing that 
\begin{align*}
\mathrm{div}_{\xi}(\left\vert \xi\right\vert^{2}\xi^{\perp})=2\xi\cdot \xi^{\perp}+\left\vert \xi\right\vert^{2}\mathrm{div}_{\xi}(\xi^{\perp})=0,     
\end{align*}
we obtain 
\begin{align}
T_{1}=-\int_{\mathbb{T}^{2}\times \mathbb{R}^{2}} \xi\cdot (\nabla_{x}\Phi_{\varepsilon}+B\xi^{\perp})f_{\varepsilon}\ \dd x\dd \xi=-\int_{\mathbb{T}^{2}} J_{\varepsilon}\cdot \nabla_{x}\Phi_{\varepsilon}\ \dd x. \label{T1 calculation}  
\end{align}
Furthermore, by \eqref{density}, we have  
\begin{align*}
\partial_{t}\rho_{\varepsilon}+\mathrm{div}_{x}J_{\varepsilon} =0,    
\end{align*}
and hence 
\begin{align}
T_{2}=-\varepsilon\int_{\mathbb{T}^{2}} \partial_{t}\Delta_{x} \Phi_{\varepsilon}\Phi_{\varepsilon}\ \dd x=\int_{\mathbb{T}^{2}} \partial_{t}\rho_{\varepsilon}\Phi_{\varepsilon}\ \dd x=-\int_{\mathbb{T}^{2}} \mathrm{div}_{x}(J_{\varepsilon})\Phi_{\varepsilon}\ \dd x=\int_{\mathbb{T}^{2}} J_{\varepsilon}\cdot \nabla \Phi_{\varepsilon}\ \dd x. \label{T2 calculation}   
\end{align}
Therefore, gathering \eqref{T1 calculation}--\eqref{T2 calculation}, we obtain  
\begin{align}
\frac{\dd}{\dd t}\mathcal{F}_{\varepsilon}(t)=T_{1}+T_{2}=0. \label{T2T4 vanish}    
\end{align}
\end{proof}

In the remainder of this section fix $\varepsilon_{0}>0$ and set
\[
F_{*}:=\sup_{0<\varepsilon\leq\varepsilon_{0}}\mathcal F_{\varepsilon}(0)<\infty.
\]
This uniform bound follows from the hypotheses of Theorem \ref{main thm 1}, since
$\mathcal F_{\varepsilon}(0)\leq 2\mathcal E_{\varepsilon}(0)+\|v(0)\|_{\infty}^{2}$.
In particular, energy conservation gives
$\sqrt\varepsilon\,\|\nabla\Phi_{\varepsilon}(t)\|_{2}\leq\sqrt{2F_{*}}$ uniformly on $[0,T]$.

We denote by $\mathcal{K}_{\varepsilon}(t)$ and $\mathcal{V}_{\varepsilon}(t)$ the kinetic and interaction parts of the modulated energy, respectively, {\it i.e.,} 
\begin{align*}
\mathcal{K}_{\varepsilon}(t):=\frac{1}{2}\int_{\mathbb{T}^{2}\times \mathbb{R}^{2}} \left\vert \xi-v(t,x)\right\vert^{2}f_{\varepsilon}(t,x,\xi)\ \dd x\dd \xi \ \mbox{and}\  \mathcal{V}_{\varepsilon}(t):=\frac{\varepsilon}{2}\int_{\mathbb{T}^{2}} \left\vert \nabla_{x}\Phi_{\varepsilon}\right\vert^{2}(t,x) \ \dd x,   
\end{align*}
so that 
\begin{align}
\mathcal{E}_{\varepsilon}(t):=\mathcal{K}_{\varepsilon}(t)+\mathcal{V}_{\varepsilon}(t). \label{modulated energy}
\end{align}
Moreover, in the sequel, we designate by $D^{\mathbf{s}}_{x}v$ the symmetrized Jacobian of $v$ defined by 
\begin{align*}
D^{\mathbf{s}}_{x}v:= \frac{1}{2}\bigl(D_{x}v+(D_{x}v)^{\top}\bigr)
\end{align*}
The main part of the proof of Theorem \ref{main thm 1} is encapsulated in the following theorem. 
\begin{thm}
With the same assumptions and notations of {\rm Theorem {\rm\ref{main thm 1}}}, the following identity holds for all $t\in \mathbb{R}_{+}${\rm:}
\begin{align}
\frac{\dd}{\dd t}\mathcal{E}_{\varepsilon}(t)=&-\int_{\mathbb{T}^{2}\times \mathbb{R}^{2}} (\xi-v)^{\otimes 2}:D^{\mathbf{s}}_xvf_{\varepsilon}\ \dd x\dd \xi +\int_{\mathbb{T}^{2}\times \mathbb{R}^{2}} (\xi-v)\cdot \nabla_{x} p f_{\varepsilon}\ \dd x\dd \xi \notag
\\
&+\varepsilon\int_{\mathbb{T}^{2}} D^{\mathbf{s}}_{x}v:\nabla_{x} \Phi_{\varepsilon}^{\otimes 2}\ \dd x. 
\label{time derivative of modulated energy}    
\end{align}
Consequently, the following estimate holds{\rm:} 
\begin{align}
\mathcal{E}_{\varepsilon}(t)\leq \mathcal{E}_{\varepsilon}(0)+2\left\Vert D^{\mathbf{s}}v\right\Vert_{\infty} \int_{0}^{t}\mathcal{E}_{\varepsilon}(\tau)\ \dd \tau+O(\sqrt{\varepsilon}) 
\qquad \mbox{for all $t\in [0,T]$},\label{est for modulated energy}      
\end{align}
so that
\begin{align}
 \underset{t\in [0,T]}{\sup}\mathcal{E}_{\varepsilon}(t)\underset{\varepsilon\rightarrow 0}{\longrightarrow}0 
 \qquad \mbox{provided}\ \mathcal{E}_{\varepsilon}(0)\underset{\varepsilon\rightarrow 0}{\longrightarrow }0. \label{asymptotic vanishing mod energy}     
\end{align}
\label{main est thm 1}
\end{thm}

\begin{proof} We divide the proof into four steps.

\smallskip
\textbf{1}. To lighten the notation, we omit the domain of integration from the integrals that follow. By Lemma  \ref{time der of total energy}, 
we have 
$$
\frac{\dd}{\dd t}\mathcal{F}_{\varepsilon}(t)=0.
$$  
Therefore, expanding $\mathcal{E}_{\varepsilon}(t)$, we obtain
\begin{align*}
\frac{\dd}{\dd t}\mathcal{E}_{\varepsilon}(t)
&=\frac{\dd}{\dd t}\Big(-\int\xi\cdot v f_{\varepsilon}\ \dd x\dd \xi+\frac{1}{2}\int \left\vert v\right\vert^{2}f_{\varepsilon}\ \dd x\dd \xi \Big)\\
&= -\int\xi\cdot v\partial_{t}f_{\varepsilon}\ \dd x\dd \xi-\int \partial_{t}(\xi\cdot v)f_{\varepsilon}\ \dd x\dd \xi+\frac{1}{2}\int\partial_{t}(\left\vert v\right\vert^{2} )f_{\varepsilon}\ \dd x\dd \xi+\frac{1}{2}\int\left\vert v\right\vert^{2}\partial_{t}f_{\varepsilon}\ \dd x\dd \xi\\
&=\int \big(\frac{1}{2}\left\vert v\right\vert^{2}-\xi\cdot v\big)\partial_{t}f_{\varepsilon}\ \dd x\dd \xi
+\int\partial_{t}\big(\frac{1}{2}\left\vert v\right\vert^{2} -\xi\cdot v\big)f_{\varepsilon}\ \dd x\dd \xi\\
&=-\int \big(\frac{1}{2}\left\vert v\right\vert^{2}-\xi\cdot v \big)
\big(\xi\cdot \nabla_{x}f_{\varepsilon}-(\nabla_{x}\Phi_{\varepsilon}
+B\xi^{\perp})\cdot \nabla_{\xi}f_{\varepsilon}\big)\ \dd x\dd \xi\\
&\quad +\int\partial_{t}\big(\frac{1}{2}\left\vert v\right\vert^{2} -\xi\cdot v\big)f_{\varepsilon}\ \dd x\dd \xi\\
&=\int \big(\partial_{t}+\xi\cdot \nabla_{x}-B\xi^{\perp}\cdot \nabla_{\xi} \big)\big(\frac{1}{2}\left\vert v\right\vert^{2}-\xi\cdot v\big)f_{\varepsilon}\ \dd x\dd \xi\\
&\quad-\int \nabla_{x}\Phi_{\varepsilon}\cdot\nabla_{\xi}
\big(\frac{1}{2}\left\vert v\right\vert^{2}-\xi\cdot v \big)f_{\varepsilon}\ \dd x\dd \xi:= \int J_{1}f_{\varepsilon}\ \dd x\dd \xi+\int J_{2}f_{\varepsilon}\ \dd x\dd \xi. 
\end{align*}
Note that, in the last equation, we have invoked integration by parts 
together with the observation that $\mathrm{div}_{\xi}(\xi^{\perp})=0$.

\smallskip
\textbf{2}. \textit{Calculation of $J_{1}$}. We start with the calculation of  $J_{1}$: 
\begin{align}
J_{1}=&\big(\partial_{t}+\xi\cdot \nabla_{x}-B\xi^{\perp}\cdot\nabla_{\xi}\big)\big(\frac{1}{2}\left\vert v\right\vert^{2}-\xi\cdot v \big)
=\sum_{j}\big(\partial_{t}+\xi\cdot \nabla_{x}-B\xi^{\perp}\cdot\nabla_{\xi}\big)\big(\frac{1}{2}v_{j}-\xi_{j}\big)v_{j}\notag\\
=&\sum_{j}v_{j}\big(\partial_{t}+\xi\cdot \nabla_{x}-B\xi^{\perp}\cdot \nabla_{\xi}\big)\big(\frac{1}{2}v_{j}-\xi_{j}\big) \notag
+\sum_{j}\big(\frac{1}{2}v_{j}-\xi_{j}\big)
\big(\partial_{t}+\xi\cdot\nabla_{x}-B\xi^{\perp}\cdot \nabla_{\xi}\big)(v_{j}) \notag\\
=&\sum_{j}v_{j}\big(\partial_{t}+\xi\cdot \nabla_{x}-B\xi^{\perp}\cdot \nabla_{\xi}\big)(v_{j})
-\sum_{j}v_{j}\big(\partial_{t}+\xi\cdot \nabla_{x}-B\xi^{\perp}\cdot \nabla_{\xi}\big)(\xi_{j}) \notag\\
&-\sum_{j}\xi_{j}\big(\partial_{t}+\xi\cdot \nabla_{x}-B\xi^{\perp}\cdot\nabla_{\xi}\big)(v_{j})\notag\\
=&\underset{:= J_{1}^{1}}{\underbrace{\sum_{j}(v_{j}-\xi_{j})\big(\partial_{t}+\xi\cdot \nabla_{x}-B\xi^{\perp}\cdot \nabla_{\xi}\big)(v_{j})}}\underset{:= J_{1}^{2}}{\underbrace{-\sum_{j}v_{j}\big(\partial_{t}+\xi\cdot \nabla_{x}-B\xi^{\perp}\cdot \nabla_{\xi}\big)(\xi_{j})}}. \label{J1 2 terms}
\end{align}
The first term in \eqref{J1 2 terms} can be written as  
\begin{align}
J_{1}^{1}&=\sum_{j}(v_{j}-\xi_{j})\big(\partial_{t}v_{j}+\sum_{k}\xi_{k}\partial_{x_{k}}v_{j}\big) \notag\\
&=\sum_{j}\left(v_{j}-\xi_{j}\right)\big(\partial_{t}v_{j}+\sum_{k}(\xi_{k}-v_{k})\partial_{x_{k}}v_{j}+\sum_{k}v_{k}\partial_{x_{k}}v_{j}\big) \notag\\
&=\sum_{j,k} \left(v_{j}-\xi_{j}\right)\left(\xi_{k}-v_{k}\right)\partial_{x_{k}}v_{j} +\sum_{j}\left(v_{j}-\xi_{j}\right)\big(\partial_{t}v_{j}+\sum_{k}v_{k}\partial_{x_{k}}v_{j}\big). 
\label{J11}
\end{align}
The second term in \eqref{J1 2 terms} can be written as 
\begin{align}
J_{1}^{2}
&=B\sum_{j}v_{j}\xi^{\perp}\cdot\nabla_{\xi} \xi_{j}
=B\sum_{j}v_{j}(\xi^{\perp})_{j}=Bv\cdot \xi^{\perp} \nonumber\\
&=-Bv^{\perp}\cdot \xi
{\underset{(v\cdot v^{\perp}=0)}
{=}}
-B v^{\perp}\cdot(\xi-v)
=(v-\xi)\cdot B v^{\perp}.\label{J12}   
\end{align}  
Thus,  we have
\begin{align*}
J_{1}=J_{1}^{1}+J^{2}_{1}=&
-\sum_{j,k}(\xi_{j}-v_{j})(\xi_{k}-v_{k})\partial_{x_{k}}v_{j}+\sum_{j}(v_{j}-\xi_{j})(\partial_{t}v_{j}+\sum_{k}v_{k}\partial_{x_{k}}v_{j}+B(v^{\perp})_{j})\\
&=-(\xi-v)^{\otimes 2}:D^{\mathbf{s}}_{x}v+(v-\xi)\cdot (\partial_{t}v+vD_{x}v+Bv^{\perp})\\
&\underset{\eqref{magnetized Euler}}{=}-(\xi-v)^{\otimes 2}:D^{\mathbf{s}}_{x}v+(\xi-v)\cdot \nabla_{x}p.
\end{align*}
To conclude, we obtain
\begin{align}
\int J_{1} f_{\varepsilon}\ \dd x\dd \xi= -\int (\xi-v)^{\otimes 2}:D^{\mathbf{s}}vf_{\varepsilon}\ \dd x\dd \xi+\int (\xi-v)\cdot \nabla_{x} p f_{\varepsilon}\ \dd x\dd \xi.   \label{J1 integral formula}
\end{align}

\smallskip
\textbf{3}.\textit{ Calculation of $J_{2}$}. Since $\varepsilon\Delta_{x} \Phi_{\varepsilon}=1-\rho_{\varepsilon}$, we have 
\begin{align*}
\int J_{2}f_{\varepsilon}\ \dd x\dd \xi=\int \nabla_{x}\Phi_{\varepsilon}\cdot v \rho_{\varepsilon}\ \dd x=\int\nabla_{x}\Phi_{\varepsilon}\cdot v\ \dd x-\varepsilon\int \nabla_{x}\Phi_{\varepsilon}\cdot v\Delta_{x} \Phi_{\varepsilon}\ \dd x.     
\end{align*}
Integrating by parts and using that $\mathrm{div}_{x}v=0$, we have
\begin{align}
\int \nabla_{x}\Phi_{\varepsilon}\cdot v\ \dd x=-\int \Phi_{\varepsilon}\mathrm{div}_{x}v\ \dd x=0.   \label{vanishing}  
\end{align}
In addition, integrating by parts yields  
\begin{align*}
-\varepsilon\int \nabla_{x}\Phi_{\varepsilon}\cdot v \Delta_{x} \Phi_{\varepsilon}\ \dd x&=\varepsilon\int \nabla_{x}(\nabla_{x} \Phi_{\varepsilon}\cdot v)\cdot \nabla_{x} \Phi_{\varepsilon}\ \dd x\\
&=\varepsilon\int \nabla^{2}_{x}\Phi_{\varepsilon}\nabla_{x} \Phi_{\varepsilon}\cdot v\ \dd x+\varepsilon\int D_{x}v\nabla_{x}\Phi_{\varepsilon}\cdot\nabla_{x}\Phi_{\varepsilon}\ \dd x.    
\end{align*}
Note that 
\begin{align*}
\varepsilon\int \nabla^{2}_{x}\Phi_{\varepsilon}\nabla_{x} \Phi_{\varepsilon}\cdot v\ \dd x=\frac{\varepsilon}{2}\int \nabla_{x}\left\vert \nabla_{x}\Phi_{\varepsilon}\right\vert^{2}\cdot v\ \dd x
=-\frac{\varepsilon}{2}\int\left\vert \nabla_{x}\Phi_{\varepsilon}\right\vert^{2}\mathrm{div}_{x}v =0.     
\end{align*}
Then
\begin{align}
-\varepsilon\int \nabla_{x}\Phi_{\varepsilon}\cdot v \Delta_{x} \Phi_{\varepsilon}\ \dd x=\varepsilon \int D_{x}v\nabla_{x}\Phi_{\varepsilon}\cdot \nabla_{x}\Phi_{\varepsilon}\ \dd x. \label{second term in J2}    
\end{align}
Thus, combining \eqref{vanishing}--\eqref{second term in J2},
we obtain
\begin{align}
\int J_{2}f_{\varepsilon}\ \dd x\dd \xi=\varepsilon\int D_{x}v\nabla_{x}\Phi_{\varepsilon}\cdot \nabla_{x} \Phi_{\varepsilon}\ \dd x=\varepsilon\int D^{\mathbf{s}}_{x}v:\nabla_{x} \Phi_{\varepsilon}^{\otimes 2}\ \dd x. \label{J2 formula} 
\end{align}
The combination of \eqref{J1 integral formula} and \eqref{J2 formula} yields \eqref{time derivative of modulated energy}.

\smallskip 
\textbf{4}. \textit{Conclusion}. 
Integrating \eqref{time derivative of modulated energy} in time, we have
\begin{align*}
\mathcal{E}_{\varepsilon}(t)&=\mathcal{E}_{\varepsilon}(0)-\int_{0}^{t}\int (\xi-v)^{\otimes 2}:D^{\mathbf{s}}_{x}vf_{\varepsilon}\ \dd x\dd \xi\dd \tau\\
&\quad+\int_{0}^{t}\int (\xi-v)\cdot\nabla_{x} pf_{\varepsilon}\ \dd x\dd \xi \dd \tau\\
&\quad+\varepsilon\int_{0}^{t}\int D^{\mathbf{s}}_{x}v:\nabla_{x} \Phi_{\varepsilon}^{\otimes 2}\ \dd x \dd \tau=\mathcal{E}_{\varepsilon}(0)+\sum_{k=1}^{3}I_{k}.   
\end{align*}
 We estimate each $I_{k}$ separately. The estimates on $I_{1}$ and $I_{3}$ are straightforward:
\begin{align}
 I_{1}\leq \left\Vert D^{\mathbf{s}}_{x}v\right\Vert_{\infty}
\int_{0}^{t}\int\left\vert \xi-v\right\vert^{2}f_{\varepsilon}\ \dd x\dd \xi \dd \tau \leq 2\left\Vert D^{\mathbf{s}}_{x}v\right\Vert_{\infty} \int_{0}^{t}\mathcal{K}_{\varepsilon}(\tau)\ \dd \tau, \label{I1}     
\end{align}
and 
\begin{align}
I_{3}\leq \left\Vert D^{\mathbf{s}}_{x}v\right\Vert_{\infty} \varepsilon\int_{0}^{t}\int \left\vert \nabla_{x}\Phi_{\varepsilon}\right\vert^{2}(\tau,x)\ \dd x \dd \tau=2\left\Vert D^{\mathbf{s}}_{x}v\right\Vert_{\infty}\int_{0}^{t}\mathcal{V}_{\varepsilon}(\tau) \ \dd \tau.  \label{I3}     
\end{align}
We proceed by estimating $I_{2}$. First, note that 
\begin{align}
&\int v\cdot \nabla_{x} pf_{\varepsilon}\ \dd x\dd \xi=\int v\cdot \nabla_{x} p\rho_{\varepsilon}\ \dd x \notag\\
&=\int v\cdot \nabla_{x} p(1-\varepsilon\Delta_{x} \Phi_{\varepsilon})\ \dd x=\int v\cdot \nabla_{x}p\ \dd x-\varepsilon\int v\cdot \nabla_{x}p\Delta_{x} \Phi_{\varepsilon}\ \dd x \notag\\
&=\varepsilon\int \nabla_{x} (v\cdot \nabla_{x} p)\cdot \nabla_{x} \Phi_{\varepsilon}\ \dd x, \label{first term in I4} 
\end{align}
where we have used that $\int v\cdot \nabla_{x}p\ \dd x=-\int p\mathrm{div}_{x} v\ \dd x=0$ in the last identity.   
Therefore, by conservation of 
energy (Lemma \ref{time der of total energy}) and Lemma \ref{basic estimates},  
we can estimate the right-hand side of \eqref{first term in I4} by  
\begin{align}
\Big|\varepsilon\int_{\mathbb T^{2}}\nabla_x(v\cdot\nabla_xp)\cdot\nabla_x\Phi_\varepsilon\,\dd x\Big|
&\leq \sqrt{2\varepsilon F_*}\big(\left\Vert D_{x}v\nabla_{x}p\right\Vert_{L^{\infty}_{t}L^{2}_{x}}+\left\Vert v\nabla_{x}^{2}p\right\Vert_{L^{\infty}_{t}L^{2}_{x}}\big)
=O(\sqrt\varepsilon). \label{nelligible term 1st}
\end{align}
In addition, using that $\partial_{t}\rho_{\varepsilon}+\mathrm{div}_{x}J_{\varepsilon}=0$, 
we see that  
\begin{align*}
\int \xi \cdot \nabla_{x} pf_{\varepsilon} \ \dd x\dd \xi&=-\int p\mathrm{div}_{x}J_{\varepsilon
}\ \dd x=\int p\partial_{t}\rho_{\varepsilon}\ \dd x\\&=-\varepsilon\int\partial_{t}\Delta_{x} \Phi_{\varepsilon}p\ \dd x=\varepsilon \int \partial_{t}\nabla_{x} \Phi_{\varepsilon}\cdot \nabla_{x} p\ \dd x.     
\end{align*}
Thus, integrating by parts in the time variable, we obtain  
\begin{align}
\int_{0}^{t}\int \xi\cdot \nabla_{x}p f_{\varepsilon}\ \dd x\dd \xi
&=\varepsilon\int_{0}^{t}\int \partial_{s}\nabla_{x}\Phi_{\varepsilon}(s,x)\cdot \nabla_{x} p(s,x)\ \dd x\dd s \notag\\
&=-\varepsilon\int\int_{0}^{t}\nabla_{x} \Phi_{\varepsilon}(s,x)\cdot \partial_{s}\nabla_{x} p(s,x)\ \dd s\dd x \notag\\
&\quad+\varepsilon\int \nabla_{x} \Phi_{\varepsilon}(t,x)\cdot\nabla_{x} p(t,x)\ \dd x-\varepsilon\int \nabla_{x} \Phi_{\varepsilon}(0,x)\cdot \nabla_{x} p(0,x)\ \dd x. \label{integration by parts in time}       
\end{align}
Applying the conservation of energy and Lemma \ref{basic estimates},
we deduce 
\begin{align}
\Big|\varepsilon\int_{0}^{t}\int_{\mathbb T^{2}}\nabla_{x}\Phi_{\varepsilon}\cdot \partial_{s}\nabla_{x}p\,\dd x\dd s\Big|
\leq T\sqrt{2\varepsilon F_*}\left\Vert \partial_{t}\nabla_{x}p\right\Vert_{L^{\infty}_{t}L^{2}_{x}}
=O(\sqrt\varepsilon). \label{square eps est}
\end{align}
By the same token, we have 
\begin{align}
\Big|\varepsilon \int \nabla_{x}\Phi_{\varepsilon}(t,x)\cdot\nabla_{x}p(t,x)\ \dd x\Big|
\leq \sqrt{2\varepsilon F_*}\left\Vert \nabla_{x}p\right\Vert_{L^{\infty}_{t}L^{2}_{x}}=O(\sqrt{\varepsilon}).  \label{square eps est 1}
\end{align}
and 
\begin{align}
\Big|\varepsilon \int \nabla_{x}\Phi_{\varepsilon}(0,x)\cdot \nabla_{x}p(0,x)\ \dd x\Big|
\leq \sqrt{2\varepsilon F_*}\left\Vert \nabla_{x}p(0,\cdot)\right\Vert_{2}=O(\sqrt{\varepsilon}).\label{square eps est 2}
\end{align}
Therefore, in view of \eqref{integration by parts in time}--\eqref{square eps est 2}, we see that 
\begin{align}
\Big|\int_{0}^{t}\int \xi\cdot \nabla_{x}pf_{\varepsilon}\ \dd x\dd \xi\Big|
\leq T\sqrt{2\varepsilon F_*}\left\Vert \partial_{t}\nabla_{x}p\right\Vert_{L^{\infty}_{t}L^{2}_{x}}
+2\sqrt{2\varepsilon F_*}\left\Vert \nabla_{x}p\right\Vert_{L^{\infty}_{t}L^{2}_{x}}
=O(\sqrt{\varepsilon}). \label{nelligible term 2nd}
\end{align}
Gathering \eqref{nelligible term 1st}--\eqref{nelligible term 2nd}, 
we deduce 
\begin{align*}
|I_{2}|&\leq T\sqrt{2\varepsilon F_*}\big(\left\Vert D_{x}v\nabla_{x}p\right\Vert_{L^{\infty}_{t}L^{2}_{x}}+\left\Vert v\nabla_{x}^{2}p\right\Vert_{L^{\infty}_{t}L^{2}_{x}}+\left\Vert \partial_{t}\nabla_{x}p\right\Vert_{L^{\infty}_{t}L^{2}_{x}}\big)
+2\sqrt{2\varepsilon F_*}\left\Vert \nabla_{x}p\right\Vert_{L^{\infty}_{t}L^{2}_{x}}\\
&\leq C_T\sqrt\varepsilon.
\end{align*}
Together with \eqref{I1}--\eqref{I3}, we conclude
\begin{align*}
\mathcal{E}_{\varepsilon}(t)\leq \mathcal{E}_{\varepsilon}(0)+2\left\Vert D^{\mathbf{s}}_{x}v\right\Vert_{\infty} \int_{0}^{t}\mathcal{E}_{\varepsilon}(\tau)\ \dd \tau+O(\sqrt{\varepsilon}),   
\end{align*}
 which leads to  \eqref{est for modulated energy}. 
 Finally, by Gr\"onwall's inequality, 
 we conclude \eqref{asymptotic vanishing mod energy}.  
\end{proof}

We are now well positioned to conclude the proof of Theorem \ref{main thm 1}.

\medskip
\textit{Proof of Theorem \ref{main thm 1}.} By Theorem \ref{main est thm 1}, 
we have  
\begin{align*}
\underset{t\in [0,T]}{\sup}\left\Vert \rho_{\varepsilon}(t,\cdot)-1\right\Vert_{\dot{H}^{-1}_{x}}^{2}=2\varepsilon\,\underset{t\in [0,T]}{\sup}\mathcal{V}_{\varepsilon}(t)\le 2\varepsilon\,\underset{t\in [0,T]}{\sup}\mathcal{E}_{\varepsilon}(t)\underset{\varepsilon \rightarrow 0}{\longrightarrow} 0.
\end{align*}
In addition, given $b\in W^{1,\infty}(\mathbb{T}^{2};\mathbb{R}^{2})$, we have 
\begin{align}
\int_{\mathbb{T}^{2}}b(x)\cdot (J_{\varepsilon}-v)(t,x)\ \dd x
&=\int_{\mathbb{T}^{2}}b(x)\cdot(J_{\varepsilon}-\rho_{\varepsilon}v)(t,x)\ \dd x \notag\\
&\quad+\int_{\mathbb{T}^{2}}b(x)\cdot(\rho_{\varepsilon}-1)v(t,x)\ \dd x
 := \mathcal{I}+\mathcal{J}.
\label{test b}
\end{align}
Thanks to the Cauchy-Schwarz inequality, we obtain
\begin{align}
\left|\mathcal{I}\right|&=\Big|\int_{\mathbb{T}^{2}\times\mathbb{R}^{2}} b(x)\cdot (\xi-v)f_{\varepsilon}(t,x,\xi)\ \dd x\dd\xi\Big|\notag\\
&\leq \Big
(\int_{\mathbb{T}^{2}\times \mathbb{R}^{2}}\left\vert \xi-v\right\vert^{2}f_{\varepsilon}(t,x,\xi)\ \dd x\dd \xi\Big)^{\frac{1}{2}}\Big(\int_{\mathbb{T}^{2}\times \mathbb{R}^{2}}\left\vert b\right\vert^{2} f_{\varepsilon}(t,x,\xi)\ \dd x\dd \xi\Big)^{\frac{1}{2}}\nonumber\\
&\leq \left\Vert b\right\Vert_{\infty} \sqrt{2\mathcal{E}_{\varepsilon}(t)}\underset{\varepsilon \rightarrow 0}{\longrightarrow} 0. \label{I est}    
\end{align}
Furthermore, by duality, we have 
\begin{align}
\mathcal{J}\leq \left\Vert bv(t,\cdot)\right\Vert_{\dot H^{1}_{x}}\left\Vert \rho_{\varepsilon}(t,\cdot)-1\right\Vert_{\dot{H}^{-1}_{x}}\leq 2\left\Vert b\right\Vert_{W^{1,\infty}_{x}}\left\Vert v\right\Vert_{L^{\infty}_{t}W^{1,\infty}_{x}}\left\Vert \rho_{\varepsilon}(t,\cdot)-1\right\Vert_{\dot{H}^{-1}_{x}}\underset{\varepsilon \rightarrow 0}{\longrightarrow} 0.     \label{J est}    
\end{align}
To conclude, combining \eqref{I est}--\eqref{J est}, we obtain 

\begin{align*}
\left\Vert J_{\varepsilon}(t)-v(t,\cdot)\,\dd x\right\Vert_{\mathrm{BL}^{*}}
:=\sup_{\|b\|_{W^{1,\infty}}\leq1}
\Big|\int_{\mathbb{T}^{2}}b(x)\cdot\bigl(J_\varepsilon(t,\dd x)-v(t,x)\,\dd x\bigr)\Big|
\underset{\varepsilon \rightarrow 0}{\longrightarrow}0
\end{align*}
uniformly for \(t\in[0,T]\).
Therefore, the current converges in the dual bounded--Lipschitz topology.
\qed

\section{The Supercritical Mean Field Limit}\label{sec mean field}
This section is devoted to the derivation of \eqref{magnetized Euler} from the large particle 
system \eqref{trajectories intro}. The total energy is given by 
\begin{align}
\mathcal{F}_{\varepsilon,N}(t):=\frac{1}{2N}\sum_{i=1}^{N}\left\vert \xi_{i}(t)\right\vert^{2}+\frac{1}{2N^{2}\varepsilon}\sum_{i\neq j}V_{\geq 0}(x_{i}(t)-x_{j}(t)),\label{total energy discrete}     
\end{align}
where we have defined $V_{\geq 0}=V+C_{0}$ for some fixed constant $C_{0}>0$ such that $V+C_{0}\geq 0$.  
Furthermore, recall that we have defined $\mathfrak{U}=-\Delta_{x}p$, or equivalently $\mathfrak{U}:= \mathrm{div}_{x}\mathrm{div}_{x}(v\otimes v)+\mathrm{div}_{x}(Bv^{\perp})$. 
We recall several fundamental functional inequalities, all of which are due to \cite{duerinckx2020mean} and are crucial for the analysis to follow. 
It will be convenient to apply the following shorthand notation: 
Given $\rho\in L^{\infty}(\mathbb{T}^{2})$ and  $X_{N}=(x_{1},\cdots\!,x_{N})\in \Delta_{N}^{c}$, 
let $\rho_{N}=\frac{1}{N}\sum_{i=1}^{N}\delta_{x_{i}}$ and define 
\begin{align*}
\mathcal{V}_{N}(\rho,\rho_{N}):=\iint_{\Delta^{c}}V(x-y)(\rho_{N}-\rho)^{\otimes 2}(\dd x\dd y).      
\end{align*}
Moreover, define 
\begin{align*}
\mathcal{K}_{\varepsilon,N}(t)
:= \frac{1}{2N}\sum_{i=1}^{N}\left\vert \xi_{i}(t)-v(t,x_{i})\right\vert^{2}.     
\end{align*}
Recalling \eqref{magnetized renormalized modulated energy},   
we may write $\mathcal{E}_{\varepsilon,N}(t)$ as 
\begin{align*}
\mathcal{E}_{\varepsilon,N}(t)=\mathcal{K}_{\varepsilon,N}(t)+\frac{1}{2\varepsilon}\mathcal{V}_{N}(1+\varepsilon\mathfrak{U},\rho_{\varepsilon,N}(t,\cdot))
+\frac{C}{2\varepsilon}\big(1+\left\Vert 1+\varepsilon \mathfrak{U}\right\Vert_{\infty}\big)\frac{1+\log N}{N}.     
\end{align*}  
The following {\it asymptotic positivity} ensures that $\mathcal{E}_{\varepsilon,N}(t)\geq 0$:

\begin{prop}[\cite{duerinckx2020mean},\,Corollary 3.5]
Suppose that
\begin{itemize}
\item[{\rm(i)}]$\rho \in L^{\infty}(\mathbb{T}^{2});$ 
   \item[{\rm(ii)}]
   $X_{N}=(x_{1},\cdots\!,x_{N})\in \Delta_{N}^{c}$ and $\rho_{N}=\frac{1}{N}\sum_{i=1}^{N}\delta_{x_{i}}$. 
\end{itemize}
Then there exist a constant $C>0$ such that
\begin{align*}
\mathcal{V}_{N}(\rho,\rho_{N}) \geq -C\big(1+\left\Vert \rho \right\Vert_{\infty}\big)
\frac{1+\log N}{N}. 
\end{align*}
\label{non negativity}
\end{prop}
\begin{prop}[\cite{duerinckx2020mean},\, Proposition 3.6]
 Let the assumptions of {\rm Proposition \ref{non negativity}} hold. 
 Then there exits some $C>0$ such that, for any $\varphi\in W^{1,\infty}(\mathbb{T}^{2})$,
 \begin{align*}
&\Big|\int_{\mathbb{T}^{2}}\varphi(x)\left(\rho_{N}-\rho\right)(\dd x)\Big|\nonumber\\
&\leq C\frac{\left\Vert \nabla_{x} \varphi\right\Vert_{\infty}}{\sqrt{N}} +\left\Vert \nabla_{x} \varphi\right\Vert_{2} \Big(\mathcal{V}_{N}(\rho,\rho_{N})
+C\big(1+\left\Vert \rho \right\Vert_{\infty}\big)\frac{1+\log N}{N}\Big)^\frac{1}{2}.
 \end{align*}
 \label{coercivity inequality}
\end{prop}

\begin{thm}[\cite{duerinckx2020mean},\,Proposition 1.1]\label{commutator estimate} 
Let the assumptions of {\rm Proposition \ref{non negativity}} hold. 
Then there is some $C>0$ such that, for any Lipschitz vector field  $v:\mathbb{T}^{2}\rightarrow \mathbb{R}^{2}$,
 \begin{align*}
&\Big\vert \iint_{\Delta^{c}}\left(v(x)-v(y)\right)\nabla_{x} V(x-y)\left(\rho_{N}-\rho\right)^{\otimes2}(\dd x\dd y)\Big\vert\\
&\leq C\left\Vert D_{x} v\right\Vert_{\infty}
\Big(\mathcal{V}_{N}(\rho,\rho_{N})+C\big(1+\left\Vert \rho\right\Vert_{\infty}\big)\frac{1
+\log N}{N}\Big). 
\end{align*}   
\label{Commutator est}
\end{thm}

By the same argument as in Lemma \ref{time der of total energy}, the total energy associated with the large particle system is conserved:
 \begin{lem} \label{Conservation of discrete energy}
Let $(x_{1}^{0},\cdots\!,x_{N}^{0})\in \Delta_{N}^{c}$, 
and let $(x_{1}(t),\cdots\!,x_{N}(t),\xi_{1}(t),\cdots\!,\xi_{N}(t))\in C^{1}(\mathbb{R};\mathbb{T}^{2N}\times \mathbb{R}^{2N})$ 
be a solution to \eqref{trajectories intro}. Let $\mathcal{F}_{\varepsilon,N}(t)$ be given by \eqref{total energy discrete}. Then
\begin{align*}
\frac{\dd}{\dd t}\mathcal{F}_{\varepsilon,N}(t)=0 \qquad \mbox{for all}\ t\in \mathbb{R}_{+}.     
\end{align*}
\end{lem}
\begin{proof}
We first have 
\begin{align}
\frac{\dd}{\dd t}\Big(\frac{1}{2N}\sum_{i=1}^{N}\left\vert \xi_{i}(t)\right\vert^{2} \Big)
&=\frac{1}{N}\sum_{i=1}^{N}\xi_{i}(t)\cdot \dot \xi_{i}(t) \notag\\
&=-\frac{1}{N}\sum_{i=1}^{N}\xi_{i}(t)\cdot\Big(B(t,x_{i}(t))\xi_{i}^{\perp}(t)+\frac{1}{N\varepsilon}\sum_{j:j\neq i}\nabla_{x} V(x_{i}(t)-x_{j}(t))\Big) \notag\\
&=-\frac{1}{N^{2}\varepsilon}\sum_{i=1}^{N}\sum_{j:j\neq i}\xi_{i}(t)\cdot \nabla_{x} V(x_{i}(t)-x_{j}(t)). \label{time derivative of discrete kinetic part}     
\end{align}
Moreover, since \(\nabla_x V\) is odd, a direct computation yields
\begin{align}
\frac{\dd}{\dd t}\Big(\frac{1}{2N^{2}\varepsilon}\sum_{i\neq j}V(x_{i}(t)-x_{j}(t))\Big)
&=\frac{1}{2N^{2}\varepsilon}\sum_{i\neq j}\nabla_{x} V(x_{i}(t)-x_{j}(t))\cdot(\dot x_{i}(t)-\dot x_{j}(t)) \notag\\
&=\frac{1}{2N^{2}\varepsilon}\sum_{i\neq j} \nabla_{x} V(x_{i}(t)-x_{j}(t))\cdot(\xi_{i}(t)-\xi_{j}(t)) \notag\\
&=\frac{1}{N^{2}\varepsilon}\sum_{i\neq j}\xi_{i}(t)\cdot\nabla_{x} V(x_{i}(t)-x_{j}(t)). \label{time der of discrete interation part}  
\end{align}
Combining \eqref{time derivative of discrete kinetic part}--\eqref{time der of discrete interation part},
we conclude that $\frac{\dd}{\dd t}\mathcal{F}_{\varepsilon,N}(t)=0$. 
\end{proof}

\begin{rem}
The conservation of energy as stated in {\rm Lemma \ref{Conservation of discrete energy}} holds whenever $V$ is even, $C^1$, and has a bounded gradient. The utility of this remark will become evident in {\rm Section \ref{wellposed-trajectories-sec}}.   \label{rem about conse of energy}  
\end{rem}

\begin{thm}
Let the assumptions of {\rm Theorem \ref{main thm renormalized}} hold, 
and let $\mathcal{E}_{\varepsilon,N}(t)$ be given by \eqref{magnetized renormalized modulated energy}. Then, for all $t\in \mathbb{R}_{+}$, 
\begin{align}
\frac{\dd}{\dd t}\mathcal{E}_{\varepsilon,N}(t)=& -\frac{1}{N}\sum_{i=1}^{N}\big(v(t,x_i(t))-\xi_i(t)\big)^{\otimes2}:D_xv(t,x_i(t))
\notag\\
&+\frac{1}{2\varepsilon}\iint_{\Delta^{c}}(v(t,x)-v(t,y))\cdot \nabla_{x} V(x-y)\left(\rho_{\varepsilon,N}-1-\varepsilon \mathfrak{U}\right)^{\otimes 2}(t,\dd x\dd y) \notag \notag\\
&-\int_{\mathbb{T}^{2}} \mathrm{div}_{x} (-\Delta_{x})^{-1}(v\mathfrak{U})(t,x)\left(\rho_{\varepsilon,N}-1-\varepsilon\mathfrak{U}\right)(t,\dd x) \notag\\
&-\int_{\mathbb{T}^{2}}\partial_{t}p(t,x)\left(\rho_{\varepsilon,N}-1-\varepsilon\mathfrak{U}\right)(t,\dd x).  
\end{align}
Consequently, it follows that 
\begin{align*}
\underset{t\in [0,T]}{\sup}\mathcal{E}_{\varepsilon,N}(t)\underset{N\rightarrow \infty}{\longrightarrow}0\ \qquad 
\mbox{provided}\ \mathcal{E}_{\varepsilon,N}(0)\underset{N\rightarrow \infty}{\longrightarrow}0\ \mbox{and}\ \frac{\log N}{\varepsilon N}\underset{N\rightarrow \infty}{\longrightarrow} 0.  \end{align*}
\label{time derivative of modulated energy statement}
\end{thm}
\begin{proof}  We divide the proof into five steps.

\smallskip
\textbf{1}. \textit{Calculation of $\frac{\dd}{\dd t}\mathcal{E}_{\varepsilon,N}(t)$}.
We calculate the time-derivative of $\mathcal{E}_{\varepsilon,N}(t)$.
By Lemma \ref{Conservation of discrete energy}, 
we see that $\frac{\dd}{\dd t}\mathcal{F}_{\varepsilon,N}(t)=0$, which implies 
\begin{align*}
\frac{\dd}{\dd t}\mathcal{E}_{\varepsilon,N}(t)
&=\frac{\dd}{\dd t}\Big(-\frac{1}{N}\sum_{i=1}^{N}\xi_{i}(t)\cdot v(t,x_{i}(t))+\frac{1}{2N}\sum_{i=1}^{N}\left\vert v(t,x_{i}(t))\right\vert^{2} \Big)\\
&\quad-\frac{\dd}{\dd t}\frac{1}{2\varepsilon}\iint_{\Delta^{c}}V(x-y)(1+\varepsilon\mathfrak{U}(t,y))\rho_{\varepsilon,N}(t,\dd x)\ \dd y \\
&\quad-\frac{\dd}{\dd t}\frac{1}{2\varepsilon}\iint_{\Delta^{c}}V(x-y)(1+\varepsilon\mathfrak{U}(t,x))\rho_{\varepsilon,N}(t,\dd y)\ \dd x\\
&\quad+\frac{\dd}{\dd t}\frac{1}{2\varepsilon}\int_{\mathbb{T}^{2}} V\ast (1+\varepsilon\mathfrak{U})(t,x)(1+\varepsilon\mathfrak{U})(t,x)\ \dd x\\
&= \frac{\dd}{\dd t}\Big(-\frac{1}{N}\sum_{i=1}^{N}\xi_{i}(t)\cdot v(t,x_{i}(t))+\frac{1}{2N}\sum_{i=1}^{N}\left\vert v(t,x_{i}(t))\right\vert^{2} \Big)\\
&\quad-\frac{\dd}{\dd t}\frac{1}{N}\sum_{i=1}^{N}V\ast \mathfrak{U}(t,x_{i}(t))\\
&\quad+\frac{\dd}{\dd t}\frac{1}{2\varepsilon}\int_{\mathbb{T}^{2}} V\ast (1+\varepsilon\mathfrak{U})(t,x)(1+\varepsilon \mathfrak{U})(t,x)\ \dd x, 
\end{align*}
where, in the last identity, we have used that $V$ is even so that  
\begin{align*}
&\iint_{\Delta^{c}}V(x-y)(1+\varepsilon\mathfrak{U}(t,y))\rho_{\varepsilon,N}(t,\dd x)\ \dd y\\
&=\iint_{\Delta^{c}}V(x-y)(1+\varepsilon\mathfrak{U}(t,x))\rho_{\varepsilon,N}(t,\dd y)\ \dd x
=\frac{\varepsilon}{N}\sum_{i=1}^{N}V\ast \mathfrak{U}(t,x_{i}(t)).  
\end{align*}
To lighten notation, we hereafter omit the time variable $t$ when no ambiguity arises. Thanks to \eqref{trajectories intro}, we obtain
\begin{align}
&\frac{\dd}{\dd t}\Big(-\frac{1}{N}\sum_{i=1}^{N}\xi_{i}\cdot v(x_{i})\Big)& \notag\\
&=-\frac{1}{N}\sum_{i=1}^{N}\dot \xi_{i}\cdot v(x_{i})-\frac{1}{N}\sum_{i=1}^{N}\xi_{i}\cdot (\partial_{t}v(x_{i})+\dot x_{i}D_{x}v(x_{i}))& \notag\\
&=\frac{1}{N}\sum_{i=1}^{N}
\Big(B(x_{i})\xi_{i}^{\perp}+\frac{1}{N\varepsilon}\sum_{j:j\neq i}\nabla_{x} V(x_{i}-x_{j})\Big)\cdot v(x_{i})-\frac{1}{N}\sum_{i=1}^{N}\xi_{i}\cdot (\partial_{t}v(x_{i})+\xi_{i}D_{x}v(x_{i}))& \notag\\
&=\frac{1}{N^{2}\varepsilon}\sum_{i\neq j}\nabla_{x} V(x_{i}-x_{j})\cdot v(x_{i})
-\frac{1}{N}\sum_{i=1}^{N}\xi_{i}\cdot(\partial_{t}v(x_{i})+\xi_{i}D_{x}v(x_{i})) \notag\\
&\quad\,+\frac{1}{N}\sum_{i=1}^{N}B(x_{i})\xi_{i}^{\perp}\cdot v(x_{i}). \label{first time derivative}
\end{align}
In addition, we have  
\begin{align}
\frac{\dd}{\dd t}\Big(\frac{1}{2N}\sum_{i=1}^{N}\left\vert v(x_{i})\right\vert^{2} \Big)
&=\frac{1}{N}\sum_{i=1}^{N}v(x_{i})\cdot\left(\partial_{t}v(x_{i})+\dot x_{i}D_{x}v(x_{i})\right) \notag\\
&=\frac{1}{N}\sum_{i=1}^{N}v(x_{i})\cdot\left(\partial_{t}v(x_{i})+\xi_{i}D_{x}v(x_{i})\right). \label{second time derivative}
\end{align}
Moreover, we compute 
\begin{align}
-\frac{\dd}{\dd t}\frac{1}{N}\sum_{i=1}^{N}V\ast \mathfrak{U}(x_{i})=&-\frac{1}{N}\sum_{i=1}^{N}\int_{\mathbb{T}^{2}} V(x_{i}-y)\partial_{t}\mathfrak{U}(y)\ \dd y \notag\\
&-\frac{1}{N}\sum_{i=1}^{N}\dot x_{i}\cdot \int_{\mathbb{T}^{2}} \nabla_{x} V(x_{i}-y)\mathfrak{U}(y)\ \dd y \notag\\
=& -\frac{1}{N}\sum_{i=1}^
{N}\int_{\mathbb{T}^{2}} V(x_{i}-y)\partial_{t}\mathfrak{U}(y)\ \dd y-\frac{1}{N}\sum_{i=1}^{N}\xi_{i}\cdot \int_{\mathbb{T}^{2}} \nabla_{x} V(x_{i}-y)\mathfrak{U}(y)\ \dd y 
\end{align}
and 
\begin{align}&
\frac{\dd}{\dd t}\frac{1}{2\varepsilon}\int_{\mathbb{T}^{2}} V\ast (1+\varepsilon\mathfrak{U})(x)(1+\varepsilon \mathfrak{U})(x)\ \dd x \notag=\frac{\varepsilon}{2}\frac{\dd}{\dd t}\int_{\mathbb{T}^{2}} V\ast \mathfrak{U}(x)\mathfrak{U}(x)\ \dd x\\
&=\varepsilon\iint_{\mathbb{T}^{2}\times \mathbb{T}^{2}}V(x-y)\mathfrak{U}(x)\partial_{t}\mathfrak{U}(y)\ \dd x\dd y=\iint_{\mathbb{T}^{2}\times \mathbb{T}^{2}} V(x-y)(1+\varepsilon\mathfrak{U}(x))\partial_{t}\mathfrak{U}(y)\ \dd x\dd y. 
\label{time derivative of corrector 2}
\end{align}
So gathering \eqref{first time derivative}--\eqref{time derivative of corrector 2},
we find 
\begin{align}
\frac{\dd}{\dd t}\mathcal{E}_{\varepsilon,N}(t)
=&\,\frac{1}{N^{2}\varepsilon}\sum_{i\neq j}\nabla_{x} V(x_{i}-x_{j})\cdot v(x_{i})
-\frac{1}{N}\sum_{i=1}^{N}\xi_{i}\cdot(\partial_{t}v(x_{i})+\xi_{i}D_{x}v(x_{i})) \notag\\
&\,+\frac{1}{N}\sum_{i=1}^{N}B(x_{i})\xi_{i}^{\perp}\cdot v(x_{i})+\frac{1}{N}\sum_{i=1}^{N}v(x_{i})\cdot\left(\partial_{t}v(x_{i})+\xi_{i}D_{x}v(x_{i})\right) \notag\\
&\,-\frac{1}{N}\sum_{i=1}^
{N}\int_{\mathbb{T}^{2}} V(x_{i}-y)\partial_{t}\mathfrak{U}(y)\ \dd y-\frac{1}{N}\sum_{i=1}^{N}\xi_{i}\cdot \int_{\mathbb{T}^{2}} \nabla_{x} V(x_{i}-y)\mathfrak{U}(y)\ \dd y \notag\\
&\,+\iint_{\mathbb{T}^{2}\times \mathbb{T}^{2}} V(x-y)(1+\varepsilon\mathfrak{U}(x))\partial_{t}\mathfrak{U}(y)\ \dd x\dd y:= \sum_{k=1}^{7}I_{k}. 
\end{align}

\smallskip
\textbf{2}. \textit{Calculation of $I_{2}+I_{3}+I_{4}+I_{6}$}. First, notice that
\begin{align}
I_{2}+I_{4}
=\frac{1}{N}\sum_{i=1}^{N}(v(x_{i})-\xi_{i})\cdot 
\big(\partial_{t}v(x_{i})+\xi_{i}D_{x}v(x_{i})\big). \label{I2 and I4 calculation} 
\end{align}
Since $v$ is a solution to \eqref{magnetized Euler}, we have
\begin{align}
\partial_{t}v(x_{i})+\xi_{i}D_{x}v(x_{i})=(\xi_{i}-v(x_{i}))D_{x}v(t,x_{i})-B(x_{i})v^{\perp}(x_{i})-\nabla_{x} p(x_{i}).  \label{manipulation of first term in I}    
\end{align}
In addition, using that $-\Delta_{x} p=\mathfrak{U}$, we obtain
\begin{align}
I_{6}&=-\frac{1}{N}\sum_{i=1}^{N}\xi_{i}\cdot \int_{\mathbb{T}^{2}} \nabla_{x} V(x_{i}-y)\mathfrak{U}(y)\ \dd y \notag\\
&=\frac{1}{N}\sum_{i=1}^{N} \xi_{i}\cdot \int_{\mathbb{T}^{2}} \nabla_{x} V(x_{i}-y)\Delta_{y} p(y)\ \dd y=-\frac{1}{N}\sum_{i=1}^{N}\xi_{i}\cdot \nabla_{x} p(x_{i}). \label{manipulation of second term in I}  
\end{align}
In the last identity above, we have integrated by parts, together with the identity 
$-\Delta_{x} V=\delta_{0}-1$. 
Therefore, gathering \eqref{I2 and I4 calculation}--\eqref{manipulation of second term in I}, 
we obtain 
\begin{align}
I_{2}+I_{4}+I_{6}=&-\frac{1}{N}\sum_{i=1}^{N}(v(x_{i})-\xi_{i})\cdot\big((v(x_{i})-\xi_{i})D_{x}v(x_{i})\big) \notag\\
&-\frac{1}{N}\sum_{i=1}^{N}v(x_{i})\cdot \nabla_{x} p (x_{i})-\frac{1}{N}\sum_{i=1}^{N}B(x_{i})(v(x_{i})-\xi_{i})\cdot v^{\perp}(x_{i}). \label{I2I4I6} 
\end{align}
Furthermore, we compute 
\begin{align}
I_{3}&=\frac{1}{N}\sum_{i=1}^{N}B(x_{i})\xi_{i}^{\perp}\cdot v(x_{i})=-\frac{1}{N}\sum_{i=1}^{N}B(x_{i})\xi_{i}\cdot v^{\perp}(x_{i}) \notag\\
&=-\frac{1}{N}\sum_{i=1}^{N}B(x_{i})(\xi_{i}-v(x_{i}))\cdot v^{\perp}(x_{i})=\frac{1}{N}\sum_{i=1}^{N}B(x_{i})(v(x_{i})-\xi_{i})\cdot v^{\perp}(x_{i}). \label{I3 expanded}      
\end{align}
Thus, combining \eqref{I2I4I6} with \eqref{I3 expanded} yields
\begin{align}
I:&= I_{2}+I_{3}+I_{4}+I_{6}\notag\\
&=-\frac{1}{N}\sum_{i=1}^{N}(v(x_{i})-\xi_{i})\cdot\big((v(x_{i})-\xi_{i})D_{x}v(x_{i})\big)
-\frac{1}{N}\sum_{i=1}^{N}v(x_{i})\cdot \nabla_{x} p(x_{i}).  \label{I2-I8 formula}  
\end{align}

\smallskip
\textbf{3}. \textit{Calculation of $I_{1}$}. We claim to have the following formula: 
\begin{align}
&\frac{1}{2\varepsilon}\iint_{\Delta^{c}}(v(x)-v(y))\cdot \nabla_{x} V(x-y)\left(\rho_{\varepsilon,N}-1-\varepsilon \mathfrak{U}\right)^{\otimes 2}(\dd x\dd y) \notag\\
&=\frac{1}{N^{2}\varepsilon}\sum_{i\neq j}v(x_{i}) \cdot \nabla_{x} V(x_{i}-x_{j}) \notag\\
&\quad\,-\frac{1}{N}\sum_{i=1}^{N}v(x_{i})\cdot \nabla_{x} p(x_{i})+\int_{\mathbb{T}^{2}} \mathrm{div}_{x} (-\Delta_{x} )^{-1}(v\mathfrak{U})(x)\left(\rho_{\varepsilon,N}-1-\varepsilon \mathfrak{U}\right)(\dd x). \label{second variation formula} 
\end{align}
To prove \eqref{second variation formula}, we first note that 
\begin{align}
&\frac{1}{2\varepsilon}\iint_{\Delta^{c}}(v(x)-v(y))\cdot \nabla_{x}V(x-y)\rho_{\varepsilon,N}^{\otimes 2}(t,\dd x\dd y)\notag\\
&=\frac{1}{\varepsilon}\iint_{\Delta^{c}}v(x)\cdot \nabla_{x}V(x-y)\rho_{\varepsilon,N}^{\otimes 2}(t,\dd x\dd y) \notag\\
&=\frac{1}{\varepsilon N^{2}}\sum_{i\neq j}v(x_{i})\cdot \nabla_{x}V(x_{i}-x_{j}).    \label{pure product} 
\end{align}
In addition, we have 
\begin{align*}
&-\frac{1}{2\varepsilon}\iint_{\Delta^{c}}(v(x)-v(y))\cdot \nabla_{x}V(x-y)\rho_{\varepsilon,N}(\dd x)(1+\varepsilon \mathfrak{U})( y)\\
&-\frac{1}{2\varepsilon}\iint_{\Delta^{c}}(v(x)-v(y))\cdot \nabla_{x}V(x-y)(1+\varepsilon\mathfrak{U})(x)\rho_{\varepsilon,N}(\dd y)\\
&= -\frac{1}{\varepsilon}\iint_{\Delta^{c}}(v(x)-v(y))\cdot \nabla_{x}V(x-y)(1+\varepsilon \mathfrak{U})(x)\rho_{\varepsilon,N}(\dd y)\\
&=-\frac{1}{\varepsilon}\iint_{\Delta^{c}}v(x)\cdot \nabla_{x} V(x-y)(1+\varepsilon \mathfrak{U})(x)\rho_{\varepsilon,N}(\dd y)\nonumber\\
&\quad+\frac{1}{\varepsilon}\iint_{\Delta^{c}}v(y)\cdot \nabla_{x}V(x-y)(1+\varepsilon\mathfrak{U})(x)\rho_{\varepsilon,N}(\dd y).     
\end{align*}
Observe that 
\begin{align*}
&-\frac{1}{\varepsilon}\iint_{\Delta^{c}}v(x)\cdot \nabla_{x}V(x-y) (1+\varepsilon \mathfrak{U})(x)\rho_{\varepsilon,N}(\dd y)\\
&=-\iint_{\Delta^{c}}v(x)\cdot \nabla_{x}V(x-y)\mathfrak{U}(x)\rho_{\varepsilon,N}(\dd y)\\
&=\int \nabla_{x} V\ast (v\mathfrak{U})(y)\rho_{\varepsilon,N}(\dd y)=\int \mathrm{div}_{x}(-\Delta_{x})^{-1}(v\mathfrak{U})(x)\rho_{\varepsilon,N}(\dd x),  
\end{align*}
and that 
\begin{align*}
&\frac{1}{\varepsilon}\iint_{\Delta^{c}} v(y)\cdot \nabla_{x}V(x-y)(1+\varepsilon \mathfrak{U})(x)\rho_{\varepsilon,N}(\dd y)\\
&=-\int v(y)\cdot \nabla_{x} V\ast \mathfrak{U}(y)\rho_{\varepsilon,N}(\dd y)
=-\frac{1}{N}\sum_{i=1}^{N}v(x_{i})\cdot \nabla_{x}p(x_{i}).    
\end{align*}
Consequently, we obtain the identity:  
\begin{align}
&-\frac{1}{2\varepsilon}\iint_{\Delta^{c}}(v(x)-v(y))\cdot \nabla_{x}V(x-y)\rho_{\varepsilon,N}(\dd x)(1+\varepsilon \mathfrak{U})( y) \notag\\
&-\frac{1}{2\varepsilon}\iint_{\Delta^{c}}(v(x)-v(y))\cdot \nabla_{x}V(x-y)(1+\varepsilon\mathfrak{U})(x)\rho_{\varepsilon,N}(\dd y)\notag\\
&=-\frac{1}{N}\sum_{i=1}^{N}v(x_{i})\cdot \nabla_{x}p(x_{i})+\int \mathrm{div}_{x}(-\Delta_{x})^{-1}(v\mathfrak{U})(x)\rho_{\varepsilon,N}(\dd x).  \label{sec identity}   
\end{align}
Furthermore, we have
\begin{align}
&\frac{1}{2\varepsilon}\iint(v(x)-v(y))\cdot \nabla_{x}V(x-y)(1+\varepsilon \mathfrak{U})(x)(1+\varepsilon\mathfrak{U}(y))\ \dd x\dd y \notag\\
&=\frac{1}{\varepsilon}\iint v(x)\cdot \nabla_{x}V(x-y)(1+\varepsilon \mathfrak{U})(x)(1+\varepsilon\mathfrak{U})(y)\ \dd x\dd y \notag\\
&=\int v(x)\cdot \nabla_{x} V\ast \mathfrak{U}(x)(1+\varepsilon \mathfrak{U})(x)\ \dd x \notag\\
&=-\varepsilon\int \mathfrak{  U}(x)\nabla_{x}V\ast (v\mathfrak{U})(x)\ \dd x=-\int \mathrm{div}_{x}(-\Delta_{x})^
{-1}(v\mathfrak{U})(x)(1+\varepsilon \mathfrak{U})(x)\ \dd x. \label{third identity}
\end{align}
Gathering \eqref{pure product}--\eqref{third identity} 
yields \eqref{second variation formula}. It follows from \eqref{second variation formula} that 
\begin{align}
I_{1}=&  \frac{1}{2\varepsilon}\iint_{\Delta^{c}}(v(x)-v(y))\cdot \nabla_{x} V(x-y)\left(\rho_{\varepsilon,N}-1-\varepsilon \mathfrak{U}\right)^{\otimes 2}(\dd x\dd y) \notag\\
&+\frac{1}{N}\sum_{i=1}^{N}v(x_{i})\cdot \nabla_{x} p(x_{i})-\int_{\mathbb{T}^{2}} \mathrm{div}_{x} (-\Delta_{x})^{-1}(v\mathfrak{U})(x)\left(\rho_{\varepsilon,N}-1-\varepsilon\mathfrak{U}\right)(\dd x). \label{I1 formula}
\end{align}
Combining \eqref{I2-I8 formula} with \eqref{I1 formula}, we conclude  
\begin{align}
I_{1}+I=&-\frac{1}{N}\sum_{i=1}^{N}(v(x_{i})-\xi_{i})\cdot((v(x_{i})-\xi_{i})D_{x}v(x_{i}))
\notag\\
&+\frac{1}{2\varepsilon}\iint_{\Delta^{c}}(v(x)-v(y))\cdot \nabla_{x} V(x-y)\left(\rho_{\varepsilon,N}-1-\varepsilon \mathfrak{U}\right)^{\otimes 2}(\dd x\dd y) \notag \notag\\
&-\int_{\mathbb{T}^{2}} \mathrm{div}_{x} (-\Delta_{x})^{-1}(v\mathfrak{U})(x)\left(\rho_{\varepsilon,N}-1-\varepsilon\mathfrak{U}\right)(\dd x). \label{I1+I formula}
\end{align}

\smallskip
\textbf{4}. \textit{Calculation of $I_{5}+I_{7}$}. 
We have 
\begin{align}
I_{5}+I_{7}&=-\frac{1}{N}\sum_{i=1}^{N}\int_{\mathbb{T}^{2}} V(x_{i}-y)\partial_{t}\mathfrak{U}(y)\ \dd y +\iint_{\mathbb{T}^{2}\times \mathbb{T}^{2}} V(x-y)(1+\varepsilon \mathfrak{U}(x))\partial_{t}\mathfrak{U}(y)\ \dd x\dd y\notag\\
&= \int_{\mathbb{T}^{2}} V\ast \partial_{t}\mathfrak{U}(x)\left(1+\varepsilon\mathfrak{U}-\rho_{\varepsilon,N}\right)(\dd x)=-\int_{\mathbb{T}^{2}}\partial_{t}p(x)\left(\rho_{\varepsilon,N}-1-\varepsilon\mathfrak{U}\right)(\dd x). \label{I7+I9 formula}
\end{align}
The combination of \eqref{I1+I formula}--\eqref{I7+I9 formula} yields 
\begin{align*}
\frac{\dd}{\dd t}\mathcal{E}_{\varepsilon,N}(t)
=& -\frac{1}{N}\sum_{i=1}^{N}(v(x_{i})-\xi_{i})\cdot\big((v(x_{i})-\xi_{i})D_{x}v(x_{i})\big)
\notag\\
&+\frac{1}{2\varepsilon}\iint_{\Delta^{c}}(v(x)-v(y))\cdot \nabla_{x} V(x-y)\left(\rho_{\varepsilon,N}-1-\varepsilon \mathfrak{U}\right)^{\otimes 2}(\dd x\dd y) \notag \notag\\
&-\int_{\mathbb{T}^{2}} \mathrm{div}_{x} (-\Delta_{x})^{-1}(v\mathfrak{U})(x)\left(\rho_{\varepsilon,N}-1-\varepsilon\mathfrak{U}\right)(\dd x)\\
&-\int_{\mathbb{T}^{2}}\partial_{t}p(x)\left(\rho_{\varepsilon,N}-1-\varepsilon\mathfrak{U}\right)(\dd x):= \sum_{k=1}^{4}J_{k}. 
\end{align*}
This concludes \eqref{time derivative of modulated energy statement}.

\smallskip
\textbf{5}. \textit{Conclusion}.
We proceed by estimating separately each one of the $J_{k}$. 
The term $J_{1}$ is recast as 
\begin{align*}
J_{1}=-\frac{1}{N}\sum_{i=1}^{N}(v(x_{i})-\xi_{i})^{\otimes 2}:D^{\mathbf{s}}_{x}v(x_{i})    
\end{align*} 
so that
\begin{align}
J_{1}\leq \left\Vert D^{\mathbf{s}}_{x}v\right\Vert_{\infty}\frac{1}{N}\sum_{i=1}^{N} \left\vert \xi_{i}-v(t,x_{i})\right\vert^{2} \leq 2\left\Vert D^{\mathbf{s}}_{x}v\right\Vert_{\infty} \mathcal{K}_{\varepsilon,N}(t)\leq 2\left\Vert D^{\mathbf{s}}_{x}v\right\Vert_{\infty}\mathcal{E}_{\varepsilon,N}(t).\label{J1 est}
\end{align}
To estimate $J_{2}$, we invoke Theorem \ref{Commutator est} 
with $\rho=1+\varepsilon\mathfrak{U}$ in order to find 
\begin{align}
|J_{2}|&\leq C\left\Vert D_{x} v\right\Vert_{\infty}
\Big(\frac{1}{2\varepsilon}\iint_{\Delta^{c}}V(x-y)
(\rho_{\varepsilon,N}-1-\varepsilon\mathfrak{U})^{\otimes 2}(\dd x\dd y)\notag\\
&
\hspace{2.7cm}
+C\big(1+\left\Vert 1+\varepsilon\mathfrak{U}\right\Vert_{\infty}\big)
\frac{1+\log N}{2\varepsilon N}\Big) \notag\\
&\leq  C\left\Vert D_{x}v\right\Vert_{\infty}\mathcal{E}_{\varepsilon,N}(t).
\label{J2 est}
\end{align}
To estimate $J_{3}$,  
we invoke Proposition \ref{coercivity inequality} with $\varphi=-\mathrm{div}_{x}(-\Delta_{x})^{-1}(v\mathfrak{U})$ and $\rho=1+\varepsilon\mathfrak{U}$ in order to find
\begin{align}
J_{3}&\leq C\frac{\left\Vert\nabla_{x}\mathrm{div}_{x}(-\Delta_{x})^{-1}(v\mathfrak{U})\right\Vert_{\infty}}{\sqrt{N}} \notag\\
&\quad+\left\Vert \nabla_{x}\mathrm{div}_{x}(-\Delta_{x})^{-1}(v\mathfrak{U})\right\Vert_{L^{\infty}_{t}L^{2}_{x}}
\Big(2\varepsilon\mathcal{V}_{\varepsilon,N}(t)+C(2+\varepsilon\left\Vert \mathfrak{U}\right\Vert_{L^{\infty}_{t}L^{\infty}_{x}})\,\frac{1+\log N}{N}\Big)^{\frac{1}{2}} \notag\\
&\leq\, C\frac{\left\Vert\nabla_{x}\mathrm{div}_{x}(-\Delta_{x})^{-1}(v\mathfrak{U})\right\Vert_{\infty}}{\sqrt{N}}+\varepsilon\left\Vert \nabla_{x}\mathrm{div}_{x}(-\Delta)^{-1}(v\mathfrak{U})\right\Vert_{L^{\infty}_{t}L^{2}_{x}}^{2} \notag\\
&\quad+2\mathcal{V}_{\varepsilon,N}(t)+C\frac{2+\varepsilon\left\Vert \mathfrak{U}\right\Vert_{\infty}}{\varepsilon}\,\frac{1+\log N}{N}.
\label{est J3} 
\end{align}
Similarly, applying Proposition \ref{coercivity inequality} with $\varphi=\partial_t p$ 
and $\rho=1+\varepsilon\mathfrak{U}$, we obtain
\begin{align}
J_{4}&\leq
C\frac{\left\Vert \partial_{t}\nabla_{x}p\right\Vert_{\infty}}{\sqrt{N}}
+\left\Vert \partial_{t}\nabla_{x}p\right\Vert_{L^{\infty}_{t}L^{2}_{x}}
\Big(\mathcal{V}_{N}(1+\varepsilon\mathfrak{U},\rho_{\varepsilon,N})
+C\big(1+\left\Vert 1+\varepsilon \mathfrak{U}\right\Vert_{\infty} \big)
  \frac{1+\log N}{N})\Big)^{\frac{1}{2}} 
 \notag\\
 &\leq C \frac{\left\Vert \partial_{t}\nabla_{x}p\right\Vert_{\infty} }{\sqrt{N}}+\frac{\varepsilon}{2}\left\Vert \partial_{t}\nabla_{x}p\right\Vert^{2}_{L^{\infty}_{t}L^{2}_{x}}+\frac{1}{2\varepsilon}\mathcal{V}_{N}(1+\varepsilon\mathfrak{U},\rho_{\varepsilon,N})+C\frac{1+\left\Vert 1+\varepsilon\mathfrak{U}\right\Vert_{\infty}}{\varepsilon}\frac{1+\log N}{N}.
\label{J4 est}  
\end{align}
Gathering \eqref{J1 est}--\eqref{J4 est} and utilizing Lemma  \ref{basic estimates},
we conclude 
\begin{align*}
\frac{\dd}{\dd t}\mathcal{E}_{\varepsilon,N}(t)\lesssim \mathcal{E}_{\varepsilon,N}(t)+\varepsilon+\frac{\log N}{N\varepsilon}+\frac{1}{\sqrt{N}}.    
\end{align*}
Therefore, it follows by Gr\"onwall's lemma that 
\begin{align*}
\underset{t\in [0,T]}{\sup}\mathcal{E}_{\varepsilon,N}(t)\underset{N\rightarrow \infty}{\rightarrow}0 \ \qquad \mbox{provided}\ \mathcal{E}_{\varepsilon,N}(0)\underset{N\rightarrow \infty}{\longrightarrow}0 \ \mbox{and}\ \frac{\log N}{\varepsilon N}\underset{N\rightarrow \infty}{\longrightarrow}0.  \end{align*}
\end{proof}

\smallskip
\textit{Proof of Theorem \ref{main thm renormalized}}.
Fix a test function $\varphi\in W^{1,\infty}(\mathbb{T}^{2})$ with $\left\Vert \varphi\right\Vert_{W^{1,\infty}}\leq 1$. 
By means of Proposition \ref{coercivity inequality}, we have 
\begin{align*}
\left|\int_{\mathbb{T}^{2}}\varphi(x)(\rho_{\varepsilon,N}-1)(t,\dd x)\right|\leq \frac{C}{\sqrt{N}}+C\varepsilon+\sqrt{2\varepsilon}\sqrt{\mathcal{E}_{\varepsilon,N}(t)}.
\end{align*}
Maximizing the previous inequality  over all $\varphi\in W^{1,\infty}(\mathbb{T}^{2})$  
and using Theorem \ref{time derivative of modulated energy statement}, we conclude 
\begin{align}
\underset{t\in [0,T]}{\sup}W_{1}(\rho_{\varepsilon,N}(t,\cdot),1)\underset{N\rightarrow \infty} {\longrightarrow} 0.  
\label{convergence-rhoN}   
\end{align}
Next, fix a test vector field $b\in W^{1,\infty}(\mathbb{T}
^{2};\mathbb{R}^{2})$ with $\left\Vert b\right\Vert_{W^{1,\infty}(\mathbb{T}^{2};\mathbb{R}^{2})}\leq 1 $. 
Then we have 
\begin{align*}
\int_{\mathbb{T}^{2}}b\cdot (J_{\varepsilon,N}-v)(t,\dd x)
&=\int_{\mathbb{T}^{2}}b\cdot\Big(\frac{1}{N}\sum_{i=1}^{N}\xi_{i}\delta_{x_{i}}-\frac{1}{N}\sum_{i=1}^{N}\delta_{x_{i}}v\Big)(t,\dd x)\\
&\quad+\int_{\mathbb{T}^{2}}b\cdot v\Big(\frac{1}{N}\sum_{i=1}^{N}\delta_{x_{i}}-1\Big)(t,\dd x)
:= \mathcal{I}(t)+\mathcal{J}(t).     
\end{align*}
To estimate $\mathcal{I}$, we apply the Cauchy-Schwarz inequality in order to find 
\begin{align}
|\mathcal{I}|=\Big|\frac{1}{N}\sum_{i=1}^{N}b(x_{i})\cdot (\xi_{i}-v(x_{i}))\Big|
&\leq \frac{1}{\sqrt{N}}\Big(\sum_{i=1}^{N}\left\vert b\right\vert ^{2}(x_{i})\Big)^{\frac{1}{2}}\frac{1}{\sqrt{N}}\Big(\sum_{i=1}^{N}\left\vert \xi_{i}-v(x_{i})\right\vert^{2}\Big)^{\frac{1}{2}} \notag\\
&\leq \left\Vert b\right\Vert_{\infty}\sqrt{2\mathcal{E}_{\varepsilon,N}(t)}. \label{I vanishing}  \end{align}
Therefore, we deduce 
\begin{align*}
\underset{t\in [0,T]}{\sup}|\mathcal{I}(t)|\leq  \underset{t\in [0,T]}{\sup}\sqrt{2\mathcal{E}_{\varepsilon,N}(t)}{\underset{N\rightarrow \infty}{\longrightarrow}0}.
\end{align*}
To estimate $\mathcal{J}$, we apply duality to obtain 
\begin{align}
\underset{t\in [0,T]}{\sup}|\mathcal{J}(t)|\leq \left\Vert b\cdot v\right\Vert_{L^{\infty}_{t}W^{1,\infty}_{x}}\underset{t\in [0,T]}{\sup}W_{1}(\rho_{\varepsilon,N}(t,\cdot),1)\underset{N\rightarrow \infty}{\longrightarrow}0
\label{J vanishing}
\end{align}
by \eqref{convergence-rhoN}. 
Combining \eqref{I vanishing}--\eqref{J vanishing}, we conclude 
\begin{align*}
\underset{t\in [0,T]}{\sup}\left\Vert (J_{\varepsilon,N}-v\,\dd x)(t,\cdot)\right\Vert_{\mathrm{BL}^{*}}\underset{N\rightarrow \infty}{\longrightarrow} 0,
\end{align*}
where the bounded Lipschitz dual norm is the one defined above.
\qed

\section{Well-Posedness of Problem \eqref{trajectories intro}}
\label{wellposed-trajectories-sec}
In this section, we establish the existence and uniqueness of global solutions to \eqref{trajectories intro}. The proof proceeds by first establishing uniform separation of the particles and then applying a blow-up alternative. The presence of a general external magnetic field requires additional pointwise control of the particle velocities, which ensures that the system remains globally Lipschitz along the flow (rather than only locally Lipschitz). Note also that the system is non-homogeneous since $B$ is time dependent. 
We begin by recalling that the $2$-D periodic Coulomb potential is a smooth local perturbation of its whole-space counterpart. More precisely,
$
V\in C^\infty(\mathbb T^2\setminus\{0\}),
$
and, defining
$
V_{\mathbb R^2}(x):=-\frac{1}{2\pi}\log|x|,
$
there exists an even function
$
V_{\mathrm{loc}}\in C^\infty\big((\mathbb R^2\setminus\mathbb Z^2)\cup\{0\}\big)
$
such that
\begin{align*}
V(x)
=V_{\mathbb{R}^{2}}(x)+V_{\mathrm{loc}}(x) \qquad \mbox{for all}\ x\in \mathbb{T}^{2}\setminus \{0\}.     
\end{align*}
We start by demonstrating the local well-posedness. 

\begin{prop}[Local well-posedness]\label{Short time well posed}
Suppose that
\begin{itemize}
\item[{\rm(i)}]$X_{N}^{0}=(x_{1}^{N,0},\cdots\!,x_{N}^{N,0})\in \Delta_{N}^{c}$ and $\Xi_{N}^{0}=(\xi_{1}^{N,0},\cdots\!,\xi_{N}^{N,0})\in \mathbb{R}^{2N};$
\smallskip
\item[{\rm(ii)}] $B\in C(\mathbb{R}_{+};W^{1,\infty}(\mathbb{T}^{2}))$. 
\end{itemize}
Then there exist $T_{\ast}>0$ such that \eqref{trajectories intro} 
has a unique solution 
$$
(X_{N}(t),\Xi_{N}(t))\in C^{1}([0,T^{\ast}];\mathbb{T}^{2N}\times \mathbb{R}^{2N}).
$$ 
\end{prop}

\begin{proof}
Let $\Psi\in C^{\infty}_{0}(\mathbb{R})$ be an even function such that $\Psi(r)=1$ for $\left| r \right|\leq \frac{1}{8}$ and $\Psi(r)=0$ for $\left| r \right|\geq \frac{1}{4}$. 
Let $\psi\in C^{\infty}(\mathbb{T}^{2})$ be the radial function defined 
by $\psi(x) := \Psi(\left| x \right|)$. Set  
\[
\psi_{\eta}(x):=\Psi(\frac{\left|x\right|^{2}}{\eta^{2}}), \qquad\ 
V_{\eta}(x):= V_{\mathrm{loc}}(x)+(1-\psi_{\eta}(x))V_{\mathbb{R}^{2}}(x). 
\]
For brevity, in what follows we omit superscript $N$ from $x_{i,\eta}^{N}$ and $\xi_{i,\eta}^{N}$. Consider the regularized system: 
\begin{align}
\begin{cases}
\dot x_{i,\eta}(t)=\xi_{i,\eta}(t),\ &x_{i,\eta}(0)=x_{i}^{0}\\[1mm]
\dot \xi_{i,\eta}(t)=-B(t,x_{i,\eta}(t))\xi_{i,\eta}^{\perp}(t)-\frac{1}{N\varepsilon}\underset{j:j\neq i}{\sum}\nabla_{x} V_{\eta}(x_{i,\eta}(t)-x_{j,\eta}(t)), \,\,\, &\xi_{i,\eta}(0)=\xi_{i}^{0}. 
\end{cases}
\label{regularized system}
\end{align}
Our main objective is to obtain a lower bound on the separation of the particles, {\it i.e.,} the quantity $\underset{i\neq j}{\min}\left\vert x_{i,\eta}(t)-x_{j,\eta}(t)\right\vert$. 
Using Theorem \ref{periodiccauchy lip thm} in the appendix, we can show that there exist a  global solution $(x_{1,\eta},\cdots\!,x_{N,\eta},\xi_{1,\eta},\cdots\!,\xi_{N,\eta})\in C^{1}(\mathbb{R}_{+};\mathbb{T}^{2N}\times \mathbb{R}^{2N})$ to \eqref{regularized system}. Note that 
\begin{align*}
\nabla_{x} V_{\eta}(x)=\nabla_{x}V_{\mathrm{loc}}(x)-2\Psi'(\frac{\left\vert x\right\vert^{2}}{\eta^{2}})\frac{x}{\eta^{2}}V_{\mathbb{R}^{2}}(x)+(1-\psi_{\eta}(x))\nabla_{x} V_{\mathbb{R}^{2}}(x). \end{align*}
Moreover, observe the estimates: 
\begin{align}
\left\vert (1-\psi_{\eta}(x))\nabla_{x} V_{\mathbb{R}^{2}}(x)\right\vert\lesssim \frac{1+\eta}{\eta}, \label{estimate for cutoff}       
\end{align}
and 
\begin{align}
\Big\vert 2\Psi'(\frac{\left\vert x\right\vert^{2}}{\eta^{2}})\frac{x}{\eta^{2}}V_{\mathbb{R}^{2}}(x)\Big\vert
\lesssim -\frac{\log(\eta)}{\eta}\lesssim \frac{1}{\eta^{2}}. \label{est for cutoff 2}      
\end{align}
We proceed by estimating $\left|x_{i,\eta}(t)-x_{i}^{0}\right|$. Integrating in time the equation for $x_{i}$ in \eqref{regularized system} and using Remark \ref{rem about conse of energy},
we obtain 
\begin{align}
\left\vert x_{i,\eta}(t)-x_{i}^{0}\right\vert 
&\leq \int_{0}^{t}\left\vert \xi_{i,\eta}(\tau)\right\vert\ \dd \tau  
\leq \Big(\int_{0}^{t}\left\vert \xi_{i}(\tau)\right\vert^{2}\ \dd\tau\Big)^{\frac{1}{2}}
\Big(\int_{0}^{t} \dd \tau\Big)^{\frac{1}{2}}\notag\\
&\leq \sqrt{T}\Big(2N\int_{0}^{t}\mathcal{F}_{\varepsilon,\eta,N}(\tau)\ \dd \tau\Big)^{\frac{1}{2}}
\leq \sqrt{2N\mathcal{F}_{\varepsilon,N}(0)}\,T, \label{6.4a}    
\end{align}
where we have denoted by $\mathcal{F}_{\varepsilon,\eta,N}(t)$ 
the total energy of the regularized system \eqref{regularized system}, defined with the non-negative kernel $V_\eta+C_0$,
and have used the fact that it is bounded by the initial energy of the original 
system \eqref{trajectories intro} because $V_{\eta}\leq V$. 
Thanks to \eqref{6.4a},
we have the bound 
\begin{align}
\left\vert x_{i,\eta}(t)-x_{j,\eta}(t)\right\vert&\geq \left\vert x_{i}^{0}-x_{j}^{0}\right\vert-\left\vert x_{i}^{0}-x_{i,\eta}(t)\right\vert-\left\vert x_{j}^{0}-x_{j,\eta}(t)\right\vert \notag\\
&\geq \underset{i\neq j}{\min}\left\vert x_{i}^{0}-x_{j}^{0}\right\vert-2\sqrt{2N\mathcal{F}_{\varepsilon,N}(0)}T.  \label{bound on distance from below}\end{align}
Now,  choose $T_{\ast}>0$ and $\eta>0$ as follows:  
\begin{itemize}
\item[{\rm(a)}]$\eta=\frac{1}{2}\underset{i\neq j}{\min}\left\vert x_{i}^{0}-x_{j}^{0}\right\vert;$
\item[{\rm(b)}] $T_{\ast}\leq \frac{\min_{i\neq j}\left\vert x_{i}^{0}-x_{j}^{0}\right\vert }{4\sqrt{2N\max\{1,\mathcal{F}_{\varepsilon,N}(0)\}}}.$  
\end{itemize}
In view of \eqref{bound on distance from below}, it follows that 
\begin{align*}
\min_{i\neq j}\left\vert x_{i,\eta}(t)-x_{j,\eta}(t)\right\vert\geq\eta\qquad \mbox{for all}\ t\in [0,T_{\ast}].      
\end{align*}
Since $\nabla_{x}V_{\eta}(x)\equiv\nabla_{x}V(x)$ on $\left\vert x\right\vert \geq \frac{\eta}{2}$, 
it follows that, for all $t\in [0,T_{\ast}]$,
\begin{align*}
\nabla_{x}V_{\eta}(x_{i,\eta}(t)-x_{j,\eta}(t))=\nabla_{x}V(x_{i,\eta}(t)-x_{j,\eta}(t))    
\end{align*}
which shows that 
$(x_{1,\eta}(t),\cdots\!,x_{N,\eta}(t),\xi_{1,\eta}(t),\cdots\!,\xi_{N,\eta}(t))$ is a solution to \eqref{trajectories intro} on $[0,T_{\ast}]$. On the open set where all pairwise distances exceed $\frac{\eta}{2}$, the vector field is locally Lipschitz; hence this solution is unique on $[0,T_*]$.
\end{proof}

\begin{lem}\label{lemma:6.2}
Let $X_{N}^{0}\in  \Delta_{N}^{c}$ and $\Xi_{N}^{0}\in \mathbb{R}^{2N}$, 
and let $(X_{N}(t),\Xi_{N}(t))\in C^{1}([0,T];\mathbb{T}^{2N}\times \mathbb{R}^{2N})$ 
be a solution to \eqref{trajectories intro} on $[0,T]$. Then
\begin{enumerate}
\item[{\rm(i)}] There exists  some constant  
$\delta=\delta(\varepsilon,N,\mathcal{F}_{\varepsilon,N}(0))>0$ such that 
\begin{align}
\min_{i\neq j}\left\vert x_{i}(t)-x_{j}(t)\right\vert\geq \delta 
\qquad\mbox{\ for all} \ t\in [0,T].  \label{bound from below on separation}    
\end{align}
\item[{\rm(ii)}] There exists 
$H_{0}=H_{0}(\varepsilon,T,\left\vert \Xi_{N}^{0}\right\vert,\delta )$ such that
\begin{align}
\underset{1\leq i\leq N}{\sup}\underset{t\in [0,T]}{\sup}\left\vert \xi_{i}(t)\right\vert\leq H_{0}. \label{H0 est} 
\end{align}
\end{enumerate}
\end{lem}

\begin{proof} We divide the proof into two steps.

\smallskip
\textbf{1}. Given indices $i_{0},j_{0}$ and $t\in [0,T]$, 
we may assume for simplicity that $\left\vert x_{i_{0}}(t)-x_{j_{0}}(t)\right\vert \leq \frac{1}{4}$. 
We split the interaction energy as follows:  
\begin{align*}
\frac{1}{2N^{2}\varepsilon}\sum_{i\neq j}V(x_{i}(t)-x_{j}(t))
=\frac{1}{2N^{2}\varepsilon}\sum_{i\neq j}V_{\mathbb{R}^{2}}(x_{i}(t)-x_{j}(t))
+\frac{1}{2N^{2}\varepsilon}\sum_{i\neq j}V_{\mathrm{loc}}(x_{i}(t)-x_{j}(t)).    
\end{align*}
Therefore, using the conservation of energy (Lemma \ref{Conservation of discrete energy}), 
we obtain 
\begin{align*}
&\frac{1}{2N^{2}\varepsilon}\sum_{i\neq j}V_{\mathbb{R}^{2}}(x_{i}(t)-x_{j}(t))+\frac{1}{2N^{2}\varepsilon}\sum_{i\neq j}V_{\mathrm{loc}}(x_{i}(t)-x_{j}(t))
\leq \mathcal{F}_{\varepsilon,N}(0),     
\end{align*}
and hence  
\begin{align*}
\frac{1}{2N^{2}\varepsilon}\sum_{i\neq j} V_{\mathbb{R}^{2}}(x_{i}(t)-x_{j}(t))\leq \mathcal{F}_{\varepsilon,N}(0)+\frac{N-1}{N\varepsilon}\left\Vert V_{\mathrm{loc}}\right\Vert_{\infty}.  
\end{align*}
It follows that 
\begin{align}
&\frac{1}{2N^{2}\varepsilon}\underset{\left\vert x_{i}(t)-x_{j}(t)\right\vert\leq 1}{\sum_{i\neq j}}V_{\mathbb{R}^{2}}(x_{i}(t)-x_{j}(t))\notag\\
&\leq \mathcal{F}_{\varepsilon,N}(0)+\frac{N-1}{N\varepsilon}\left\Vert V_{\mathrm{loc}}\right\Vert_{\infty} 
- \frac{1}{2N^{2}\varepsilon}\underset{\left\vert x_{i}(t)-x_{j}(t)\right\vert>1 }{\sum_{i\neq j}}V_{\mathbb{R}^{2}}(x_{i}(t)-x_{j}(t)). \label{est on sum near 0}   
\end{align}
For centered representatives, $\left\vert x_{i}(t)-x_{j}(t)\right\vert\leq \frac{\sqrt{2}}{2}<1$; hence
\begin{align}
\bigg\vert \frac{1}{2N^{2}\varepsilon}\underset{\left\vert x_{i}(t)-x_{j}(t)\right\vert>1 }{\sum_{i\neq j}}V_{\mathbb{R}^{2}}(x_{i}(t)-x_{j}(t))\bigg\vert =0.\label{estimate of sum far from 0}
\end{align}
Putting \eqref{estimate of sum far from 0} inside \eqref{est on sum near 0}, we conclude 
\begin{align*}
-\frac{1}{4\pi N^{2}\varepsilon}\log(\left\vert x_{i_{0}}(t)-x_{j_{0}}(t)\right\vert)\leq \mathcal{F}_{\varepsilon,N}(0)+\frac{N-1}{N\varepsilon}\left\Vert V_{\mathrm{loc}}\right\Vert_{\infty}.       
\end{align*}
Thus, the inequality \eqref{bound from below on separation} follows by taking 
\begin{align*}
\delta=\min\big\{\frac14,
\exp\big(-4\pi N^{2}\varepsilon(\mathcal{F}_{\varepsilon,N}(0)+\frac{N-1}{N\varepsilon}\Vert V_{\mathrm{loc}}
\Vert_{L^{\infty}([-\frac12,\frac12]^2)})\big)\big\}.
\end{align*}

\smallskip
2. Multiplying the second equation in \eqref{trajectories intro} by $\xi_{i}$ yields
\begin{align*}
\frac{\dd}{\dd t}\frac{1}{2}\left\vert \xi_{i}(t)\right\vert^{2} &=-\frac{1}{N\varepsilon}\sum_{j:j\neq i} \nabla_{x}V(x_{i}(t)-x_{j}(t))\cdot \xi_{i}(t)\leq \frac{\left\vert \xi_{i}(t)\right\vert }{N\varepsilon}\sum_{j:j\neq i}(\frac{1}{\delta}+\left\Vert \nabla_{x}V_{\mathrm{loc}}\right\Vert_{\infty})\\
&\leq \frac{1}{\varepsilon}(\frac{1}{\delta}+\left\Vert \nabla_{x}V_{\mathrm{loc}}\right\Vert_{\infty})\left\vert \xi_{i}(t)\right\vert=C\left\vert \xi_{i}(t)\right\vert.    
\end{align*}
Hence, the above inequality yields
\begin{align*}
\underset{t\in [0,T]}{\sup}\underset{1\leq i\leq N}{\sup}\left\vert\xi_{i}(t) \right\vert \leq \max_i|\xi_i^0|+CT,
\end{align*}
and so the claim follows by taking $H_{0}=\max_i|\xi_i^0|+CT$. 
\end{proof}

We can now conclude the global well-posedness by a blow-up alternative argument. 

\begin{thm}
Suppose that $X_{N}^{0}=(x_{1}^{N,0},\cdots\!,x_{N}^{N,0})\in \Delta_{N}^{c}$,
$\Xi_{N}^{0}=(\xi_{1}^{N,0},\cdots\!,\xi_{N}^{N,0})\in \mathbb{R}^{2N}$, 
and $B\in C(\mathbb{R}_{+};W^{1,\infty}(\mathbb{T}^{2}))$. 
Then there exists a unique global solution $(X_{N}(t),\Xi_{N}(t))\in C^{1}(\mathbb{R}_{+}; \mathbb{T}^
{2N}\times \mathbb{R}^{2N})$ to \eqref{trajectories intro}.    
\end{thm}
\begin{proof}  We divide into two steps.

\smallskip
\textbf{1. Existence}.
By 
Theorem \ref{Short time well posed}, there exist some $T_{\ast}>0$ and a solution $(X_{N},\Xi_{N})\in C^{1}([0,T_{\ast}];\mathbb{T}^{2N}\times \mathbb{R}^{2N})$ to \eqref{trajectories intro}. We proceed by induction. Suppose that we have a solution $(X_{N},\Xi_{N})\in C^{1}([0,nT_{\ast}];\mathbb{T}^{2N}\times \mathbb{R}^{2N})$ for \eqref{trajectories intro}. Consider the  Cauchy problem 
\begin{align}
\begin{cases}
\dot y_{i}(t)=\lambda_{i}(t),\\[1mm]
\dot \lambda_{i}(t)=-B(t+nT_{\ast},y_{i}(t))\lambda_{i}^{\perp}(t)-\frac{1}{N\varepsilon}\underset{j:j\neq i}{\sum}\nabla_{x}V(y_{i}(t)-y_{j}(t)),\\
\ y_{i}(0)=x_{i}(nT_{\ast}),\quad \lambda_{i}(0)=\xi_{i}(nT_{\ast}). 
\end{cases} \label{iterative step}   
\end{align}
By Lemma \ref{lemma:6.2}, we have
\begin{align*}
\min_{i\neq j}\left\vert y_{i}(0)-y_{j}(0)\right\vert=\min_{i\neq j}\left\vert x_{i}(nT_{\ast})-x_{j}(nT_{\ast})\right\vert\geq \delta(\varepsilon,N,\mathcal{F}_{\varepsilon,N}(0)).      
\end{align*}
Therefore, by Proposition \ref{Short time well posed}, 
it follows that there exists a solution 
$(Y_{N},\Lambda_{N})\in C^{1}([0,T_{\ast}];\mathbb{T}^{2N}\times \mathbb{R}^{2N})$ 
to \eqref{iterative step} 
(note that we can take the same $T_{\ast}$ because it depends only on the initial separation $\delta$ and the initial energy $\mathcal{F}_{\varepsilon,N}(0)$). Set 
\begin{align*}
Z_{N}(t):=\begin{cases}
X_{N}(t) &\quad \mbox{for $t\in [0,nT_{\ast}]$},\\
Y_{N}(t-nT_{\ast})&\quad \mbox{for $t\in [nT_{\ast},(n+1)T_{\ast}]$},
\end{cases}
\end{align*}
and 
\begin{align*}
\Pi_{N}(t):=\begin{cases}
\Xi_{N}(t)&\quad\mbox{for $t\in [0,nT_{\ast}]$},\\
\Lambda_{N}(t-nT_{\ast}) &\quad\mbox{for $t\in [nT_{\ast},(n+1)T_{\ast}]$}. 
\end{cases}    
\end{align*}
Then it is readily checked that 
\begin{itemize}
\item[{\rm(a)}] $(Z_{N},\Pi_{N})\in C^{1}([0,(n+1)T_{\ast}];\mathbb{T}^{2N}\times \mathbb{R}^{2N});$
\smallskip
\item[{\rm(b)}] $(Z_{N},\Pi_{N})$ is a solution to \eqref{trajectories intro} on $[0,(n+1)T_{\ast}]$.
\end{itemize}
This concludes the induction step, and hence it follows that, for every $n\in \mathbb{N}$, 
there exists a solution on $[0,nT_{\ast}]$ to \eqref{trajectories intro}. 

\smallskip
\textbf{2. Uniqueness}. Let $(X_{N},\Xi_{N})\in C^{1}([0,T];\mathbb{T}^{2N}\times \mathbb{R}^{2N})$ and $(Y_{N},\Lambda_{N})\in C^{1}([0,T];\mathbb{T}^{2N}\times \mathbb{R}^{2N})$ be two solutions of \eqref{trajectories intro} with the same initial data $(X_{N}^{0},\Xi_{N}^{0})$. Denote 
\begin{align*}
d_{\min}=\min\Big\{\underset{t\in [0,T]}{\min}\underset{i\neq j}{\min}\left\vert x_{i}(t)-x_{j}(t)\right\vert,\underset{t\in[0,T]}{\min}\min_{i\neq j} \left\vert y_{i}(t)-y_{j}(t)\right\vert\Big\}.     
\end{align*}
Note that $(X_{N},\Xi_{N})$ and $(Y_{N},\Lambda_{N})$ satisfy the truncated equation \eqref{regularized system} with $\eta=d_{\min}$. 
We write system \eqref{regularized system} concisely as 
\begin{align*}
\frac{\dd}{\dd t}(X_{N}(t),\,\Xi_{N}(t))
=(F_{1}(t,X_{N}(t),\Xi_{N}(t)),\,F_{2}(t,X_{N}(t),\Xi_{N}(t)))  \end{align*}
where $F_{1}:\mathbb{R}_{+}\times \mathbb{T}^{2N}\times \mathbb{R}^{2N}\rightarrow \mathbb{R}^{2N}$ and $F_{2}:\mathbb{R}_{+}\times \mathbb{T}^{2N}\times \mathbb{R}^{2N}\rightarrow \mathbb{R}
^{2N}$ are the vector fields defined by 
\begin{align*}
F_{1}(t,X_{N},\Xi_{N}):= (\xi_{1},\cdots\!,\xi_{N}),    
\end{align*}
and 
\begin{align*}
F_{2}(t,X_{N},\Xi_{N}):=\Big(-\xi_{i}^{\perp}B(t,x_{i})-\frac{1}{N\varepsilon}\sum_{j:j\neq i}\nabla_{x}V_{\eta}(x_{i}-x_{j})\Big)_{i=1}^{N}.
\end{align*}
We have 
\begin{align*}
&\frac{\dd}{\dd t}(\left\vert X_{N}(t)-Y_{N}(t)\right\vert^{2}+\left\vert \Xi_{N}(t)-\Lambda_{N}(t)\right\vert^{2})\\
&=2(X_{N}(t)-Y_{N}(t))\cdot (F_{1}(t,X_{N}(t),\Xi_{N}(t))-F_{1}(t,Y_{N}(t),\Lambda_{N}(t)))\\
&\quad+2(\Xi_{N}(t)-\Lambda_{N}(t))\cdot (F_{2}(t,X_{N}(t),\Xi_{N}(t))-F_{2}(t,Y_{N}(t),\Lambda_{N}(t))):= \mathcal{T}_{1}+\mathcal{T}_{2}.    \end{align*}
For brevity, we omit the time variable in what follows. We estimate $\mathcal{T}_{1}$:\begin{align}
\mathcal{T}_{1}\leq 2\left\vert X_{N}-Y_{N}\right\vert\left\vert \Xi_{N}-\Lambda_{N}\right\vert\leq \left\vert X_{N}-Y_{N}\right\vert^{2}+\left\vert \Xi_{N}-\Lambda_{N}\right\vert^{2}.  \label{T1 estimate}        
\end{align}
To estimate $\mathcal{T}_{2}$,  start by noting that 
\begin{align*}
|\mathcal{T}_{2}|&\leq 2\left\vert \Xi_{N}-\Lambda_{N}\right\vert \left\vert F_{2}(X_{N},\Xi_{N})-F_{2}(Y_{N},\Lambda_{N})\right\vert\\
&\leq 2\left\vert \Xi_{N}-\Lambda_{N}\right\vert \left\vert F_{2}(X_{N},\Xi_{N})-F_{2}(X_{N},\Lambda_{N})\right\vert+2\left\vert \Xi_{N}-\Lambda_{N}\right\vert \left\vert F_{2}(X_{N},\Lambda_{N})-F_{2}(Y_{N},\Lambda_{N})\right\vert\\
&:=\mathcal{T}_{2}^{1}+\mathcal{T}_{2}^{2}. 
\end{align*}
We have 
\begin{align*}
\left\vert F_{2}(X_{N},\Xi_{N})-F_{2}(X_{N},\Lambda_{N})\right\vert
\leq \left\Vert B\right\Vert_{\infty}\big\vert (\xi_{1}^{\perp},\cdots\!,\xi_{N}^{\perp})-(\lambda_{1}^{\perp},\cdots\!,\lambda_{N}^{\perp})\big\vert
=\left\Vert B\right\Vert_{\infty}\left\vert \Xi_{N}-\Lambda_{N}\right\vert,          
\end{align*}
so that
\begin{align}
\mathcal{T}_{2}^{1}\leq  2\left\Vert B\right\Vert_{\infty} \left\vert \Xi_{N}-\Lambda_{N}\right\vert^{2}.\label{est on T12}     
\end{align}
We introduce the notation 
\begin{itemize}
  \item[{\rm(a)}]$L=L(d_{\min})=\mathrm{Lip}(\nabla_{x}V_{\eta})$.
   \item[{\rm(b)}] $H_{0}$ as in \eqref{H0 est}. 
\end{itemize}
With this notation, we estimate 
\begin{align*}
&\left\vert F_{2}(X_{N},\Lambda_{N})-F_{2}(Y_{N},\Lambda_{N})\right\vert\\
&\leq H_{0}\left\Vert \nabla_xB\right\Vert_{L^{\infty}([0,T]\times\mathbb T^2)}\left\vert X_{N}-Y_{N}\right\vert
+\frac{L}{N\varepsilon}
\bigg(\sum_{i=1}^{N}\Big(\sum_{j\ne i}
\bigl(|x_i-y_i|+|x_j-y_j|\bigr)\Big)^2\bigg)^{1/2}\\
&\leq \big(H_{0}\left\Vert \nabla_xB\right\Vert_{L^{\infty}([0,T]\times\mathbb T^2)}+\frac{2L}{\varepsilon}\big)\left\vert X_{N}-Y_{N}\right\vert.
\end{align*}
It follows that 
\begin{align}
\mathcal{T}^{2}_{2}\leq \big(H_{0}\left\Vert B\right\Vert_{L^{\infty}_{t}W^{1,\infty}_{x}} +\frac{2L}{\varepsilon}\big)\big(\left\vert X_{N}-Y_{N}\right\vert^{2}+\left\vert \Xi_{N}-\Lambda_{N}\right\vert^{2}\big). \label{est on T22} 
\end{align}
Gathering \eqref{est on T12}--\eqref{est on T22}, we have
\begin{align}
\mathcal{T}_{2}
\leq \big(H_{0}\left\Vert B\right\Vert_{L^{\infty}_{t}W^{1,\infty}_{x}}+\frac{2L}{\varepsilon}+2\left\Vert B\right\Vert_{\infty}\big)\big(\left\vert X_{N}-Y_{N}\right\vert^{2}+\left\vert \Xi_{N}-\Lambda_{N}\right\vert^{2}\big). \label{T2 estimate}  
\end{align}
To conclude, combining \eqref{T1 estimate}--\eqref{T2 estimate}  yields the following estimate 
for some constant $C>0$: 
\begin{align*}
\frac{\dd}{\dd t}(\left\vert X_{N}(t)-Y_{N}(t)\right\vert^{2}+\left\vert \Xi_{N}(t)-\Lambda_{N}(t)\right\vert^{2})\leq C(\left\vert X_{N}(t)-Y_{N}(t)\right\vert^{2} +\left\vert \Xi_{N}(t)-\Lambda_{N}(t)\right\vert^{2} ).     
\end{align*}
Then the uniqueness follows directly by Gr\"onwall's lemma. 
\end{proof}

\appendix
\section{Global Well-Posedness of System \eqref{regularized system}}
In this appendix we elaborate on the global well-posedness of the system of trajectories \eqref{regularized system}, as it falls slightly outside the scope of the classical Cauchy-Lipschitz theorem.  
\begin{thm}\label{periodiccauchy lip thm}
Suppose that $(X_{N}^{0},\Xi_{N}^{0})\in \mathbb{T}^{2N}\times \mathbb{R}^{2N}$, 
$V\in W^{2,\infty}(\mathbb{R}^{2})$, and $B\in C(\mathbb{R}_{+};W^{1,\infty}(\mathbb{R}^{2})),$ 
where $V$ and $B(t,\cdot)$ are $\mathbb{Z}^{2}$-periodic.
\begin{enumerate}
\item[{\rm(i)}] There exist a unique solution $(X_{N}(t),\Xi_{N}(t))\in C^{1}(\mathbb{R}_{+};\mathbb{R}^{2N}\times \mathbb{R}^{2N})$ of the Cauchy problem{\rm :} 
\begin{align}
\begin{cases}
\dot x_{i}(t)=\xi_{i}(t),\\
\dot \xi_{i}(t)=-B(t,x_{i}(t))\xi_{i}^{\perp}(t)-\frac{1}{N\varepsilon}\underset{j:j\neq i}{\sum}\nabla_{x} V(x_{i}(t)-x_{j}(t)),\\
x_{i}(0)=x_{i}^{0}, \quad \xi_{i}(0)=\xi_{i}^{0}. 
\end{cases}
\label{ode system appendix}
\end{align}
\item[{\rm(ii)}] If $(X_{N}(t),\Xi_{N}(t))$ is the unique global solution to \eqref{ode system appendix}, define $(\widetilde{X}_{N}(t),\widetilde{\Xi}_{N}(t)):= (X_{N}(t)\mod 1,\Xi_{N}(t))$. Moreover, let $(\widetilde{B},\widetilde{V})$ be the restriction of $(B,V)$ 
to $\mathbb{T}^{2}$, respectively.  
Then $(\widetilde{X}_{N}(t),\widetilde{\Xi}_{N}(t))\in C^{1}(\mathbb{R}_{+};\mathbb{T}^{2N}\times \mathbb{R}^{2N})$  is the unique solution to the Cauchy problem{\rm:} 
\begin{align}
\begin{cases}
\dot x_{i}(t)=\xi_{i}(t),\\
\dot \xi_{i}(t)=-\widetilde{B}(t,x_{i}(t))\xi_{i}^{\perp}(t)-\frac{1}{N\varepsilon}\underset{j:j\neq i}{\sum}\nabla_{x} \widetilde{V}(x_{i}(t)-x_{j}(t)), \\
\ x_{i}(0)=x_{i}^{0},\quad \xi_{i}(0)=\xi_{i}^{0}. 
\end{cases}
\label{ode system appendix torus}
\end{align}
\end{enumerate}
\end{thm}

\begin{proof} We divide the proof into two steps.

\smallskip
\textbf{1.} The reason we cannot apply the Cauchy-Lipschitz theorem  directly is because the product $B(t,x_{i})\xi_{i}^{\perp}$ is only locally Lipschitz. To circumvent this obstacle, fix a smooth truncation $\chi_{R}$ such that  
\begin{itemize}
 \item[{\rm(a)}] $\chi_{R}\in C^{\infty}_{0}(\mathbb{R}^{2})$; 
\item[{\rm(b)}] $\chi_{R}\equiv 1$ on $B_{R}(0)$ for some $R>0$ to be defined in the sequel. 
\end{itemize}
Consider the following truncated system: 
\begin{align}\tag{N$R$}
\begin{cases}
\dot x_{i}(t)=\xi_{i}(t),\\
\dot \xi_{i}(t)=-B(t,x_{i}(t))\chi_{R}(\xi_{i}(t))\xi_{i}^{\perp}(t)-\frac{1}{N\varepsilon}\underset{j:j\neq i}{\sum}\nabla_{x} V(x_{i}(t)-x_{j}(t)), \\
x_{i}(0)=x_{i}^{0},\quad \xi_{i}(0)=\xi_{i}^{0}. 
\end{cases}
\label{ode system appendix TRUNCATED}
\end{align}
By the Cauchy-Lipschitz theorem, there exists 
a unique global solution $(X_{N,R}(t),\Xi_{N,R}(t))\in C^{1}(\mathbb{R}_{+};\mathbb{R}^{2N}\times \mathbb{R}^{2N})$ to \eqref{ode system appendix TRUNCATED}.  
Next, we obtain a bound on the velocities. 
Multiplying the second equation in \eqref{ode system appendix TRUNCATED} by $\xi_{i}$, we obtain 
\begin{align*}
\frac{\dd}{\dd t}\frac{1}{2}\left\vert \xi_{i}\right\vert^{2}(t)=- \frac{1}{N\varepsilon} \sum_{j:j\neq i}\nabla_{x}V(x_{i}-x_{j})\cdot \xi_{i}(t)\leq \frac{1}{\varepsilon}\left\Vert \nabla_{x}V\right\Vert_{\infty}\left\vert \xi_{i}(t)\right\vert.     
\end{align*}
It follows that 
\begin{align*}
\frac{\dd}{\dd t}\left\vert \xi_{i}(t)\right\vert\leq \frac{2\left\Vert \nabla_{x}V\right\Vert_{\infty}}{\varepsilon}\,\Longrightarrow\, 
\underset{1\leq i\leq N}{\sup}\underset{t\in [0,T]}{\sup}\left\vert \xi_{i}(t)\right\vert
\leq\max_i|\xi_i(0)| +\frac{\left\Vert \nabla_{x}V\right\Vert_{\infty}T} {\varepsilon}:= R_{0}.
\end{align*}
Therefore, choosing $R>R_{0}$, it follows that $\chi_{R}(\xi_{i}(t))=1$ for all $t\in [0,T]$ and all $1\leq i\leq N$. 
Thus, $(X_{N,R}(t),\Xi_{N,R}(t))$ is a solution of the original system \eqref{ode system appendix} on $[0,T]$. 

\smallskip
\textbf{2.} Define  
$(\widetilde{X}_{N}(t),\widetilde{\Xi}_{N}(t))
:=  (X_{N}(t) \mod 1,\Xi_{N}(t))\in C^{1}(\mathbb{R}_{+};\mathbb{T}^{2N}\times \mathbb{R}^{2N})$.  
Given $t_{0}\in \mathbb{R}_{+}$, 
there are both an interval $\mathcal{Q}$ containing $t_{0}$ 
and a multi-integer $\mathbf{k}\in \mathbb{Z}^{2N}$ 
such that $\widetilde{X}_{N}(t)=X_{N}(t)+\mathbf{k}$ for all $t\in \mathcal{Q}$. 
To show that $(\widetilde{X}_{N}(t),\widetilde{\Xi}_{N}(t))$ is the asserted 
solution, we observe that, for all $t\in \mathcal{Q}$,
\begin{align*}
\frac{\dd}{\dd t}\widetilde{x}_{i}(t)=\frac{\dd}{\dd t}(x_{i}(t)+\mathbf{k}_{i})=\frac{\dd}{\dd t}x_{i}(t)=\xi_{i}(t)=\widetilde{\xi}_{i}(t).    
\end{align*}
Moreover, since $B$ and $V$ are $\mathbb{Z}^{2}$-periodic, we see that, for all $t\in \mathcal{Q}$,
\begin{align*}
\frac{\dd}{\dd t}\xi_{i}(t)&=-B(t,x_{i}(t))\xi_{i}^{\perp}(t)-\frac{1}{N\varepsilon}\sum_{j:j\neq i}\nabla_{x}V(x_{i}(t)-x_{j}(t))\\
&=-B(t,x_{i}(t)+\mathbf{k}_{i})\xi_{i}^{\perp}(t)-\frac{1}{N\varepsilon}\sum_{j:j\neq i}\nabla_{x}V(x_{i}(t)-x_{j}(t)+\mathbf{k}_{i}-\mathbf{k}_{j})\\
&=-B(t,\widetilde{x}_{i}(t))\xi_{i}^{\perp}(t)-\frac{1}{N\varepsilon}\sum_{j:j\neq i} \nabla_{x}V(\widetilde{x}_{i}(t)-\widetilde{x}_{j}(t)). 
\end{align*}
\end{proof}

\bigskip
\noindent{\bf Acknowledgments.}
IBP, GQC, and DF acknowledge support by EPSRC grant EP/V051121/1. GQC was also partially supported by EPSRC grant
EP/V008854. DF was also partially supported by the Fundamental Research Funds for the Central Universities No. 2233100021 and No. 2233300008. For the purpose of open access, the authors have applied a CC BY public copyright license to any Author Accepted Manuscript (AAM) version
arising from this submission.

\medskip
\noindent{\bf Conflict of Interest:} The authors declare that they have no conflict of interest.
The authors also declare that this manuscript has not been previously published,
and will not be submitted elsewhere before your decision.

\bigskip
\noindent{\bf Data availability:} Data sharing is not applicable to this article as no datasets were generated or analyzed during the current study.


\bigskip
\medskip
\begin{thebibliography}{}

\bibitem{bahouri2011fourier}
H. Bahouri, J.-Y. Chemin, and R. Danchin,
\newblock {\em Fourier Analysis and Nonlinear Partial Differential Equations},
\newblock Grundlehren der Mathematischen Wissenschaften, {\bf 343}, Springer, Heidelberg, 2011.

\bibitem{porat2023derivation}
I. Ben-Porat,
\newblock Derivation of Euler's equations of perfect fluids from von Neumann's equation with magnetic field,
\newblock {\it J. Stat. Phys.}, {\bf 190(7)} (2023), 121.

\bibitem{porat2025quantum}
I. Ben-Porat, G.-Q. Chen, and D. Yuan,
\newblock Quantum quasi-neutral limits and isothermal Euler equations,
\newblock {\it Nonlinearity}, {\bf 39(6)} (2026), 065005.

\bibitem{porat2025propagation}
I. Ben-Porat, A. Gagnebin, M. Iacobelli, and J. Junn{\'e},
\newblock Propagation of velocity moments for the magnetized Vlasov--Poisson system with space-time dependent magnetic fields,
\newblock Preprint arXiv:2510.22753, 2025.

\bibitem{ben2025derivation}
I. Ben-Porat, M. Iacobelli, and A. Rege,
\newblock Derivation of Yudovich solutions of incompressible Euler from the Vlasov--Poisson system,
\newblock {\it SIAM J. Math. Anal.}, {\bf 57(1)} (2025), 886--906.

\bibitem{bostan2023asymptotic}
M. Bostan and A.-T. Vu,
\newblock Asymptotic behavior of the two-dimensional Vlasov--Poisson--Fokker--Planck equation with a strong external magnetic field,
\newblock {\it Kinet. Relat. Models}, {\bf 18(1)} (2025), 101--147.

\bibitem{brenier2000convergence}
Y. Brenier,
\newblock Convergence of the Vlasov--Poisson system to the incompressible Euler equations,
\newblock {\it Commun. Partial Differ. Equs.}, {\bf 25(3--4)} (2000), 737--754.

\bibitem{brenier1994limite}
Y. Brenier and E. Grenier,
\newblock Limite singuli{\`e}re du syst{\`e}me de Vlasov--Poisson dans le r{\'e}gime de quasi neutralit{\'e}: le cas ind{\'e}pendant du temps,
\newblock {\it C. R. Acad. Sci. Paris S{\'e}r. I Math.}, {\bf 318(2)} (1994), 121--124.

\bibitem{brenier2003incompressible}
Y. Brenier, N. Mauser, and M. Puel,
\newblock Incompressible Euler and e-MHD as scaling limits of the Vlasov--Maxwell system,
\newblock {\it Commun. Math. Sci.}, {\bf 1(3)} (2003), 437--447.

\bibitem{bresch2020modulated}
D. Bresch, P.-E. Jabin, and Z. Wang,
\newblock Modulated free energy and mean field limit,
\newblock {\it S{\'e}minaire Laurent Schwartz---EDP et applications}, {\bf 2019--2020} (2020), Exp. No. 2, 1--22.

\bibitem{burby2025hamiltonian}
J. W. Burby, D. A. Kaltsas, P. J. Morrison, E. Tassi, and
G. N. Throumoulopoulos,
\newblock Hamiltonian formulation of the quasineutral
Vlasov--Poisson system,
\newblock Preprint arXiv:2506.21415, 2025.

\bibitem{chen2026vpfp}
L. Chen, J. Jung, P. Pickl, and Z. Wang,
\newblock On the mean-field limit of
Vlasov--Poisson--Fokker--Planck equations,
\newblock {\it J. Math. Phys.}, {\bf 67} (2026), 061503.

\bibitem{chen2015introduction}
F. F. Chen,
\newblock {\em Introduction to Plasma Physics and Controlled Fusion}, 3rd ed.,
\newblock Springer, Cham, 2016.

\bibitem{chen2025meanfield}
X. Chen, S. Shen, and Z. Zhang,
\newblock On the mean-field and semiclassical limit from quantum $N$-body dynamics,
\newblock {\it J. Funct. Anal.}, {\bf 289(10)} (2025), 111100.

\bibitem{gagnebin2025relativistic}
A. Gagnebin, M. Iacobelli, A. Rege, and S. Rossi,
\newblock From relativistic Vlasov--Maxwell to electron-MHD in the quasineutral regime,
\newblock Preprint arXiv:2505.11428, 2025.

\bibitem{gallagher2005pressureless}
I. Gallagher and L. Saint-Raymond,
\newblock On pressureless gases driven by a strong inhomogeneous magnetic field,
\newblock {\it SIAM J. Math. Anal.}, {\bf 36(4)} (2005), 1159--1176.

\bibitem{golse2003vlasov}
F. Golse and L. Saint-Raymond,
\newblock The Vlasov--Poisson system with strong magnetic field in quasineutral regime,
\newblock {\it Math. Models Methods Appl. Sci.}, {\bf 13(5)} (2003), 661--714.

\bibitem{grenier1996oscillations}
E. Grenier,
\newblock Oscillations in quasineutral plasmas,
\newblock {\it Commun. Partial Differ. Equs.},
{\bf 21(3--4)} (1996), 363--394.

\bibitem{grenier1998semiclassical}
E. Grenier,
\newblock Semiclassical limit of the nonlinear Schr{\"o}dinger equation in small time,
\newblock {\it Proc. Amer. Math. Soc.}, {\bf 126(2)} (1998), 523--530.

\bibitem{griffin2021recent}
M. Griffin-Pickering and M. Iacobelli,
\newblock Recent developments on quasineutral limits for Vlasov-type equations, in:
\newblock {\em Recent Advances in Kinetic Equations and Applications}, Springer INdAM Series, {\bf 48}, Springer, Cham, 2022, 211--231.

\bibitem{han2011quasineutral}
D. Han-Kwan,
\newblock Quasineutral limit of the Vlasov--Poisson system with massless electron,
\newblock {\it Commun. Partial Differ. Equs.}, {\bf 36(8)} (2011), 1385--1425.

\bibitem{han2014quasineutral}
D. Han-Kwan and M. Iacobelli,
\newblock The quasineutral limit of the Vlasov--Poisson equation in Wasserstein metric,
\newblock {\it Commun. Math. Sci.}, {\bf 15(2)} (2017), 481--509.

\bibitem{han2021newton}
D. Han-Kwan and M. Iacobelli,
\newblock From Newton's second law to Euler's equations of perfect fluids,
\newblock {\it Proc. Am. Math. Soc.}, {\bf 149(7)} (2021), 3045--3061.

\bibitem{iacobelli2026quasineutral}
M. Iacobelli,
\newblock Quasineutral plasmas and the geometry of kinetic stability,
\newblock Preprint arXiv:2605.28435, 2026.

\bibitem{jiang2026yukawa}
N. Jiang, Z. Qiao, J. Wu, and J. Zhang,
\newblock Commutator estimates uniform in the screening parameter
and mean-field limits for Yukawa interactions,
\newblock Preprint arXiv:2608.22038, 2026.

\bibitem{ju2026zero}
Q. Ju and C. Liu,
\newblock The zero-electron-mass limit in the Euler--Poisson system
with non-constant ion density for all time,
\newblock {\it Discrete Contin. Dyn. Syst. B},
{\bf 41} (2026), 1--23.

\bibitem{jungel2003convergence}
A. J{\"u}ngel and S. Wang,
\newblock Convergence of nonlinear Schr{\"o}dinger--Poisson systems to the compressible Euler equations,
\newblock {\it Commun. Partial Differ. Equs.}, {\bf 28(5--6)} (2003), 1005--1022.

\bibitem{lions1991propagation}
P.-L. Lions and B. Perthame,
\newblock Propagation of moments and regularity for the 3-dimensional Vlasov--Poisson system,
\newblock {\it Invent. Math.}, {\bf 105(1)} (1991), 415--430.
\bibitem{Peng2024} Y.J. Peng and  
C. Liu, 
\newblock
Global non-relativistic quasi-neutral limit for a two-fluid Euler-Maxwell system
\newblock {\it J. Differ. Equ.}, {\bf 385} (2024), 362--394.

\bibitem{masmoudi2001vlasov}
N. Masmoudi,
\newblock From Vlasov--Poisson system to the incompressible
Euler system,
\newblock {\it Commun. Partial Differ. Equs.},
{\bf 26(9--10)} (2001), 1913--1928.

\bibitem{Peng2017} Y.J. Peng and V. Wasiolek, 
\newblock
Global quasi-neutral limit of Euler–Maxwell systems with velocity dissipation
\newblock {\it J. Math. Anal. Appl.}, {\bf 451(1)} (2017), 146--174.

\bibitem{pu2016quasineutral}
X. Pu,
\newblock Quasineutral limit of the Euler--Poisson system
under strong magnetic fields,
\newblock {\it Discrete Contin. Dyn. Syst. Ser. S},
{\bf 9(6)} (2016), 2095--2111.

\bibitem{puel2002convergence}
M. Puel,
\newblock Convergence of the Schr{\"o}dinger--Poisson system to the incompressible Euler equations,
\newblock {\it Commun. Partial Differ. Equs.}, {\bf 27(11--12)} (2002), 2311--2331.

\bibitem{puel2004quasineutral}
M. Puel and L. Saint-Raymond,
\newblock Quasineutral limit for the relativistic Vlasov--Maxwell system,
\newblock {\it Asymptot. Anal.}, {\bf 40(3--4)} (2004), 303--352.

\bibitem{rege2025stability}
A. Rege,
\newblock Stability estimates for magnetized Vlasov equations,
\newblock {\it J. Differ. Equ.}, {\bf 425} (2025), 763--788.

\bibitem{rosenzweig2021quantum}
M. Rosenzweig,
\newblock From quantum many-body systems to ideal fluids,
\newblock Preprint arXiv:2110.04195, 2021.

\bibitem{rosenzweig2023rigorous}
M. Rosenzweig,
\newblock On the rigorous derivation of the incompressible Euler equation from Newton's second law,
\newblock {\it Lett. Math. Phys.}, {\bf 113} (2023), 13.

\bibitem{rosenzweig2025commutators}
M. Rosenzweig,
\newblock Commutators, mean-field, and supercritical mean-field limits for Coulomb/Riesz gases,
\newblock {\it Journ{\'e}es {\'e}quations aux d{\'e}riv{\'e}es partielles} (2025), Talk No. 7, 1--32.

\bibitem{rosenzweig2025lake}
M. Rosenzweig and S. Serfaty,
\newblock The lake equation as a supercritical mean-field limit,
\newblock {\it J. {\'E}c. polytech. Math.}, {\bf 12} (2025), 1019--1068.

\bibitem{duerinckx2020mean}
S. Serfaty,
\newblock Mean field limit for Coulomb-type flows (Appendix with M. Duerinckx),
\newblock {\it Duke Math. J.}, {\bf 169(15)} (2020), 2887--2935.

\end{thebibliography}
\end{document}